\documentclass[12pt,reqno]{amsart}
\usepackage{amsfonts,amsrefs,latexsym,amsmath, amssymb, mathrsfs, verbatim}
\usepackage{url,color}
\usepackage{pifont}
\usepackage{upgreek}
\usepackage{fancyhdr}
\usepackage{hyperref}
\usepackage{calligra}
\usepackage{marvosym}
\usepackage[percent]{overpic}
\usepackage{pict2e}

\usepackage[hypcap=false]{caption}
\usepackage{xcolor}
\usepackage{wasysym}
\usepackage{accents}

\topmargin - .23 in

\newtheorem{theorem}{Theorem}[section]
\newtheorem{proposition}{Proposition}[section]
\newtheorem{lemma}[proposition]{Lemma}

\theoremstyle{definition}
\newtheorem{definition}{Definition}[section]
\newtheorem{remark}{Remark}[section]

\theoremstyle{plain}

\DeclareMathAlphabet{\mathcalligra}{T1}{calligra}{m}{n}
\DeclareFontShape{T1}{calligra}{m}{n}{<->s*[2.2]callig15}{}

\newcommand{\leftexp}[2]{{\vphantom{#2}}^{#1}{#2}}

\numberwithin{equation}{subsection}

\newcommand{\weight}{\mathscr{W}}

\newcommand{\Tboot}{T_{(Boot)}}

\newcommand{\Ifact}{\mathcal{I}}

\begin{document}
\title{Stable ODE-type blowup for some quasilinear wave equations 
	with derivative-quadratic nonlinearities
}
\author[JS]{Jared Speck$^{* \dagger}$}

\thanks{$^{\dagger}$JS gratefully acknowledges support from NSF grant \# DMS-1162211,
from NSF CAREER grant \# DMS-1454419,
from a Sloan Research Fellowship provided by the Alfred P. Sloan foundation,
and from a Solomon Buchsbaum grant administered by the Massachusetts Institute of Technology.
}

\thanks{$^{*}$Massachusetts Institute of Technology, Cambridge, MA, USA.
\texttt{jspeck@math.mit.edu}}

\begin{abstract}
We prove a constructive stable ODE-type blowup result for
open sets of solutions to a family of quasilinear wave equations
in three spatial dimensions
featuring a Riccati-type derivative-quadratic semilinear term.
The singularity is more severe than a shock
in that the solution itself blows up like the log
of the distance to the blowup-time. We assume that the quasilinear terms 
satisfy certain structural assumptions, which in particular ensure 
that the ``elliptic part'' of the wave operator vanishes precisely at the singular points.
The initial data are compactly supported and 
can be small or large in $L^{\infty}$, but
the spatial derivatives must initially satisfy a nonlinear
smallness condition compared to the time derivative.
The first main idea of the proof is to construct a quasilinear integrating factor, 
which allows us to reformulate the wave equation as a 
system whose solutions remain regular, 
all the way up to the singularity. This is equivalent to constructing 
quasilinear vectorfields adapted to the nonlinear flow. 
The second main idea is to exploit some 
crucial monotonic terms in various estimates, especially the energy estimates, 
that feature the integrating factor.
The availability of the monotonicity is tied to our assumptions on the data 
and on the structure of the quasilinear terms.
The third main idea is to propagate the relative smallness of 
the spatial derivatives all the way up to the singularity 
so that the solution behaves, in many ways, like
an ODE solution. As a corollary of our main results, we show that there are quasilinear wave equations
that exhibit two distinct kinds of blowup: the formation of shocks for one non-trivial set of data,
and ODE-type blowup for another non-trivial set.
\bigskip


\bigskip

\end{abstract}

\maketitle

\centerline{\today}

\tableofcontents
\setcounter{tocdepth}{1}

\newpage

\section{Introduction}
\label{S:INTRO}
A fundamental issue surrounding the study of nonlinear hyperbolic PDEs
is that singularities can form in finite time, starting from smooth initial data.
For a given singularity-forming solution, perhaps the most basic question one can ask is 
whether or not it is stable under perturbations
of its initial data. Our main result provides an affirmative answer to this question
for some solutions to a class of quasilinear wave equations. Specifically, in three spatial dimensions, 
we provide a sharp, constructive proof of stable ODE-type blowup for solutions corresponding to an \underline{open} 
set (in a suitable Sobolev topology in which there are no radial weights in the norms)
of initial data for a class of quasilinear wave equations 
that are well-modeled by 
\begin{align} \label{E:FIRSTMODELWAVE}
-\partial_t^2 \Phi + \frac{1}{1 + \partial_t \Phi} \Delta \Phi = - (\partial_t \Phi)^2.
\end{align}
For the solutions from our main results, $\Phi$ itself blows up.
This is a much more drastic singularity compared to the case of the formation of shocks, 
which for equations of type \eqref{E:FIRSTMODELWAVE} would correspond to the blowup
of $\partial^2 \Phi$ with $\Phi$ and $\partial \Phi$ remaining bounded;
see Subsect.\ \ref{SS:PRIORBREAKDOWNRESULTS} for further discussion.
As a corollary of our main results, we show 
(see Subsect.\ \ref{SS:DIFFERENTKINDSOFSINGULARITYFORMATION})
that there are quasilinear wave equations
that exhibit \emph{two distinct kinds of blowup}: 
ODE-type blowup for one non-trivial (but not necessarily open) set of initial data,
and the formation of a shock for a different non-trivial set of data.
We view this as a parable highlighting two key phenomena 
that would have to be accounted for in any sufficiently broad theory of
singularity formation in solutions to quasilinear wave equations; 
i.e., in principle, a quasilinear\footnote{For certain \emph{semilinear} wave equations,
it is well-known that
different kinds of blowup can occur: blowup of the $L^{\infty}$ norm of the solution itself (known as Type I blowup)
and a different kind of blowup in which the solution remains bounded in an appropriate Sobolev norm
(known as Type II blowup); see Subsect.\ \ref{SS:PRIORBREAKDOWNRESULTS}.} 
wave equation can admit radically 
different types of singularity-forming solutions.\footnote{These phenomena can also be exhibited in the
much simpler setting of quasilinear transport equations. For example, the inhomogeneous Burgers equation 
$\partial_t \Psi + \Psi \partial_x \Psi = \Psi^2$ admits the $T$-parameterized family of 
spatially homogeneous singularity-forming solutions $\Psi_{(ODE);T} := (T-t)^{-1}$ as well as solutions that form shocks,
i.e., $|\partial_x \Psi|$ blows up but $|\Psi|$ remains bounded.}
It is only for concreteness that we restrict our attention to three spatial dimensions;
our approach can be applied to any number of spatial dimensions, with only slight modifications needed.
See Subsects.\ \ref{SS:STATEMENTOFWAVEEQUATION} and \ref{SS:WEIGHTASSUMPTIONS} 
for a precise description of the class of equations that we treat,
Subsubsect.\ \ref{SSS:SUMMARYOFRESULTS} for a summary of our main results,
and Sect.\ \ref{S:MAINTHM} for the detailed statement of our main theorem.

Obtaining a sharp description of the blowup is
particularly important if one aims to weakly continue the solution
past the singularity, as is sometimes possible; one expects to need sharp information
in order to even properly set up the problem of weakly continuing.\footnote{The most significant weak continuation result in more than one spatial dimension
is Christodoulou's recent solution \cite{dC2017} of the restricted 
shock development
problem in compressible fluid mechanics, which, roughly speaking,
is a local well-posedness result for weak solutions and their corresponding 
hypersurfaces of discontinuity, starting from the first shock, whose formation
from smooth initial conditions was described in detail in his breakthrough work 
for relativistic fluids \cite{dC2007} and in his follow-up work with Miao \cite{dCsM2014}
on non-relativistic compressible fluids.
The term ``restricted'' means that the jump in entropy across the shock hypersurface was ignored.}
Our work shows that a standard type of weak continuation 
is not possible for the solutions that we study,
since $\Phi$ itself blows up.

The precise algebraic details of the weight
$
\frac{1}{1 + \partial_t \Phi}
$
in front of the Laplacian term in equation \eqref{E:FIRSTMODELWAVE} 
are not important for our proof.
What is important is that the weight decays at an appropriate rate
as $\partial_t \Phi \to \infty$, that is, as the singularity forms;
see Subsect.\ \ref{SS:WEIGHTASSUMPTIONS} for our assumptions on the weight.
As we will explain, this decay yields
a friction-type spacetime integral that is important for closing the energy estimates,
and it also helps us to prove that spatial derivative terms remain small relative to the time derivative terms, 
up to the singularity. The problem of providing a sharp description of blowup for solutions to
derivative-quadratic semilinear wave equations, such as
$-\partial_t^2 \Phi + \Delta \Phi = - (\partial_t \Phi)^2$,
remains open, even though John \cite{fJ1981} showed, via proof by contradiction,
that all non-trivial, smooth, compactly supported solutions to the equation 
$-\partial_t^2 \Phi + \Delta \Phi = - (\partial_t \Phi)^2$ 
in three spatial dimensions blow up in finite time.

Our results show, in part due to the weight in front of the Laplacian, 
that the spatial-derivative-involving nonlinearities in equation \eqref{E:FIRSTMODELWAVE} (and the other equations that we study)
exhibit a subcritical\footnote{In contrast,
for the semilinear equation $-\partial_t^2 \Phi + \Delta \Phi = - (\partial_t \Phi)^2$, 
our approach suggests, but does not prove, that the blowup-rate for the 
Laplacian term $\Delta \Phi$ might be critical with respect to the expected blowup-rate
for the other two terms in the equation, i.e., that all terms might blow up at the same rate.} 
blowup-rate relative to the pure time derivative terms.
However, as we explain below, this 
\emph{subcritical behavior does not seem detectable relative to the standard partial derivatives $\partial_{\alpha}$};
to detect the behavior, we will use a combination of ``quasilinear vectorfield derivatives'' $\mathcal{I} \partial_{\alpha}$ 
and standard derivatives $\partial_{\alpha}$,
where $\Ifact$ is a ``quasilinear integrating factor'' that we describe below. In fact, we will show that
$\Ifact \partial_{\alpha} \Phi$ remains bounded up to the singularity and that the singularity formation 
coincides with the vanishing of $\Ifact$.
In total, our approach allows us to treat the equations under study 
as quasilinear perturbations of the Riccati ODE
$\frac{d^2}{dt^2} \Phi = (\frac{d}{dt} \Phi)^2$. 
By ``perturbation of the Riccati ODE'', 
we mean in particular that the singularity formation 
is similar to the one that occurs in the following $T$-parameterized family of ODE solutions to \eqref{E:FIRSTMODELWAVE}:
\begin{align} \label{E:ODEBLOWWINGUPSOLUTION}
	\Phi_{(ODE);T}(t) := \ln \left((T-t)^{-1}\right),
\end{align}
where $T \in \mathbb{R}$ is the blowup-time.
Our methods are tailored to the quadratic term on RHS~\eqref{E:FIRSTMODELWAVE}
in that they do not apply, at least in their current form, to 
semilinear terms of type $(\partial_t \Phi)^p$ for $p \neq 2$. 
However, derivative-quadratic terms are of particular interest in view of the fact
that they commonly arise in nonlinear field theories
(though the derivative-quadratic terms in such theories 
are often not Riccati-type, like the one featured on RHS~\eqref{E:FIRSTMODELWAVE}).

There are many results on stable breakdown for wave equations, some of which 
we review in Subsect.\ \ref{SS:PRIORBREAKDOWNRESULTS}.
The ``theory'' of stable breakdown is quite fragmented 
in that the techniques that have been employed
vary wildly between different classes of equations. 
In particular, the techniques that have been developed
do not seem to apply to the equations under study here.
This will become clear after we describe the main ideas
of our proof (see Subsect.\ \ref{SS:IDEASBEHINDPROOF}) 
and review prior works on stable
breakdown. Although ODE-type blowup is arguably the simplest blowup scenario, 
there do not seem to be any prior constructive stable blowup results 
of this type for scalar wave equations with derivative-quadratic nonlinearities, 
in any number of spatial dimensions. 
We mention, however, that in \cite{iRjS2014b}, 
we proved, using rather different techniques specialized to Einstein's equations,
a singularity formation result for Einstein's equations
that can be interpreted as a stable ODE-type blowup result
for the first derivatives\footnote{Relative to a geometrically defined coordinate system, 
the second fundamental form of the metric blows up, though the metric
components do not; this can be viewed as the blowup of the first derivatives of the metric.} 
of a solution
to a quasilinear system with derivative-quadratic nonlinearities.

\subsection{Paper outline}
\label{SS:PAPEROUTLINE}
\begin{itemize}
	\item In the remainder of Sect.\ \ref{S:INTRO},
		we summarize our results, outline their proofs, 
		place our work in context by discussing prior works,
		discuss (see Subsect.\ \ref{SS:DIFFERENTKINDSOFSINGULARITYFORMATION}) 
		a corollary (which we described just below equation \eqref{E:FIRSTMODELWAVE}) of our main results,
		and summarize our notation.
	\item In Sect.\ \ref{S:MATHEMATICALSETUP},
		we define the quantities play a role in our analysis
		and derive various evolution equations.
	\item In Sect.\ \ref{S:DATAANDBOOTSTRAP}, we state our assumptions on the
		initial data and state bootstrap assumptions that are useful for studying the solution.
	\item In Sect.\ \ref{S:ENERGY}, we derive energy identities.
	\item In Sect.\ \ref{S:ESTIMATES}, which is the main section of the article,
		we derive a priori estimates that in particular yield strict improvements
		of the bootstrap assumptions.
	\item In Sect.\ \ref{S:WELLPOSEDNESS}, we state a standard local well-posedness result
		and continuation criteria for the equations under study.
	\item In Sect.\ \ref{S:MAINTHM}, we prove the main theorem.
\end{itemize}

\subsection{The class of wave equations under study}
\label{SS:STATEMENTOFWAVEEQUATION}
Our main theorem concerns solutions to the Cauchy problem for
quasilinear wave equations in three spatial dimensions 
of the following form:
\begin{subequations}
\begin{align} \label{E:WAVE}
	- 
	\partial_t^2 \Phi
	+
	\weight(\partial_t \Phi)
	\Delta \Phi
	& = - (\partial_t \Phi)^2,
		\\
	(\partial_t \Phi|_{\Sigma_0},\partial_1 \Phi|_{\Sigma_0}, \partial_2 \Phi|_{\Sigma_0}, \partial_3 \Phi|_{\Sigma_0}) 
	& = (\mathring{\Psi}_0,\mathring{\Psi}_1,\mathring{\Psi}_2,\mathring{\Psi}_3),
	\label{E:DATAWAVE}
\end{align}
\end{subequations}
where throughout, $\Sigma_t$ denotes the hypersurface of constant time $t$.
Our use of the notation ``$\mathring{\Psi}_{\alpha}$'' for the data functions 
is tied to our use of the renormalized solutions variables $\Psi_{\alpha}$ that
we will use in studying solutions; see Def.~\ref{D:RENORMALIZEDSOLUTION}.

\begin{remark}[\textbf{Viewing} \eqref{E:WAVE} \textbf{as an equation in} $\partial \Phi$]
\label{R:NOPHIINWAVEEQUATION}
Since $\Phi$ itself is not featured in equation \eqref{E:WAVE}
(only its derivatives appear), 
we only need to prescribe the derivatives of $\Phi$ along $\Sigma_0$
in order to solve for $\lbrace \partial_{\alpha} \Phi \rbrace_{\alpha = 0,1,2,3}$.
This is relevant in that we do not bother to derive estimates for
$\Phi$ itself (see, however, Remark~\ref{R:BLOWUPOFPHI}).
\end{remark}

In \eqref{E:WAVE}, 
$\Delta := \sum_{a=1}^3 \partial_a^2$ is the standard Euclidean Laplacian on $\mathbb{R}^3$
and $\weight = \weight(\partial_t \Phi)$ is a nonlinear ``weight function''
verifying certain technical conditions stated below,
specifically \eqref{E:WEIGHTISPOSITIVE}-\eqref{E:WEIGHTVSWEIGHTDERIVATIVECOMPARISON}.
Prototypical examples of
weights verifying \eqref{E:WEIGHTISPOSITIVE}-\eqref{E:WEIGHTVSWEIGHTDERIVATIVECOMPARISON}
are the functions
\begin{align} \label{E:POWERLAWWEIGHT}
\weight(y)
& = \frac{1}{1 + y^M}
&
\mbox{or }
\weight(y)
& = \frac{1}{(1 + y)^M},
\end{align}
where $M \geq 1$ is an integer,
and the function
\begin{align} \label{E:EXPWEIGHT}
\weight(y)
& = \exp(-y).
\end{align}

\subsection{Rough summary of the results and discussion of the proof}
\label{SS:IDEASBEHINDPROOF}

\subsubsection{Rough summary of the results}
\label{SSS:SUMMARYOFRESULTS}
We now briefly summarize the main results;
see Theorem~\ref{T:STABILITYOFODEBLOWUP} for precise statements.

\begin{theorem}[\textbf{Stable ODE-type blowup} (rough version)]
\label{T:ROUGHMAINTHM}
Under suitable assumptions (stated in Subsect.\ \ref{SS:WEIGHTASSUMPTIONS}) 
on the weight $\weight(\partial_t \Phi)$, 
there exists an open set of compactly supported initial data for equation \eqref{E:WAVE},
with $\mathring{\Psi}_{\alpha} \in H^5(\mathbb{R}^3)$,
such that the solution blows up in finite time in a manner similar
to the solutions $\Phi_{(ODE);T}$ from \eqref{E:ODEBLOWWINGUPSOLUTION}. 
In particular, there exists a time $0 < T_{(Lifespan)} < \infty$
such that\footnote{More precisely, one can conclude that $\| \Phi\|_{L^{\infty}(\Sigma_t)}$ blows up at
$t = T_{(Lifespan)}$ if initial data for $\Phi$ itself is prescribed; Remark~\ref{R:NOPHIINWAVEEQUATION}.}
$
\| \partial_t \Phi \|_{L^{\infty}(\Sigma_t)}
$
and 
$\| \Phi\|_{L^{\infty}(\Sigma_t)}$ blow up as 
$t \uparrow T_{(Lifespan)}$.
The data functions $\lbrace \mathring{\Psi}_{\alpha} \rbrace_{\alpha=0,1,2,3}$ 
are allowed to be large or small as measured 
by a Sobolev norm without radial weights, but 
$\lbrace \mathring{\Psi}_a \rbrace_{a=1,2,3}$, 
$\nabla \mathring{\Psi}_0$,
and their spatial derivatives up to top order must satisfy 
a nonlinear smallness condition
relative to $\max_{\Sigma_0} [\mathring{\Psi_0}]_+$.

Moreover, let the integrating factor $\Ifact$ be the solution to
\begin{align} \label{E:INTROINTEGRATINFACTORODEANDIC}
	\partial_t \Ifact
	& = 
	- \Ifact \partial_t \Phi,
	&& \Ifact|_{\Sigma_0} = 1.
\end{align}
Then $\Ifact$, the variables 
\begin{align} \label{E:RENORMALIZEDRICCATIBLOWUPVARIABLES}
	\Psi_{\alpha} := \Ifact \partial_{\alpha} \Phi,
\end{align}
and their partial derivatives with respect to the Cartesian coordinates
\textbf{remain regular all the way up to time $T_{(Lifespan)}$},
except possibly at the top derivative level due to 
the vanishing of the weight $\weight(\partial_t \Phi)$
(which appears in the energy estimates)
as $\partial_t \Phi \uparrow \infty$.
\end{theorem}	

\begin{remark}[\textbf{Maximal development}]
	\label{R:MAXIMALDEVELOPMENT}
	We anticipate that the sharp results of Theorem~\ref{T:ROUGHMAINTHM} should be useful
	for obtaining detailed information about the solution not just up to the first singular time,
	but also up to the boundary
	of the maximal development.\footnote{The maximal development of the data is, roughly,
the largest possible classical solution that is uniquely determined by the data. 
Readers may consult \cites{jSb2016,wW2013} for further discussion.
\label{FN:MAXIMALDEVELOPMENT}}
In the context of shock formation for fluids, Christodoulou \cite{dC2007}*{Chapter 15} 
used similar sharp estimates to follow the solution up to boundary.
Broadly similar results were obtained by Merle--Zaag in \cite{fMzH2012a},
in which, in the case of one spatial dimension,
they gave a sharp description of the 
boundary of the maximal development for \emph{any} singularity-forming solution
to the semilinear focusing wave equation $-\partial_t^2 \Psi + \partial_x^2 \Psi = - |\Psi|^{p-1} \Psi$
with $p > 1$ and showed in particular that characteristic points on the boundary are isolated.
\end{remark}

\begin{remark}[\textbf{The blowup of} $\Phi$]	
	\label{R:BLOWUPOFPHI}
	We now make some remarks on the blowup of $\Phi$ itself since,
	as we highlighted in Remark~\ref{R:NOPHIINWAVEEQUATION}, one does
	not need to prescribe the initial data for $\Phi$ itself
	(and since in the rest of the paper we do not assume that initial data for $\Phi$ itself are prescribed).
	If one does prescribe its initial data, 
	then the results of Theorem~\ref{T:STABILITYOFODEBLOWUP} can easily be used to show
	that $\Phi$ itself blows up at time $T_{(Lifespan)}$ (such a result is not stated in Theorem~\ref{T:STABILITYOFODEBLOWUP}).
	This is philosophically important in that it dashes 
	any hope of weakly continuing the solution past the singularity, at least in a standard sense.
	To deduce the blowup for $\Phi$, one can first use
	\eqref{E:INTROINTEGRATINFACTORODEANDIC} and the fundamental theorem of calculus
	to deduce that
	$\ln \Ifact(t,\underline{x}) + \Phi(t,\underline{x}) = \Phi(0,\underline{x})$,
	where $\Phi(0,\cdot)$ is a regular function that by assumption verifies 
	$\| \Phi(0,\cdot) \|_{L^{\infty}} < \infty$.
	Since the singularity formation for $\partial_t \Phi$ yielded by Theorem~\ref{T:STABILITYOFODEBLOWUP}
	coincides with the vanishing of $\Ifact$ for the first time at $t = T_{(Lifespan)}$,
	it follows that $\lim_{t \uparrow T_{(Lifespan)}} \sup_{s \in [0,t)} \| \Phi\|_{L^{\infty}(\Sigma_s)} = \infty$,
	as is claimed in Theorem~\ref{T:ROUGHMAINTHM}.
\end{remark}

\subsubsection{The main ideas behind the proof}
\label{SSS:DISCUSSIONOFTHEPROOF}
The initial data that we consider are such that the spatial derivatives
of $\Phi$ up to top order are initially small relative to $\partial_t \Phi$.
We also assume that the spatial derivatives of $\partial_t \Phi$ up to top order are initially small.
The smallness assumptions that we need to close the proof are nonlinear in nature,\footnote{In particular, our smallness assumptions
on the data \eqref{E:DATAWAVE}
are \emph{not} generally invariant under rescalings of the form
$(\mathring{\Psi}_0,\mathring{\Psi}_1,\mathring{\Psi}_2,\mathring{\Psi}_3)
\rightarrow
\uplambda^{-1}
(\mathring{\Psi}_0,\mathring{\Psi}_1,\mathring{\Psi}_2,\mathring{\Psi}_3)
$ if $\uplambda$ is too large.} 
for reasons described just below \eqref{E:CARICATURENERGYAPRIORI}; 
see Subsect.\ \ref{SS:SMALLNESSASSUMPTIONS} 
for our precise smallness assumptions
and Subsect.\ \ref{SS:EXISTENCEOFDATA} for a simple proof that such data exist.
In our analysis, we propagate certain aspects of this smallness all the way up to the singularity.
As we mentioned earlier, this allows us to effectively treat equation \eqref{E:WAVE} as a perturbation of 
the Riccati ODE $\frac{d^2}{dt^2} \Phi = (\frac{d}{dt} \Phi)^2$.
We again stress that the vanishing of the coefficient $\weight(\partial_t \Phi)$ of the Laplacian term in \eqref{E:WAVE}
as the singularity forms is important 
for our estimates, in particular for showing that spatial derivative terms 
remain relatively small.

A key point is that it does not seem possible to follow the solution all the way to the singularity 
by studying the wave equation in the form
\eqref{E:WAVE}. To caricature the situation, let us pretend that the singularity occurs at $t = 1$.
Our proof shows, roughly, that $\partial^k \Phi$ blows up like $(1-t)^{-k}$, 
where $\partial^k$ denotes $k^{th}$-order Cartesian coordinate partial derivatives.
This means, in particular, that commuting equation \eqref{E:WAVE} with more and more spatial derivatives makes the
singularity strength of the nonlinear terms worse and worse, 
which is a serious obstacle to closing nonlinear estimates.
For this reason, as our statement of Theorem~\ref{T:ROUGHMAINTHM} already makes clear, 
our proof is fundamentally based on the solution to \eqref{E:INTROINTEGRATINFACTORODEANDIC},
that is, the integrating factor
$\Ifact$ solving the transport equation
$\partial_t \Ifact = - \Ifact \partial_t \Phi$
with initial conditions $\Ifact|_{\Sigma_0} = 1$.
Note that the finite-time blowup $\partial_t \Phi \uparrow \infty$ 
would follow from the finite-time vanishing of $\Ifact$.
Indeed, to show that a singularity forms,
\emph{we will show that $\Ifact$ vanishes in finite time}.
Using $\Ifact$, we are able to transform the wave equation into
a ``regularized'' system, equivalent to equation \eqref{E:WAVE} up to the singularity, 
that we analyze to show that the weighted derivatives
$\lbrace \Psi_{\alpha} := \Ifact \partial_{\alpha} \Phi \rbrace_{\alpha=0,1,2,3}$,
$\Ifact$, and their \emph{Cartesian} spatial partial derivatives remain
bounded, in appropriate norms (some with $\Ifact$ weights), 
all the way up to the singularity.
In particular, our proof relies on a combination of the derivatives 
$\lbrace \Ifact \partial_{\alpha} \rbrace_{\alpha=0,1,2,3}$ and 
$\lbrace \partial_{\alpha} \rbrace_{\alpha=0,1,2,3}$, where the weighted derivatives 
$ \Ifact \partial_{\alpha}$ act first.
Here and throughout, 
$\partial_0 = \partial_t$ and 
$\lbrace \partial_i \rbrace_{i=1,2,3}$ 
are the standard Cartesian coordinate spatial partial derivatives.

In Prop.~\ref{P:RENORMALIZEDEVOLUTOINEQUATIONS}, we derive
the ``regularized'' system of equations verified by
$\lbrace \Psi_{\alpha} \rbrace_{\alpha=0,1,2,3}$.
Here we only note that the system is first-order hyperbolic
and that a \emph{seemingly dangerous factor of $\Ifact^{-1}$ appears in the equations}
(recall that $\Ifact$ vanishes at the singularity).
However, the factor $\Ifact^{-1}$ is multiplied by the weight
$\weight = \weight(\Ifact^{-1} \Psi_0)$
from equation \eqref{E:WAVE}, 
and due to our assumptions on $\weight$,
we are able to show that the product 
$\Ifact^{-1} \weight(\Ifact^{-1} \Psi_0)$
remains uniformly bounded up to the singularity.
Moreover, the spatial derivatives of the product 
$\Ifact^{-1} \weight(\Ifact^{-1} \Psi_0)$
also are controllable up to the singularity;
it is in this sense that the equations verified by $\lbrace \Psi_{\alpha} \rbrace_{\alpha=0,1,2,3}$
can be viewed as a ``regularizing'' of the original problem.
The proof (see Lemma~\ref{L:ESTIMATESINVOLVINGWEIGHT}) of these bounds for the product $\Ifact^{-1} \weight(\Ifact^{-1} \Psi_0)$
constitutes the most technical analysis of the article
and is based on separately treating regions
where $\Ifact$ is large and $\Ifact$ is small.

To prove that $\partial_t \Phi$ blows up, 
we derive, in an appropriate \emph{localized} region of spacetime,
a pointwise bound for $\Psi_0$ of the form $\Psi_0 \gtrsim 1$.
In view of the evolution equation for $\Ifact$, we see that
such a bound is strong enough to drive $\Ifact$ to $0$ in finite time.
To prove that $\Psi_0 \gtrsim 1$, we of course rely on the size assumptions described 
in the first paragraph of this subsusbsection,
which in particular include the assumption that $\Psi_0|_{\Sigma_0} \gtrsim 1$ (in a localized region).
If we caricature the situation by assuming the estimate\footnote{Here we use the notation ``$A \sim B$''
to imprecisely indicated that $A$ is well-approximated by $B$.} 
$\Psi_0 \sim \delta$
for some $\delta > 0$, then it follows 
from the evolution equation for $\Ifact$ that
$\Ifact \sim 1 - \delta t$, 
$\partial_t \Phi \sim (1- \delta t)^{-1}$,
$\ln \Ifact + \Phi \sim data$,
and thus $\Phi \sim \ln (1- \delta t)^{-1} + data$,
where $data$ is a smooth function determined by the initial data.
Note that $\ln (1- \delta t)^{-1}$ is one of the ODE blowup solutions \eqref{E:ODEBLOWWINGUPSOLUTION}.
It is in this sense that our results yield the stability of ODE-type blowup.

In reality, to close the proof sketch described above, 
we must overcome several major difficulties. The first
is that the blowup time is not known in advance. 
However, we are able to make a good approximate guess
for it, which is sufficient for closing a bootstrap argument.
We now describe what we mean by this. 
The discussion in the previous paragraph suggests 
that the (future) blowup time is approximately $\frac{1}{\mathring{A}_*}$, 
where $\mathring{A}_* := \max_{\Sigma_0} [\mathring{\Psi}_0]_+$ (where $\mathring{A}_* > 0$ by assumption).
Indeed, if we set all spatial derivative terms equal to zero in equation \eqref{E:WAVE},
then the time of first blowup is precisely $\frac{1}{\mathring{A}_*}$.
Our main theorem confirms that for data with small spatial derivatives,
the blowup time is a small perturbation of $\frac{1}{\mathring{A}_*}$.
This is conceptually important in that it enables us to use a bootstrap argument
in which we only aim to control the solution for times less than
$\frac{2}{\mathring{A}_*}$; the factor of $2$ gives us a sufficient margin of error 
to show that the singularity does form, and it allows us, in most cases,
to soak factors of $\frac{1}{\mathring{A}_*}$ into the constants ``$C$''
in our estimates; see Subsect.\ \ref{SS:CONVENTIONSFORCONSTANTS} for further
discussion on our conventions for constants.

The second and main difficulty that we encounter in our proof
is that we need to derive energy estimates
for $\lbrace \Psi_{\alpha} \rbrace_{\alpha=0,1,2,3}$
that hold up to the singularity and, 
at the same time, to control the integrating factor $\Ifact$;
most of our work in this paper is towards this goal.
Our energies are roughly of the following form,
where $V = (V_0,V_1,V_2,V_3)$ should be thought of as
some $k^{th}$ Cartesian spatial derivative of $(\Psi_0,\Psi_1,\Psi_2,\Psi_3)$:
\begin{align} \label{E:INTROBASICENERGY}
	\mathbb{E}[V]
	& = \mathbb{E}[V](t)
		:= \int_{\Sigma_t}
					\left\lbrace
						V_0^2
						+
						\sum_{a=1}^3
						\weight(\Ifact^{-1} \Psi_0)
						V_a^2
					\right\rbrace
			 \, d \underline{x}.
	\end{align}
	For the data under study, $\mathbb{E}[V](0)$ is small.
	Since $\mathcal{I}$ is small near the singularity and $\Psi_0$ is order-unity,
	our assumptions on $\weight$ imply that the factor $\weight(\Ifact^{-1} \Psi_0)$
	on RHS~\eqref{E:INTROBASICENERGY} is small near the singularity, i.e., 
	the energy provides only weak control over $V_a^2$. 
	This makes it difficult to control certain terms in the energy identities,
	which arise from commutator error terms (that are generated upon commuting the evolution equations for 
	$\lbrace \Psi_{\alpha} \rbrace_{\alpha = 0,1,2,3}$ with spatial derivatives)
	and from the basic integration by parts argument that we use to derive the energy
	identities.
	To control the most difficult error integrals,
	we exploit the following spacetime integral, which also appears in the energy identities
	(roughly it is generated when $\partial_t$ falls on the weight $\weight(\Ifact^{-1} \Psi_0)$ on RHS~\eqref{E:INTROBASICENERGY}):
	\begin{align} \label{E:INTROKEYSPACETIMEINTEGRAL}
	\mathfrak{I}[V](t) :=
	\sum_{a=1}^3
		\int_{s=0}^t
		\int_{\Sigma_s}
			(\Ifact^{-1} \Psi_0)^2 \weight'(\Ifact^{-1} \Psi_0)
			(V_a)^2
		\, d \underline{x}
		\, ds,
	\end{align}
	where $\weight'(y) = \frac{d}{dy} \weight(y)$.
	A good model scenario to keep in mind is the case $\weight = \frac{1}{1 + \partial_t \Phi}$
	in regions where $\partial_t \Phi$ is large (and thus the energy \eqref{E:INTROBASICENERGY} is weak),
	in which case $\weight' = - \frac{1}{(1 + \partial_t \Phi)^2}$,
	and the factor 
	$
	(\Ifact^{-1} \Psi_0)^2 \weight'(\Ifact^{-1} \Psi_0)
	$
	on RHS~\eqref{E:INTROKEYSPACETIMEINTEGRAL}
	can be expressed as
	$
	- \frac{(\partial_t \Phi)^2}{(1 + \partial_t \Phi)^2}
	$.
	In view of our assumptions on $\weight$, the 
	term $\weight'(\Ifact^{-1} \Psi_0)$ has a \emph{quantitatively negative} sign
	in the difficult regions where $\Ifact$ is small (which is equivalent to the largeness of $\partial_t \Phi$). 
	More precisely, 
	\eqref{E:INTROKEYSPACETIMEINTEGRAL} has a \emph{friction-type} sign.
	This is important because the difficult error integrals mentioned above 
	can be bounded by $\lesssim \varepsilon \mathfrak{I}[V](t)$,
	where, roughly, $\varepsilon$ is the small $L^{\infty}$ size of the spatial derivatives.
	For this reason, the integral \eqref{E:INTROKEYSPACETIMEINTEGRAL} can be used
	to absorb the difficult error integrals. 
	In total, this allows us to prove a priori energy estimates, roughly of the following form:
	\begin{align} \label{E:CARICATURENERGYAPRIORI}
		\mathbb{E}[V](t)
		+ 
		\mathfrak{I}[V](t)
		& \leq data \times C \exp(C t),
	\end{align}
	where ``$data$'' is, roughly, the small size of the spatial derivatives of the data.
	For our proof to close, RHS~\eqref{E:CARICATURENERGYAPRIORI} must be sufficiently small.
	Thanks to our bootstrap assumption that $t < \frac{2}{\mathring{A}_*}$,
	it suffices to choose that initial data so that
	$
	data \times C \exp(\frac{C}{\mathring{A}_*})
	$
	is sufficiently small. This is one example of the \emph{nonlinear smallness} of the spatial derivatives, 
	relative to $\mathring{A}_*$, that we impose to close the proof.
	In reality, to make this procedure work,
	we must separately treat regions where $\Ifact$ is small and $\Ifact$ is large; see
	Prop.~\ref{P:APRIORIESTIMATES} and its proof for the details.
	We stress that absorbing the difficult error integrals into the friction integral \eqref{E:INTROKEYSPACETIMEINTEGRAL}
	\emph{is crucial for showing that the energies remain bounded up to the vanishing of $\Ifact$},
	which is in turn central for our approach.
	In the model case $\weight = \frac{1}{1 + \partial_t \Phi}$,
	if we had instead tried to directly control the difficult error integrals by the energy, 
	then we would have obtained the inequality
	$\mathbb{E}[V](t) \leq 
	C
	\int_{s=0}^t
		\left\|
			\partial_t \Phi
		\right\|_{L^{\infty}(\Sigma_s)}
		\mathbb{E}[V](s)
	\, ds
	+ 
	\cdots
	$.
	Since 
		$\left\|
			\partial_t \Phi
		\right\|_{L^{\infty}(\Sigma_s)}$
	goes to infinity at a non-integrable rate\footnote{With a bit of additional effort,
	Theorem~\ref{T:STABILITYOFODEBLOWUP} could be sharpened to show that
	$
	\displaystyle
	\left\|
			\partial_t \Phi
		\right\|_{L^{\infty}(\Sigma_t)}$
		blows up like
		$\frac{c}{T_{(Lifespan)} - t}$,
		where $c$ is a positive data-dependent constant.}
	as the blowup-time is approached,
	this would have led to a priori energy estimates
	allowing for the possibility that the energies blow up at the singularity, 
	which would have completely invalidated our philosophy of obtaining non-singular estimates 
	for the $\lbrace \Psi_{\alpha} \rbrace_{\alpha=0,1,2,3}$.
	We also highlight that the regularity theory of $\Ifact$ is somewhat subtle at top order:
	our proof requires that we show that $\Ifact$ and $\Phi$ have the same degree of differentiability 
	and that the estimates for $\Ifact$ do not involve any dangerous factors of $\Ifact^{-1}$;
	these features are not immediately apparent from equation \eqref{E:INTROINTEGRATINFACTORODEANDIC}.

The circle of ideas tied to the ``regularization approach'' that we have taken 
here seems to be new in the context of proving the stability of ODE-type blowup
for a quasilinear wave equation.
However, our approach has some parallels with
the known proofs of stable shock formation in multiple spatial dimensions, which we describe 
in Subsect.\ \ref{SS:PRIORBREAKDOWNRESULTS}. In those problems, the crux of the proofs also involve
quasilinear integrating factors that ``hide'' the singularity.
In shock formation problems, the integrating factor is 
tied to nonlinear geometric optics,\footnote{In the shock formation problems 
described in Subsect.\ \ref{SS:PRIORBREAKDOWNRESULTS}, 
the integrating factor is the inverse foliation density 
of a family of characteristic hypersurfaces, which are the level sets of an eikonal function.} 
and its top-order regularity theory 
is very difficult (much more so than the top-order regularity theory of the integrating factor employed in the present article). 
The proofs also crucially rely on friction-type spacetime integrals,
in analogy with \eqref{E:INTROKEYSPACETIMEINTEGRAL},
that are available because the integrating factor has a negative derivative (in an appropriate direction)
in regions near the singularity.
However, in shock formation problems, the top-order energy identities feature dangerous terms,
analogous to terms of strength $\Ifact^{-1}$, which leads to a priori energy estimates
allowing for the possibility that the high-order energies might blow up like $\Ifact^{-P}$ 
for some large universal constant $P$. This makes it difficult to derive the non-singular estimates
at the lower derivative levels, which are central for closing the proof. In contrast, in our work here,
the difficult factors of $\Ifact^{-1}$ are always multiplied by the term $\weight$, 
which effectively ameliorates them, making it easier to close the energy estimates. 
On the other hand, the singularities that form in the solutions from our main results are much
more severe in that $\Phi$ and $\partial_t \Phi$ blow up;
in contrast, in the shock formation results 
(see Subsect.\ \ref{SS:PRIORBREAKDOWNRESULTS})
for equations whose principal part is similar
to that of \eqref{E:FIRSTMODELWAVE},
$|\Phi|$ and $\lbrace |\partial_{\alpha} \Phi| \rbrace_{\alpha=0,1,2,3}$ remain
bounded up to the singularity, while
$\max_{\alpha, \beta = 0,1,2,3}|\partial_{\alpha} \partial_{\beta} \Phi|$
blows up in finite time. Our approach to proving Theorem~\ref{T:ROUGHMAINTHM} 
also has some parallels with Kichenassamy's stable blowup results
\cite{sKi1996a} for semilinear wave equations with exponential nonlinearities,
but we defer further discussion of this point until the next subsection.

\subsection{Our results in the context of prior breakdown results}
\label{SS:PRIORBREAKDOWNRESULTS}
There are many prior breakdown results for solutions to various hyperbolic equations,
especially of wave type. Here we give a non-exhaustive account of some of these works,
which is meant to give the reader some feel 
for the kinds of results that are known
and how they compare with/contrast against our main results.
In particular, we aim to expose how the
proof techniques vary wildly between different types of blowup results.
We separate the results into seven classes.
\begin{enumerate}
	\item (\textbf{Proofs of blowup by contradiction})
		For various hyperbolic systems, there are proofs of blowup by contradiction,
		based on showing that for smooth solutions, 
		certain spatially averaged quantities satisfy ordinary differential
		inequalities that force them to blow up in finite time, 
		contradicting the assumed smoothness.
		Notable contributions of this type are
		John's works \cites{fJ1979,fJ1981} on several classes of nonlinear wave equations 
		with signed nonlinearities and Sideris' proof \cite{tS1984} of blowup for various hyperbolic systems,
		for semilinear wave equations in higher dimensions \cite{tS1984b} (which improved upon Kato's result \cite{tK1980}),
		and for the compressible Euler equations in three spatial dimensions \cite{tS1985}.
		See also \cite{yGsTZ1998} for similar results in the case of the relativistic Euler and Euler--Maxwell equations.
		None of these results yield constructive information about the nature of the blowup, nor do they
		apply to the wave equations under study here.
\item (\textbf{Blowup for semilinear wave equations with power-law nonlinearities})
		There are many interesting constructive blowup results, 
		in various spatial dimensions, 
		for focusing semilinear wave equations of the form
		$\square_m \Phi = - |\Phi|^{p-1} \Phi$,
		where $\square_m := - \partial_t^2 + \Delta$ is the wave operator of
		the Minkowski metric $m$. 
		A notable difference between these works
		and our work here is that these works relied on a careful analysis 
		of the spectrum of a linearized operator.
		We now discuss some specific examples.
		
		In \cite{rD2010}, in three spatial dimensions and under the assumption of radial symmetry, 
		Donninger proved
		the nonlinear stability of the ODE blowup solutions 
		$\Phi_{(ODE);T} := c_p(T-t)^{- \frac{2}{p-1}}$,
		where\footnote{Actually, for convenience, Donninger considered the semilinear term
		$- \Phi^p$ in \cite{rD2010}. However, as he noted there, his work could be extended to apply to
		the term $- |\Phi|^{p-1} \Phi$ for $p > 1$.}  
		$p=3,5,7,\cdots$.
		In three spatial dimensions, in the subcritical cases $p \in (1,3]$,
		Donninger--Sch\"{o}rkhuber proved \cite{rDbS2012} an asymptotic stability result
		for $\Phi_{(ODE);T}$
		under radially symmetric perturbations of the data
		in the energy space, thereby sharpening (in the near-ODE case) the works \cites{fMhZ2003,fMhZ2005},
		which yielded the \emph{all} solutions that blow up 
		do so at the rate
		$
		(T-t)^{- \frac{2}{p-1}}
		$,
		but which did not yield the asymptotic profile.
		In \cite{rDbS2014}, Donninger--Sch\"{o}rkhuber 
		extended their stability results (still within radial symmetry)
		to the supercritical cases $p > 3$, but 
		they assumed additional regularity on the initial data 
		(which they believed to be essential for closing the proof).
		
		In three spatial dimensions, in the critical case $p=5$,
		there are many blowup results tied to the ground
		state solution $W(r) := (1 + r^2/3)^{-1/2}$.
		In \cite{ceKfM2008}, for solutions with (conserved) energy below that of the ground state,
		Kenig--Merle established a sharp dichotomy showing that solutions blow up 
		in finite time to the past and future
		if $\| \Phi \|_{\dot{H}^1(\Sigma_0)} > \| W \|_{\dot{H}^1(\Sigma_0)}$, while they
		exist globally and scatter if $\| \Phi \|_{\dot{H}^1(\Sigma_0)} < \| W \|_{\dot{H}^1(\Sigma_0)}$.
		For the same equation, the authors of \cite{jKwSdT2009} proved the existence of
		radially symmetric ``slow'' Type II blowup solutions $\Phi(t,r) = \lambda^{1/2}(t) W(\lambda(t) r) + w(t,r)$,
		where $w$ is a small error term,
		$\lambda(t) := t^{-1-\nu}$, $\nu > 1/2$, and the singularity occurs at $t=0$.
		In this context, a Type II singularity is such that the solution remains uniformly bounded in the energy space
		(which is critical) up to the time of first blowup.
		The results were extended to $\nu > 0$ in \cite{jKwS2014}.
		In \cite{rDmHjKwS2014}, the results were extended 
		to cases in which $\lambda(t)$ does not behave like a power law.
		In \cite{mHpR2012}, Hillairet--Rapha\"{e}l constructed
		Type II blowup solutions for the critical focusing wave equation in four spatial dimensions.
		Jendrej treated the case of five spatial dimensions in \cite{jJ2017}.
		For the radial critical focusing wave equation in three spatial dimensions,
		the work \cite{tDcKfM2011} yielded that if the blowup-time
		$T$ is finite
		and if the quantitative type II condition
		$\sup_{t \in [0,T)} 
		\left\lbrace
			\| \partial_t \Phi \|_{L^2(\Sigma_t)}^2 
			+ 
			\| \nabla \Phi \|_{L^2(\Sigma_t)}^2
		\right\rbrace
		\leq \| \nabla W \|_{L^2}^2 + \eta_0
		$
		holds, where $W$ is the ground state and $\eta_0 > 0$ is a small constant,
		then the blowup asymptotics are of the type exhibited by the solutions constructed in \cite{jKwSdT2009}.
		The results were extended to the non-radial case in three and five spatial dimensions in 
		\cite{tDcKfM2012}. Similar results were obtained in the case of four 
		spatial dimensions in \cite{rCcKaLwS2014} in the radial case.
		In \cite{tDcKfM2013}, the authors gave a detailed description 
		of the possible large-time behaviors of all finite-energy
		radial solutions to the focusing critical wave equation in three spatial
		dimensions, extending the work \cite{tDcKfM2012b}, where information
		along a sequence of times was obtained.
		In \cite{jJ2016}, for $n \in \lbrace 3,4,5 \rbrace$ spatial dimensions,
		Jendrej proved an upper bound for
		the blowup rate $\lambda(t)$
		for Type II blowup solutions whose asymptotics are
		$\Phi(t,r) = [\lambda(t)]^{(n-2)/2} W(\lambda(t) r) + w(t,r)$,
		with $w$ sufficiently regular.
		\item (\textbf{Constructive blowup results for wave maps})
		There are similar blowup results for some wave maps whose targets
		admit a non-trivial harmonic map.
		For example, for the critical case of the wave maps equation
		$\square_m \Phi = \Phi(|\partial_t \Phi|^2 - |\nabla \Phi|^2)$,
		where $\Phi:\mathbb{R}^{1+2} \rightarrow \mathbb{S}^2$,
		under the equivariant symmetry assumption
		$\Phi(t,r) = (k \theta,\phi(t,r))$, where
		the first and second entries on the RHS are Euler angles parameterizing $\mathbb{S}^2$
		and $k \in \mathbb{Z}_+$,
		there are blowup results tied to the 
		ground state $Q(r) := 2 \arctan(r^k)$.
		In \cite{iRjS2010}, Rodnianski--Sterbenz
		gave a sharp description of \emph{stable} blowup 
		when $k \geq 4$. They showed that (under the symmetry assumptions),
		there is an open set of data with energy slightly larger than the ground state
		whose solutions blow up at a time $T < \infty$.
		Moreover, the asymptotics can be described as
		$\phi(t,r) = Q(t,r/\lambda(t)) + q(t,r)$,
		where $\lambda(t) \to 0$ as $t \uparrow T$,
		$\lambda(t) \geq \frac{T-t}{|\ln(T-t)|^{1/4}}$,
		and
		$(q,\partial_t q)$ is small in $\dot{H}^1 \times L^2$.
		In particular, $Q$ is the universal blowup profile.
		A key point of the proof is to derive and analyze an appropriate \emph{modulation equation},
		that is, the ODE (which is coupled to the PDE) that governs the evolution of $\lambda(t)$.
		The function $\lambda(t)$ is somewhat analogous to the integrating factor $\Ifact$ that we use in our work here.
		In \cite{pRiR2012},
		Rodnianski--Rapha\"{e}l extended the results to all cases $k \geq 1$,
		proving \emph{stable} blowup with
	 	$\lambda(t) = c_k \left(1 + o(1) \right) \frac{T-t}{\left| \ln(T-t) \right|^{\frac{1}{2k-2}}}$
		as $t \uparrow T$ in the cases $k \geq 2$,
		and, in the case $k=1$,
		$\lambda(t) = (T-t) \exp \left(- \sqrt{\ln|T-t|} + O(1) \right)$
		as $t \uparrow T$.
		In \cite{jKwSdD2008}, in the case $k=1$,
		the authors proved the existence 
		of a continuum of related solutions 
		(believed to be non-generic)
		exhibiting the blowup rates
		$\lambda(t) = (T-t)^{\nu}$, where $\nu > 3/2$.
		The results were extended to $\nu > 1$ in \cite{cGjK2015}.
		In \cites{rCcKaLwS2015a}, in the equivariance class $k=1$, 
		the authors proved that within the sub-class of degree $0$ maps (i.e., in radial coordinates $(t,r)$, one assumes
		$\Phi(0,0) = \Phi(0,\infty) = 0$),
		there exist solutions with energy above but arbitrarily close to twice the energy of the ground
		state that blow up in finite time. 
		Within the sub-class of degree $1$ maps (i.e., $\Phi(0,0) = 0$ and $\Phi(0,\infty) = \pi$),
		for maps with energy bigger than that of the ground state but less than three times 
		the energy of the ground state, the authors show that if a singularity forms,
		then the solution has asymptotics whose blowup profiles
		are the same as those from the works \cites{jKwSdD2008,iRjS2010,pRiR2012}.
		For equivariant wave maps $\Phi : \mathbb{R} \times \mathbb{S}^2 \rightarrow \mathbb{S}^2$,
		in the class $k=1$,
		Shahshahani proved \cite{sS2016} the existence of a continuum of blowup solutions.
		In \cites{rD2011,rDbSpA2012},
		in the supercritical context of equivariant wave maps from $\mathbb{R}^{1+3}$ into $\mathbb{S}^3$,
		the authors proved the stability of self-similar blowup solutions
		$\Phi_T(t,r) := 2 \arctan \left(\frac{r}{T-r} \right)$. 
		More precisely, the results were conditional and relied 
		on some mode stability results for which there is strong evidence
		in the form of analysis and numerics.
	\item (\textbf{Blowup for semilinear wave equations with exponential nonlinearities})
		In \cite{sKi1996a}, for the focusing semilinear wave equation $\square_m \Phi = - e^{\Phi}$
		in three spatial dimensions,
		Kichenassamy proved that the singular solution $\ln(2/t^2)$ 
		is stable under perturbations of the data along the constant-time hypersurface
		$\lbrace t = -1 \rbrace$. 
		Moreover, he showed that the blowup-hypersurface is
		of the form $\lbrace t = f(x) \rbrace$, where $f(x)$ loses Sobolev regularity 
		compared to the initial data.
		It would be interesting to see if our main results could be extended to show a similar result
		for the equations under study here. More precisely, we conjecture that 
		a portion of the blowup surface is $\lbrace \Ifact = 0 \rbrace$ for the solutions under study here.
		Kichenassamy's work has one key feature in common with ours: 
		he devises a reformulation of the wave equation in which no singularity is visible, in his
		case by constructing a singular change of coordinates adapted to the singularity; 
		this is broadly similar to the approach taken by authors who have proved shock formation results, 
		as we describe just below. 
		However, unlike the ``forwards approach'' that we take in this article, 
		Kichenassamy used a ``backwards approach'' 
		in which he first solved problems in which the singularity was \emph{prescribed} 
		along blowup-hypersurfaces and then showed that the map from the singularity to the Cauchy data along
		$\lbrace t = -1 \rbrace$ is invertible. His proof of the invertibility of the 
		map from the singularity data to the Cauchy data
		was based on studying appropriately linearized versions of 
		the equations and on using Fuchsian techniques. 
		The linearized equations exhibited derivative loss, and 
		Kichenassamy used a Nash--Moser approach to overcome the loss.
	\item (\textbf{Shock formation for quasilinear equations})
	Roughly speaking, a shock singularity is such that the solution remains bounded but
	one of its derivatives blows up. There are many classical shock formation results 
	in one spatial dimension, based on the method of characteristics,
	with important contributions coming from
	Riemann \cite{bR1860},
	Lax \cites{pL1964,pL1972,pL1973},
	and John \cite{fJ1974}, among others. Even in one spatial dimension, the field is still active, 
	as is evidenced by the recent work of \cite{dCdRP2016} of Christodoulou--Perez, 
	in which they significantly sharpened John's work \cite{fJ1974}, giving a complete
	description of the singularity.

	Alinhac \cites{sA1999a,sA1999b,sA2001b} obtained the first results
	on shock formation without symmetry assumptions in more than one spatial
	dimension. The main new difficulty in more than one spatial dimension
	is that the method of characteristics must
	be supplemented with energy estimates, which leads to enormous technical complications.
	Alinhac's work applied 
	to small-data solutions to a class of
	scalar quasilinear wave equations
	of the form 
	\begin{align} \label{E:ALINHACWAVE}
		(g^{-1})^{\alpha \beta}(\partial \Phi) \partial_{\alpha} \partial_{\beta} \Phi = 0
	\end{align}
	that fail to satisfy the null condition.
	He showed that for a set of ``non-degenerate'' small data,
	$\Phi$ and $\partial \Phi$ remain bounded, while $\partial^2 \Phi$ blows up in finite time
	due to the intersection of the characteristics.
	Alinhac's proof fundamentally relied on nonlinear geometric optics,
	that is, on an eikonal function, which is a solution
	to the eikonal equation
	\begin{align} \label{E:WAVEEIKONAL}
	(g^{-1})^{\alpha \beta}(\partial \Phi) \partial_{\alpha} u \partial_{\beta} u
	& = 0,
\end{align}
supplemented with appropriate initial data.
The level sets of $u$ are characteristic hypersurfaces for equation \eqref{E:ALINHACWAVE}.
As it turns out, the intersection of the level sets of $u$ is tied to the formation of a singularity
in the solution to \eqref{E:ALINHACWAVE}.
Much like in the present work, the main estimates in Alinhac's proof did not concern singularities;
the crux of his proof was to construct a system of geometric coordinates, one of which is $u$, 
and to prove that relative to them,
the solution remains very smooth, except possibly at the very high derivative levels. 
He then showed that a singularity forms 
in the standard second-order derivatives $\partial^2 \Phi$ 
as a consequence of a finite-time degeneracy 
between the geometric coordinates and the standard coordinates; roughly, the level sets of $u$
intersect and cause the blowup, much like in the classical example of the blowup of solutions to Burgers' equation.
The main challenge in the proof is that
to derive energy estimates relative to the geometric coordinates, one must control the eikonal function,
whose top-order regularity properties are difficult to obtain; naive estimates lead to the loss of a derivative.
The regularity properties of eikonal functions had previously been understood
in certain problems for quasilinear wave equations in which singularities did not form. 
For example, the first quasilinear wave problem in which the regularity properties of eikonal functions were fully exploited was the
celebrated proof \cites{dCsK1993} of the stability of Minkowski spacetime. 
Eikonal functions have also played a central role
in proofs of low-regularity well-posedness for quasilinear wave equations \cites{sKiR2003,sKiR2005d,hSdT2005,sKiRjS2015}.
However, unlike in these problems,
in the problem of shock formation, the top-order geometric energy estimates feature a degenerate weight
that vanishes as the shock forms, which leads to a priori estimates allowing for the possibility that the high-order
energies might blow up; note that this possible geometric energy blowup is distinct from the formation of the shock, 
which happens at the low derivative levels relative to the standard coordinates. 
The ``degenerate weight'' is the inverse foliation density of the level sets of $u$.
It vanishes when the characteristics intersect, 
and it is in some ways analogous to the integrating factor $\Ifact$ that we use in our work here.
Alinhac closed the singular top-order energy estimates with a Nash--Moser iteration scheme
that was adapted to the singularity and that handled the issue of the regularity theory of $u$ in a different
way than \cites{dCsK1993,sKiR2003,sKiR2005d,hSdT2005,sKiRjS2015}. 
He then used a ``descent scheme'' to show that the top-order geometric energy blowup
does not propagate down too far to the lower derivative levels. Consequently,
the solution remains highly differentiable relative to the geometric coordinates.

Due to his reliance on the Nash--Moser iteration scheme, Alinhac's proof applied only to 
``non-degenerate'' initial data such that the first
singularity is isolated in the constant-time hypersurface of first blowup, and his framework breaks down
precisely at the time of first blowup. 
For this reason, his approach is inadequate
for following the solution to the boundary of the maximal development
of the data (see Footnote~\ref{FN:MAXIMALDEVELOPMENT}), 
which intersects the future of the first singular time.
In his breakthrough work \cite{dC2007}, Christodoulou overcame this drawback and significantly sharpened
Alinhac's results for the quasilinear wave equations of irrotational relativistic fluid mechanics.
More precisely, Christodoulou's proof yielded a sharp description of the solution up to the boundary of the maximal development. 
This information was essential even for setting up the shock development problem, which, roughly speaking,
is the problem of uniquely extending the solution past the singularity in a weak sense, subject to 
appropriate jump conditions.
We note that the shock development problem
in relativistic fluid mechanics was solved in spherical symmetry by Christodoulou--Lisibach in \cite{dCaL2016}
and, by Christodoulou in a recent breakthrough work \cite{dC2017}, 
for the non-relativistic compressible Euler equations without symmetry assumptions 
in a restricted case 
(known as the restricted shock development problem)
such that the jump in entropy across the shock hypersurface is ignored.

The wave equations studied by Christodoulou in \cite{dC2007} form a sub-class of the ones \eqref{E:ALINHACWAVE} 
studied by Alinhac. They enjoy special
properties that Christodoulou used in his proofs, notably
an Euler-Lagrange formulation whose Lagrangian is invariant under various symmetry groups.
The main technical improvement afforded by Christodoulou's framework is that 
in closing the energy estimates, he avoided using a Nash--Moser iteration scheme
and instead used an approach similar to the one employed in the aforementioned works 
\cites{dCsK1993,sKiR2003}. This approach is more robust and is capable of accommodating solutions such that 
the blowup occurs along a hypersurface, which, in the problem of shock formation, 
is what typically occurs along a portion
of the boundary of the maximal development.
Christodoulou's results have since been extended in many directions, including
to apply to other wave equations \cites{dCsM2014,jS2016b},
different sets of initial data \cites{jSgHjLwW2016,sMpY2017,sM2016},
the compressible Euler equations with non-zero vorticity \cites{jLjS2016a,jLjS2016b,jS2017a},
systems of wave equations with multiple speeds \cite{jS2017b},
and quasilinear systems in which a solution to a transport equation forms a shock
\cite{jS2017c}.
Some of the earlier extensions are explained in detail in the survey article \cite{gHsKjSwW2016}.
\item (\textbf{Breakdown results for Einstein's equations})
	The Einstein equations of general relativity have many remarkable properties and as such, 
	it is not surprising that there are stable breakdown results that are specialized to 
	these equations. Here we simply highlight the following constructive 
	results in three spatial dimensions without symmetry assumptions:
	Christodoulou's breakthrough results \cite{dC2009} on the formation of trapped surfaces
	and the stable singularity formation results \cites{iRjS2014a,iRjS2014b,jL2013}.
	The work \cite{iRjS2014b} can be viewed as a stable ODE-type blowup result 
	for Einstein's equations in which the wave speed became infinite at the singularity.
	Note that in contrast, for equation \eqref{E:WAVE}, the wave speed vanishes when $\partial_t \Phi$
	blows up. 
\item (\textbf{Finite time degeneration of hyperbolicity})
	In \cite{jS2016}, we studied the wave equations
	$- \partial_t^2 \Psi + (1 + \Psi)^P \Delta \Psi = 0$
	in three spatial dimensions, for $P = 1,2$. We showed that there
	exists an open set of data such that $\Psi$ is initially small but
	$1 + \Psi$ vanishes in finite time,
	corresponding to a breakdown in the hyperbolicity of the equation,
	but without any singularity forming. The difficult part of the proof is 
	closing the energy estimates in regions where $1 + \Psi$ is small.
	The proof has some features in common with the proof of the main results of this paper.
	For example, the proof relies on monotonicity tied to the sign of $\partial_t \Psi$
	and the small size of $\nabla \Psi$. In particular, this leads to the availability
	of a friction-type integral in the energy identities,
	analogous to the one \eqref{E:INTROKEYSPACETIMEINTEGRAL},
	which is crucially important for controlling certain error terms.
\end{enumerate}

\subsection{Different kinds of singularity formation within the same quasilinear hyperbolic system}
\label{SS:DIFFERENTKINDSOFSINGULARITYFORMATION}
In this subsection, we show that there are quasilinear wave equations,
closely related to the wave equation \eqref{E:WAVE},
that can exhibit two distinct kinds of blowup: 
ODE-type blowup for one set of data, and
the formation of shocks for another set.
The ODE-type blowup is provided by our main
results, so in this subsection, 
most of the discussion is centered on
shock formation. Our discussion is based on ideas and
techniques found in the works \cites{dC2007,jS2016b}.

To initiate the discussion, we define
\begin{align} \label{E:PHIALPHA}
	\Phi_0 
	& := \partial_t \Phi.
\end{align}
For convenience, we will restrict our discussion to the specific
weight $\weight = \frac{1}{1 + \partial_t \Phi} = \frac{1}{1 + \Phi_0}$,
though similar results hold for any weight that verifies the assumptions
of Subsect.\ \ref{SS:WEIGHTASSUMPTIONS}.
To proceed, we differentiate equation \eqref{E:WAVE} with $\partial_t$
to obtain the following closed equation in $\Phi_0$:
\begin{align} \label{E:PHI0CLOSEDEQUATION}
\partial_t^2 \Phi_0
-
\frac{1}{1 + \Phi_0}
\Delta \Phi_0
=
-
\frac{1}{1 + \Phi_0}
(\partial_t \Phi_0)^2
+
\frac{2 \Phi_0}{1 + \Phi_0} \partial_t \Phi_0
+
\frac{3 \Phi_0^2}{1 + \Phi_0} \partial_t \Phi_0.
\end{align}
In the remainder of our discussion of shock formation,
we will only consider plane symmetric solutions,
that is, solutions that depend only on $t$ and $x^1$.
Note that in equation \eqref{E:PHI0CLOSEDEQUATION}, 
$\Delta = \partial_1^2$ for plane symmetric solutions.

To study plane symmetric solutions to \eqref{E:PHI0CLOSEDEQUATION}, we will use the characteristic vectorfields
\begin{align}
	L := \partial_t + \frac{1}{\sqrt{1 + \Phi_0}} \partial_1,
	&& \underline{L} := \partial_t - \frac{1}{\sqrt{1 + \Phi_0}} \partial_1.
\end{align}
We next define the characteristic coordinate $u$ to be the solution to the following transport equation:
\begin{align} \label{E:EIKONALDEF}
	L u 
	& = 0,
	&
	u|_{\Sigma_0}
	& = 1 - x^1.
\end{align}
We view $u$ as a new coordinate adapted to the characteristics,
and we will use the ``geometric'' coordinate system $(t,u)$ when analyzing solutions.
In particular, the level sets of $u$ are characteristic for equation \eqref{E:PHI0CLOSEDEQUATION}.
We define $\Sigma_t^{u'}$, 
relative to the geometric coordinates,
to be the following subset:
$\Sigma_t^{u'} := \lbrace (t,u) \ | \ 0 \leq u \leq u' \rbrace$.
Note that $\Sigma_0^1$ can be identified with an orientation reversed version of the unit $x^1$ interval $[0,1]$.
Associated to $u$, we define the \emph{inverse foliation density} $\upmu > 0$
by 
\begin{align} \label{E:MUDEF}
	\upmu & := \frac{1}{\partial_t u}.
\end{align}
$1/\upmu$ is a measure of the density of the level sets of $u$.
$\upmu = 0$ corresponds to the intersection of the characteristics,
that is, the formation of a shock.
From \eqref{E:EIKONALDEF}, it follows that $\upmu|_{\Sigma_0} = \sqrt{1 + \Phi_0} = 1 + \mathcal{O}(\Phi_0)$ 
(for $\Phi_0$ small). 
One can check that from the above definitions, we have
$L u = 0$, $L t = 1$, $\upmu \underline{L} t = \upmu $, 
and $\upmu \underline{L} u = 2 $.
In particular, 
$L = \frac{d}{dt}$ along the integral curves of $L$
and $\upmu \underline{L} = 2 \frac{d}{du}$ along the integral curves of
$\upmu \underline{L}$.

Next, we differentiate equation \eqref{E:WAVE}
with $\partial_t$ and carry out tedious but straightforward calculations 
to obtain the following system in $\Phi_0$ and $\upmu$:
\begin{subequations}
\begin{align}
	L (\upmu \underline{L} \Phi_0)
	& = 
		- 
		\frac{1}{2(1 + \Phi_0)} (L \Phi_0) (\upmu \underline{L} \Phi_0)
		\label{E:LONOUTSIDE} \\
& \ \
	+
		\upmu 
		\frac{\Phi_0}{1 + \Phi_0} 
		\left\lbrace
			1 + \frac{3}{2} \Phi_0
		\right\rbrace
		L \Phi_0
		+
		\frac{\Phi_0}{1 + \Phi_0} 
		\left\lbrace
			1 + \frac{3}{2} \Phi_0
		\right\rbrace
		(\upmu \underline{L} \Phi_0),
		\notag \\
	\upmu \underline{L} L \Phi_0
	& =
-
\frac{\upmu}{4(1 + \Phi_0)}
(L \Phi_0)^2
- 
\frac{3}{4(1 + \Phi_0)} (L \Phi_0) (\upmu \underline{L} \Phi_0)
	\label{E:LBARONOUTSIDE} \\
& \ \
		+
		\upmu 
		\frac{\Phi_0}{1 + \Phi_0} 
		\left\lbrace
			1 + \frac{3}{2} \Phi_0
		\right\rbrace
		L \Phi_0
		+
		\frac{\Phi_0}{1 + \Phi_0} 
		\left\lbrace
			1 + \frac{3}{2} \Phi_0
		\right\rbrace
		(\upmu \underline{L} \Phi_0),
	\notag
	\\
L \upmu
	& = \frac{1}{4(1 + \Phi_0)} \upmu L \Phi_0
		+
		\frac{1}{4(1 + \Phi_0)} (\upmu \underline{L} \Phi_0).
		\label{E:UPMUEVOLUTION}
\end{align}
\end{subequations}

For convenience, we will prove shock formation only for a restricted class of
initial data supported in $\Sigma_0^1$;
as can easily be inferred from our proof,
the shock-forming solutions that we will construct are stable under plane symmetric perturbations,
and our approach could be applied to a much larger set of plane symmetric initial data.
Specifically, we will prove shock formation for solutions corresponding to initial data such that 
\begin{align} \label{E:SHOCKDATA}
	\sup_{\Sigma_0^1} |\Phi_0| 
	& \leq \epsilon,
	&&
	L \Phi_0|_{\Sigma_0} = 0,
	&
	\sup_{\Sigma_0^1} |\underline{L} \Phi_0|
	& = 4,
\end{align}
such that $\underline{L} \Phi_0|_{\Sigma_0^1}$ is \emph{negative} at some maximum of $|\underline{L} \Phi_0|$ on $\Sigma_0^1$,
and such that $\epsilon$ is small. The negativity of $\underline{L} \Phi_0$ will drive the vanishing of $\upmu$.
To show the existence of such data, it is convenient to refer to the Cartesian coordinate $x^1$.
Specifically, we fix a smooth non-trivial function $f = f(x^1)$ supported in $\Sigma_0^1$
and set
$\Phi_0|_{\Sigma_0^1}(x^1) := \upkappa f(\uplambda x^1)$,
where $\upkappa$ and $\uplambda$ are real parameters.
Note that $\partial_1 \Phi_0|_{\Sigma_0}(x^1) = \upkappa \uplambda f'(\uplambda x^1)$.
We then set 
$\partial_t \Phi_0|_{\Sigma_0^1}
:= 
- \frac{1}{\sqrt{1 + \Phi_0|_{\Sigma_0^1}}} \partial_1 \Phi_0|_{\Sigma_0^1}
$, which implies that 
$L \Phi_0|_{\Sigma_0^1} = 0$
and
$\underline{L} \Phi_0|_{\Sigma_0^1}(x^1) = - 2 \upkappa \uplambda \frac{1}{\sqrt{1 + \upkappa f(\uplambda x^1)}} f'(\uplambda x^1)$.
We now choose $|\upkappa|$ sufficiently small and $\uplambda$
sufficiently large, which allows us to achieve \eqref{E:SHOCKDATA} with $\epsilon > 0$ arbitrarily small.
Moreover, by adjusting the sign of $\upkappa$, we can ensure that
$\underline{L} \Phi_0|_{\Sigma_0^1}$ is negative at some maximum of $|\underline{L} \Phi_0|$ on $\Sigma_0^1$.
We also note that from domain of dependence considerations, 
it follows that in terms of the geometric coordinates,
solutions with data supported in $\Sigma_0^1$ 
vanish when $u \leq 0$, and that the level set 
$\lbrace u = 0 \rbrace$ can be described in Cartesian coordinates
as $\lbrace (t,x^1) \ | \ 1 - x^1 + t = 0 \rbrace$.

To derive estimates, we make the following
bootstrap assumptions on any region of classical existence such that
$0 \leq t \leq 2$ and $0 \leq u \leq 1$:
\begin{align} \label{E:INTROSHOCKBA}
0 < \upmu \leq 3,
\qquad
|\Phi_0| \leq \sqrt{\epsilon},
\qquad
|L \Phi_0| \leq \sqrt{\epsilon},
\qquad
|\underline{L} \Phi_0| \leq 5.
\end{align}
Note also that the solution verifies $\Phi_0(t,u=0) = 0$
and that the assumptions \eqref{E:INTROSHOCKBA} are consistent 
with the initial data when $\epsilon$ is small.

We now derive estimates. We define
\begin{align} \label{E:INTROSHOCKGOODDERIVATIVES}
	Q(t,u) 
	& := \sup_{(t',u') \in [0,t] \times [0,u]} \left\lbrace |\Phi_0|(t',u') + |L \Phi_0|(t',u') \right\rbrace.
\end{align}
Note that $Q(0,u) \lesssim \epsilon$, 
while our data support assumptions imply that $Q(t,0) = 0$.
Using the evolution equation \eqref{E:LBARONOUTSIDE}, the bootstrap assumptions,
the fact that $L = \frac{d}{dt}$ along the integral curves of $L$,
and the fact that $\upmu \underline{L} = 2 \frac{d}{du}$ along the integral curves of
$\upmu \underline{L}$,
we deduce 
$Q(t,u)\leq C Q(0,u) + c \int_{t'=0}^t Q(t',u) \, dt' + c \int_{u'=0}^u Q(t,u') \, du'$.
From this estimate and Gronwall's inequality (in two variables), we deduce that there are constants $C > 0$ and $c' > c$
such that for $0 \leq t \leq 2$ and $0 \leq u \leq 1$, we have
\begin{align} \label{E:APRIORIESTIMATEINTROSHOCKGOODDERIVATIVES}
	Q(t,u) \leq  C Q(0,u) e^{c't} e^{c'u} \leq C Q(0,u) e^{3c'} \leq C e^{3c'} \epsilon \lesssim \epsilon.
\end{align}
Using the estimate \eqref{E:APRIORIESTIMATEINTROSHOCKGOODDERIVATIVES} 
and the bootstrap assumptions for
$\upmu$
and $\upmu \underline{L} \Phi_0$
to control the terms on RHS~\eqref{E:LONOUTSIDE},
we deduce
$|L (\upmu \underline{L} \Phi_0)| \lesssim \epsilon$.
Integrating this estimate along the integral curves of $L$
and using that $\upmu(0,u) = 1 + \mathcal{O}(\epsilon)$,
we find that for $0 \leq t \leq 2$ and $0 \leq u \leq 1$, we have
$\upmu \underline{L} \Phi_0(t,u) = \upmu \underline{L} \Phi_0(0,u) + \mathcal{O}(\epsilon)
= \underline{L} \Phi_0(0,u) + \mathcal{O}(\epsilon)$.
Inserting this information into \eqref{E:UPMUEVOLUTION},
we deduce
$L \upmu = \frac{1}{4} \underline{L} \Phi_0(0,u) + \mathcal{O}(\epsilon)$.
Integrating in time and using the initial condition 
$\upmu(0,u) = 1 + \mathcal{O}(\epsilon)$, we deduce that
$\upmu(t,u) = 1 + \frac{1}{4} \underline{L} \Phi_0(0,u) t + \mathcal{O}(\epsilon)
= 1 + \frac{1}{4} \upmu \underline{L} \Phi_0(t,u) t + \mathcal{O}(\epsilon)
$.
We now note that if $\epsilon$ is sufficiently small, then the above estimates
yield strict improvements of the bootstrap assumptions \eqref{E:INTROSHOCKBA}.
By a standard continuity argument in $t$ and $u$, 
this justifies the bootstrap assumptions
and shows that the solution exists 
on regions of the form
$0 \leq t \leq 2$ and $0 \leq u \leq 1$,
as long as $\upmu$ remains positive;
the positivity of $\upmu$ and the above estimates guarantee that
$|\Phi_0| + \max_{\alpha = 0,1} |\partial_{\alpha} \Phi_0|$ is finite. 
Moreover, since (by construction) 
$\sup_{\Sigma_0^1} |\underline{L} \Phi_0| =  4$
and since there is a value $u_* \in (0,1)$ such that
$\underline{L} \Phi_0(0,u_*) = - 4$,
the above estimates for 
$\upmu \underline{L} \Phi_0$
and
$\upmu$ 
guarantee that $\min_{\Sigma_t^u} \upmu = 1 + \mathcal{O}(\epsilon) - t$
and that $\upmu \underline{L} \Phi_0(t,u) \leq - 1$ at points $(t,u)$ with $\upmu(t,u) \leq 1/4$
and $0 \leq t \leq 2$.
It follows that
$\min_{\Sigma_t^1} \upmu$ cannot remain positive for times larger than 
$1 + \mathcal{O}(\epsilon)$
and that $\min_{\Sigma_t^1} \upmu \leq 1/4 \implies \sup_{\Sigma_t^1} |\underline{L} \Phi_0| \geq \frac{1}{\min_{\Sigma_t^1} \upmu}$.
This implies that $\sup_{\Sigma_t^1} |\underline{L} \Phi_0|$ blows up 
at some time $t_{(Shock)} = 1 + \mathcal{O}(\epsilon)$
while $|\Phi_0|$ and $|L \Phi_0|$ remain uniformly bounded by $\lesssim \epsilon$.
We have thus shown that a shock forms.

We now revisit the solutions from our main results under the weight $\weight = \frac{1}{1 + \partial_t \Phi}$.
Notice that for such solutions,
$\Phi_0$ also solves equation \eqref{E:PHI0CLOSEDEQUATION}
but is such that such that $|\Phi_0|$ blows up at the singularity.
This is \emph{different blowup behavior} compared to the shock-forming solutions
to equation \eqref{E:PHI0CLOSEDEQUATION}
constructed above, in which
$|\Phi_0|$ remained bounded.
Notice also that our main theorem requires, roughly, that $\Phi_0|_{\Sigma_0}$
should not be too small, which is in contrast to the initial data
for the shock-forming formation solutions described above.
To close this subsection, we clarify that it could be, in principle, 
that the ODE-type blowup solutions that we have constructed
are \emph{unstable} when viewed as solutions to equation \eqref{E:PHI0CLOSEDEQUATION},
even though they are stable solutions of the 
original wave equation \eqref{E:WAVE}. The key point is that
to solve \eqref{E:PHI0CLOSEDEQUATION} (viewed as a wave equation for $\Phi_0$), 
we need to prescribe the data functions
$\Phi_0|_{\Sigma_0}$
and
$\partial_t \Phi_0|_{\Sigma_0}$,
whereas for the ODE-type blowup solutions we have constructed,
we can freely prescribe (in plane symmetry) only $\Phi_0|_{\Sigma_0}$;
the quantity $\partial_t \Phi_0|_{\Sigma_0}$ is not ``free,'' but rather
is uniquely determined from $\Phi_0|_{\Sigma_0}$
via the wave equation \eqref{E:WAVE}.
Put differently, 
the ODE-type blowup solutions that we have constructed yield
``special'' solutions to equation \eqref{E:PHI0CLOSEDEQUATION}
that are constrained by the fact that 
$\Phi_0$ is the time derivative of a solution
to the original wave equation \eqref{E:WAVE}.
In contrast, we note that we expect that the methods of \cite{jSgHjLwW2016} could be used 
to show that the plane symmetric shock-forming
solutions to \eqref{E:PHI0CLOSEDEQUATION} that we constructed above
are stable under perturbations that break the plane symmetry.

\subsection{Notation}
\label{SS:NOTATION}
In this subsection, we summarize some notation that we use throughout.

\begin{itemize}
	\item $\lbrace x^{\alpha} \rbrace_{\alpha=0,1,2,3}$ denotes
		the standard Cartesian coordinates on $\mathbb{R}^{1+3} = \mathbb{R} \times \mathbb{R}^3$
		and
			$
	\displaystyle
	\partial_{\alpha} 
	:=
	\frac{\partial}{\partial x^{\alpha}}
	$
	denotes the corresponding coordinate partial derivative vectorfields.
		$x^0 \in \mathbb{R}$ is the time coordinate and $\underline{x} := (x^1,x^2,x^3) \in \mathbb{R}^3$
		are the spatial coordinates.
	\item We often use the alternate notation $x^0 = t$ and $\partial_0 = \partial_t$.
	\item $\Sigma_t := \lbrace (t,\underline{x}) \ | \ \underline{x} \in \mathbb{R}^3 \rbrace$
		is the standard flat hypersurface of constant time.
	\item Greek ``spacetime'' indices such as $\alpha$ vary over $0,1,2,3$ and
		Latin ``spatial'' indices such as $a$ vary over $1,2,3$.
		We use primed indices, such as $a'$, in the same way that we use their non-primed counterparts.
		We use Einstein's summation convention in that repeated indices are summed
		over their respective ranges.
	\item We sometimes omit the arguments of functions appearing in pointwise inequalities. For example,
		we sometimes write $|f| \leq C \mathring{\upepsilon}$
		instead of $|f(t,\underline{x})| \leq C \mathring{\upepsilon}$.
	\item $\nabla^k \Psi$ denotes the array comprising all $k^{th}-$order
		derivatives of $\Psi$ with respect to the Cartesian spatial coordinate vector fields.
		We often use the alternate notation $\nabla \Psi$ in place of $\nabla^1 \Psi$.
		For example, $\nabla^1 \Psi = \nabla \Psi := (\partial_1 \Psi, \partial_2 \Psi, \partial_3 \Psi)$.
	\item 
	$|\nabla^{\leq k} \Psi| 
	:= \sum_{k'=0}^k |\nabla^{k'} \Psi|
	$.
	\item 
		$|\nabla^{[a,b]} \Psi| 
		:= \sum_{k'=a}^b |\nabla^{k'} \Psi|
		$.
	\item $H^N(\Sigma_t)$ denotes the standard Sobolev space of functions on $\Sigma_t$
	with corresponding norm 
	\[
	\displaystyle
	\| f \|_{H^N(\Sigma_t)}
	:= 
	\left\lbrace
		\sum_{a_1 + a_2 + a_3 \leq N}
		\int_{\underline{x} \in \mathbb{R}^3}
			|\partial_1^{a_1} \partial_2^{a_2} \partial_3^{a_3}f(t,\underline{x})|^2
		\, d \underline{x}
	\right\rbrace^{1/2}.
	\]
	In the case $N=0$, we use the notation ``$L^2$'' in place of ``$H^0$.''
	\item 
	$L^{\infty}(\Sigma_t)$ denotes the standard Lebesgue space of functions on $\Sigma_t$
	with corresponding norm 
	$
	\displaystyle
	\| f \|_{L^{\infty}(\Sigma_t)}
	:= 
	\mbox{\upshape ess sup}_{\underline{x} \in \mathbb{R}^3}
	|f(t,\underline{x})|
	$.
	\item Above and throughout, $d \underline{x} = d x^1 dx^2 dx^3$ is the standard Euclidean integration measure on $\Sigma_t$.
	\item If $A$ and $B$ are two quantities, then we often write 
		$A \lesssim B$
		to indicate that ``there exists a constant $C > 0$ such that $A \leq C B$.''
	\item We sometimes write $\mathcal{O}(B)$ to denote a quantity $A$ 
	with the following property: there exists a constant $C > 0$ such that $|A| \leq C |B|$.
	\item Explicit and implicit constants
	are allowed to depend on the data-size parameters
	$\mathring{A}$
	and 
	$\mathring{A}_*^{-1}$
	from Subsect.\ \ref{SS:DATAASSUMPTIONS},
	in a manner that we more fully explain in Subsect.\ \ref{SS:CONVENTIONSFORCONSTANTS}.
\end{itemize}

\section{Mathematical setup and the evolution equations}
\label{S:MATHEMATICALSETUP}
In this section, we state our assumptions on the nonlinearities,
define the quantities that we will study in the rest of the paper,
and derive evolution equations.

\subsection{Assumptions on the weight}
\label{SS:WEIGHTASSUMPTIONS}
Let $\weight$ be the scalar function from equation \eqref{E:WAVE}.
We assume that there are constants $C_k > 0$ such that
\begin{align}
	\weight(y)
	& > 0,
	&& y \in (-1/2,\infty),
		\label{E:WEIGHTISPOSITIVE} \\
	\weight(0)
	& = 1,
	&&	
		\label{E:WOFZEROISONE} \\
	\weight'(y)
	& \leq 0,
	&& y \in [0,\infty),
		\label{E:WEIGHTNEGATIVEDERIVATIVE} \\
	\left|
		\left\lbrace
			(1 + y)^2
			\frac{d}{dy}
		\right\rbrace^k
		\left[
			(1 + y) \weight(y)
		\right]
	\right|
	& \leq C_k,
	\qquad
	0 \leq k \leq 5,
	&& y \in (-1/2,\infty).
	\label{E:ESTIMATEFORDERIVATIVESOFWEIGHT}
\end{align}

We also assume that there is a constant $\upalpha > 0$ such that
\begin{align} \label{E:WEIGHTVSWEIGHTDERIVATIVECOMPARISON}
	\weight(y)
	& \leq \upalpha 
	\left|
		\weight'(y)
	\right|^{1/2},
	&& y \in [1,\infty).
\end{align}

Note that 
\eqref{E:WEIGHTISPOSITIVE},
\eqref{E:WEIGHTNEGATIVEDERIVATIVE},
and \eqref{E:WEIGHTVSWEIGHTDERIVATIVECOMPARISON} imply in particular that
\begin{align} \label{E:WEIGHTPRIMEISNEGATIVE}
\weight'(y)
	& < 0,
	&& y \in [1,\infty).
\end{align}

\subsection{The integrating factor and the renormalized solution variables}
\label{SS:INTEGRATINGFACTOR}

\subsubsection{Definitions}
\label{SSS:DEFINITIONSOFINTEGRATINGFACTORANDRENORMALIZEDVARIABLES}
As we described in Subsubsect.\ \ref{SSS:DISCUSSIONOFTHEPROOF},
our analysis fundamentally relies on the following integrating factor.

\begin{definition}[\textbf{The integrating factor}]
\label{D:INTEGRATINGFACTOR}
Let $\Phi$ be the solution to the wave equation \eqref{E:WAVE}.
We define $\Ifact = \Ifact(t,\underline{x})$
to be the solution to the following transport equation:
\begin{align} \label{E:INTEGRATINFACTORODEANDIC}
	\partial_t \Ifact
	& = 
	- \Ifact \partial_t \Phi,
	&& \Ifact|_{\Sigma_0} = 1.
\end{align}

Moreover, we define
\begin{align} \label{E:IFACTMIN}
	\Ifact_{\star}(t) 
	& := \min_{\Sigma_t} \Ifact.
\end{align}
\end{definition}

\begin{remark}[\textbf{The vanishing of} $\Ifact$ \textbf{implies singularity formation}]
	\label{R:VANISHINGOFINTEGRATINGFACTORISSINGULARITY}
	It is straightforward to see from \eqref{E:INTEGRATINFACTORODEANDIC}
	that if $\Ifact(T,\underline{x}) = 0$ for some $T > 0$ and for one or more $\underline{x} \in \mathbb{R}^3$, then
	at such values of $\underline{x}$, we have
	$\lim_{t \uparrow T} \sup_{s \in [0,t)} \partial_t \Phi(s,\underline{x}) = \infty$.
	In fact, it follows that $\int_{s=0}^t |\partial_t \Phi(s,\underline{x})| \, ds = \infty$.
\end{remark}

Most of our effort will go towards analyzing the following ``renormalized''
solution variables. We will show that they remain regular up to the singularity.

\begin{definition}[\textbf{Renormalized solution variables}]
\label{D:RENORMALIZEDSOLUTION}
Let $\Phi$ be the solution to the wave equation \eqref{E:WAVE}
and let $\Ifact$ be as in Def.~\ref{D:INTEGRATINGFACTOR}.
For $\alpha = 0,1,2,3$, we define
\begin{align} \label{E:PSIDEF}
	\Psi_{\alpha}
	& := \Ifact \partial_{\alpha} \Phi.
\end{align}
\end{definition}

\subsubsection{A crucial identity for $\Ifact$ and the $\Ifact$-weighted evolution equations}
\label{SSS:EVOLUTIONEQUATIONS}
Our main goal in this subsubsection is to derive evolution equations
for the renormalized solution variables; see Prop.~\ref{P:RENORMALIZEDEVOLUTOINEQUATIONS}.
As a preliminary step, we first provide a lemma
that shows that
$\partial_i \Ifact$ can be controlled in terms $\Psi_i$ and the initial data,
and that no singular factors of $\Ifact^{-1}$ appear in the relationship.
Though simple, the lemma is crucial for the top-order regularity theory of $\Ifact$.

\begin{lemma}[\textbf{Identity for the spatial derivatives of the integrating factor}]
The following identity holds for $i=1,2,3$:
\begin{align} \label{E:FORMULAFORDERIVATIVESOFINTEGRATINGFACTOR}
	\partial_i \Ifact
	& = - \Psi_i 
		+ 
		\Ifact \mathring{\Psi}_i.
\end{align}
\end{lemma}

\begin{proof}	
	Dividing equation \eqref{E:INTEGRATINFACTORODEANDIC}
	by $\Ifact$ and then applying $\partial_i$,
	we compute that
	\begin{align} \label{E:CRUCIALEVOLUTIONEQUATIONFORINTEGRATINGFACTORSPATIALDERIVATIVES}
	\partial_t
	\left\lbrace
		\frac{\partial_i \Ifact + \Psi_i}{\Ifact}
	\right\rbrace
	& = 0.
\end{align}
Integrating \eqref{E:CRUCIALEVOLUTIONEQUATIONFORINTEGRATINGFACTORSPATIALDERIVATIVES} with respect
to time and using the initial conditions $\Ifact|_{\Sigma_0} = 1$ and $\Psi_i|_{\Sigma_0} = \mathring{\Psi}_i$,
we arrive at \eqref{E:FORMULAFORDERIVATIVESOFINTEGRATINGFACTOR}.

\end{proof}

We now derive the main evolution equations that we will study in the remainder of the paper.

\begin{proposition}[\textbf{$\Ifact$-weighted evolution equations}]
	\label{P:RENORMALIZEDEVOLUTOINEQUATIONS}
	For solutions to the wave equation \eqref{E:WAVE},
	the renormalized solution variables of Def.~\ref{D:RENORMALIZEDSOLUTION}
	verify the following system:
	\begin{subequations}
	\begin{align} 
	\partial_t \Psi_0
	& = 
		\weight(\Ifact^{-1} \Psi_0) \sum_{a=1}^3 \partial_a \Psi_a
		+
		\Ifact^{-1} \weight(\Ifact^{-1} \Psi_0) \sum_{a=1}^3 \Psi_a \Psi_a
		-
		\weight(\Ifact^{-1} \Psi_0) \sum_{a=1}^3 \mathring{\Psi}_a \Psi_a,
		\label{E:PARTALTPSI0EVOLUTION} \\
	\partial_t \Psi_i
	& 
	= \partial_i \Psi_0
	- \mathring{\Psi}_i \Psi_0.
	\label{E:PARTALTPSIIEVOLUTION}
	\end{align}
	\end{subequations}
	
\end{proposition}

\begin{proof}
We first prove \eqref{E:PARTALTPSI0EVOLUTION}.
From equations \eqref{E:WAVE} and \eqref{E:INTEGRATINFACTORODEANDIC},
we deduce 
$\partial_t(\Ifact \partial_t \Phi) 
= \Ifact \weight(\partial_t \Phi) \Delta \Phi
= \weight(\partial_t \Phi) \sum_{a=1}^3 \partial_a (\Ifact \partial_a \Phi)
	- 
	\weight(\partial_t \Phi) \sum_{a=1}^3 (\partial_a \Ifact) \partial_a \Phi
$.
Using equation \eqref{E:FORMULAFORDERIVATIVESOFINTEGRATINGFACTOR} to substitute for 
$\partial_a \Ifact$ and appealing to Def.~\eqref{D:RENORMALIZEDSOLUTION},
we arrive at the desired equation \eqref{E:PARTALTPSI0EVOLUTION}.

To prove \eqref{E:PARTALTPSIIEVOLUTION}, we first use
Def.~\eqref{D:RENORMALIZEDSOLUTION} and the symmetry property $\partial_t \partial_i \Phi =  \partial_i \partial_t \Phi$ 
to obtain
$\partial_t \Psi_i = (\partial_t \ln \Ifact) \Psi_i + \partial_i \Psi_0 - (\partial_i \ln \Ifact) \Psi_0$.
Using \eqref{E:INTEGRATINFACTORODEANDIC} to replace 
$\partial_t \ln \Ifact$ with $- \Ifact^{-1} \Psi_0$
and equation \eqref{E:FORMULAFORDERIVATIVESOFINTEGRATINGFACTOR}
to replace
$
- \partial_i \ln \Ifact
$
with 
$\Ifact^{-1} \Psi_i - \mathring{\Psi}_i$,
we conclude \eqref{E:PARTALTPSIIEVOLUTION}.


\end{proof}

\section{Assumptions on the initial data and bootstrap assumptions}
\label{S:DATAANDBOOTSTRAP}
In this section, we state our size assumptions on the data
$
	(\partial_t \Phi|_{\Sigma_0},\partial_1 \Phi|_{\Sigma_0}, \partial_2 \Phi|_{\Sigma_0}, \partial_3 \Phi|_{\Sigma_0}) 
	= (\mathring{\Psi}_0,\mathring{\Psi}_1,\mathring{\Psi}_2,\mathring{\Psi}_3),
$
for the wave equation \eqref{E:WAVE}
and formulate bootstrap assumptions that are convenient
for studying the solution. We also precisely describe the smallness assumptions
that we need to close our estimates and show the existence of initial data
that verify the smallness assumptions.

\subsection{Assumptions on the data}
\label{SS:DATAASSUMPTIONS}
We assume that the initial data 
are compactly supported 
and verify the following size assumptions
for $i=1,2,3$:
\begin{subequations}
\begin{align} \label{E:DATASIZE}
	&
	\| \nabla^{\leq 2} \mathring{\Psi}_i \|_{L^{\infty}(\Sigma_0)}
	+
	\| \nabla^{[1,3]} \mathring{\Psi}_0 \|_{L^{\infty}(\Sigma_0)}
		\\
& \ \
	+
	\| \mathring{\Psi}_i \|_{H^5(\Sigma_0)}
	+
	\mathring{\upepsilon}^{3/2} \| \nabla \mathring{\Psi}_0 \|_{L^2(\Sigma_0)}
	+
	\| \nabla^2 \mathring{\Psi}_0 \|_{H^3(\Sigma_0)}
	\leq \mathring{\upepsilon},
	\notag
\end{align}
\begin{align}
\| \mathring{\Psi}_0 \|_{L^{\infty}(\Sigma_0)}
& \leq \mathring{A},
\label{E:LARGEDATASIZE}
\end{align}
\begin{align}
 - \frac{1}{4} 
  & \leq \min_{\Sigma_0} \mathring{\Psi}_0,
	\label{E:PSI0NOTTOONEGATIVE}
\end{align}	
\end{subequations}
where $\mathring{\upepsilon} > 0$ 
and $\mathring{A} > 0$
are two data-size parameters that 
we will discuss below
(roughly, $\mathring{\upepsilon}$ will have to be small for our proofs to close).
Roughly, in our analysis, we will propagate the above
size assumptions during the solution's classical lifespan.
A possible exception can occur for the top-order spatial derivatives of $\Psi_i$, 
which we are are not able to control uniformly in the norm $\| \cdot \|_{L^2(\Sigma_t)}$
due to the presence of the weight $\weight$ in our energy,
which can go to $0$ as the singularity forms
(see Def.~\ref{D:L2CONTROLLINGQUANTITY}).

We now introduce the 
crucial parameter $\mathring{A}_*$ that 
controls the time of first blowup;
our analysis shows that for
$\mathring{\upepsilon}$ sufficiently small, the time
of first blowup is 
$\left\lbrace
	1 + \mathcal{O}(\mathring{\upepsilon})
\right\rbrace
\mathring{A}_*^{-1}$;
see also Remark~\ref{R:CRUCIALDELTAPARAMETER}.

\begin{definition}[\textbf{The parameter that controls the time of first blowup}]
	\label{D:CRUCIALDATASIZEPARAMETER}
	We define the data-dependent parameter $\mathring{A}_*$ as follows:
	\begin{align} \label{E:CRUCIALDATASIZEPARAMETER}
		\mathring{A}_*
		& := 
		\max_{\Sigma_0} [\mathring{\Psi}_0]_+,
  \end{align}
	where $[\mathring{\Psi}_0]_+ := \max \lbrace \mathring{\Psi}_0,0 \rbrace$.
\end{definition}
Our main results concern solutions such that
$\mathring{A}_* > 0$, so we will assume in the
rest of the article that this is the case.

\begin{remark}[\textbf{The relevance of $\mathring{A}_*$}]
\label{R:CRUCIALDELTAPARAMETER}
The solutions that we study are such that\footnote{Here ``$A \sim B$'' imprecisely indicates
that $A$ is well-approximated by $B$.}
$
\partial_t \Ifact
= - \Psi_0
$
and
$
\partial_t \Psi_0 \sim 0
$
(throughout the evolution).
Hence, by the fundamental theorem of calculus,
we have
$\Psi_0(t,\underline{x}) \sim \mathring{\Psi}_0(\underline{x})$
and $\Ifact(t,\underline{x}) \sim 1 - t \mathring{\Psi}_0(\underline{x})$. 
From this last expression, we see that
$\Ifact$ is expected to vanish for the first time at approximately
$t = \mathring{A}_*^{-1}$ which, since $\partial_t \Ifact = - \Ifact \partial_t \Phi$,
implies the blowup of $\partial_t \Phi$ (see Remark~\ref{R:VANISHINGOFINTEGRATINGFACTORISSINGULARITY}).
See Lemmas~\ref{L:POINTWISEFORPSIANDPARTIALTPSI} 
and \ref{L:INTEGRATINGFACTORCRUCIALESTIMATES}
for the precise statements.
\end{remark}

\subsection{Bootstrap assumptions}
\label{SS:BOOTSTRAP}
To prove our main results, we find it convenient to rely on
a set of bootstrap assumptions, 
which we provide in this subsection.

\medskip

\noindent \underline{\textbf{The size of} $T_{(Boot)}$}.
We assume that $T_{(Boot)}$
is a bootstrap time with
\begin{align} \label{E:BOOTSTRAPTIME}
	0 < T_{(Boot)} \leq 2 \mathring{A}_*^{-1},
\end{align}
where $\mathring{A} > 0$ is the data-size parameter from Def.~\ref{D:CRUCIALDATASIZEPARAMETER}.
The assumption \eqref{E:BOOTSTRAPTIME} 
gives us a sufficient margin of error
to prove that finite-time blowup occurs
(see Remark~\ref{R:CRUCIALDELTAPARAMETER}).

\medskip

\noindent \underline{\textbf{Blowup has not yet occurred}}.
Recall that for the solutions under study, the vanishing of $\Ifact$
will coincide with the formation of a singularity in $\partial_t \Phi$.
For this reason,
we assume that for $t \in [0,T_{(Boot)})$,
we have
\begin{align} \label{E:HYPERBOLICBOOTSTRAP}
	\Ifact_{\star}(t) > 0,
\end{align}
where $\Ifact_{\star}$ is defined in \eqref{E:IFACTMIN}.

\medskip 
\underline{\textbf{The solution is contained in the regime of hyperbolicity}}.\footnote{In particular, the assumptions of
Subsect.\ \ref{SS:WEIGHTASSUMPTIONS} guarantee that $\weight(\Ifact^{-1} \Psi_0) > 0$
whenever \eqref{E:BOOTSTRAPRATIO} holds.}
We assume that for $(t,\underline{x}) \in [0,T_{(Boot)}) \times \mathbb{R}^3$, we have
\begin{align} \label{E:BOOTSTRAPRATIO}
	\frac{\Psi_0(t,\underline{x})}{\Ifact(t,\underline{x})}
		> - \frac{1}{2}.
\end{align}

\medskip
\noindent \underline{\textbf{Smallness of} $\Ifact$ \textbf{implies largeness of} $\Psi_0$}.
We assume that for $(t,\underline{x}) \in [0,T_{(Boot)}) \times \mathbb{R}^3$,
\begin{align} 	\label{E:BOOTSTRAPSMALLINTFACTIMPMLIESPSI0ISLARGE}
	\Ifact(t,\underline{x}) 
		\leq \frac{1}{8}
		\implies
		\Psi_0(t,\underline{x})
		\geq \frac{1}{8} \mathring{A}_*.
\end{align}

\medskip

\noindent \underline{$L^{\infty}$ \textbf{bootstrap assumptions}}.
We assume that for $t \in [0,T_{(Boot)})$, we have
\begin{subequations}
\begin{align} \label{E:PSI0ITSELFBOOTSTRAP}
	\| \Psi_0 \|_{L^{\infty}(\Sigma_t)}
	& \leq 
		\mathring{A}	
		+ 
		\varepsilon,
			\\
	\| \nabla^{[1,3]} \Psi_0 \|_{L^{\infty}(\Sigma_t)}
	& \leq \varepsilon,
			 \label{E:PSI0DERIVATVESBOOTSTRAP} \\
	\| \nabla^{\leq 2} \Psi_i \|_{L^{\infty}(\Sigma_t)}
	& \leq \varepsilon,
			\label{E:PSIIANDDERIVATIVESBOOTSTRAP} 
			\\
	\| \Ifact \|_{L^{\infty}(\Sigma_t)}
	& \leq 
		1
		+
		2 \mathring{A}_*^{-1} \mathring{A}	
		+ 
		\varepsilon,
			\label{E:IFACTITSELFBOOTSTRAP} 
\end{align}
\end{subequations}
where $\varepsilon > 0$ is a small bootstrap parameter;
we describe our smallness assumptions in the next subsection.

\begin{remark}[\textbf{The solution remains compactly supported in space}]
	\label{R:BOUNDEDWAVESPEED}
		From the bootstrap assumptions 
		and the assumptions of Subsect.\ \ref{SS:WEIGHTASSUMPTIONS} on $\weight$,
		we see that the wave speed
		$\left\lbrace
			\weight(\Ifact^{-1} \Psi_0)
		\right\rbrace^{1/2}$ 
		associated to equation~\eqref{E:WAVE}
		remains uniformly bounded for $(t,\underline{x}) \in [0,T_{(Boot)}) \times \mathbb{R}^3$.
		It follows that there exists a large, data-dependent ball $B \subset \mathbb{R}^3$
		such that $\Psi_{\alpha}(t,\underline{x})$ and $\Ifact - 1$
		vanish for $(t,\underline{x}) \in [0,T_{(Boot)}) \times B^c$.
\end{remark}

\subsection{Smallness assumptions}
\label{SS:SMALLNESSASSUMPTIONS}
For the rest of the article, 
when we say that ``$A$ is small relative to $B$,''
we mean that $B > 0$ and that there exists a continuous increasing function 
$f :(0,\infty) \rightarrow (0,\infty)$ 
such that 
$
\displaystyle
A < f(B)
$.
For brevity, we typically do not 
specify the form of $f$.

In the rest of the article, we make the following
relative smallness assumptions. We
continually adjust the required smallness
in order to close the estimates.
\begin{itemize}
	\item The bootstrap parameter $\varepsilon$ from Subsect.\ \ref{SS:BOOTSTRAP}
		is small relative to $1$ (i.e., in an absolute sense, without regard for the other parameters).
	\item $\varepsilon$ is small relative to $\mathring{A}^{-1}$,
		where $\mathring{A}$ is the data-size parameter 
		from \eqref{E:LARGEDATASIZE}.
	\item $\varepsilon$ is small relative to 
		the data-size parameter $\mathring{A}_*$ 
		from \eqref{E:CRUCIALDATASIZEPARAMETER}.
	\item We assume that
\begin{align} \label{E:DATAEPSILONVSBOOTSTRAPEPSILON}
	\varepsilon^{4/3}
	& \leq
	\mathring{\upepsilon}
	\leq \varepsilon,
\end{align}
where $\mathring{\upepsilon}$ is the data smallness parameter from \eqref{E:DATASIZE}.
\end{itemize}
The first two assumptions will allow us to control error terms that,
roughly speaking, are of size $\varepsilon \mathring{A}^k$ 
for some integer $k \geq 0$. The third assumption 
is relevant because the expected blowup-time is 
approximately $\mathring{A}_*^{-1}$
(see Remark~\ref{R:CRUCIALDELTAPARAMETER});
the assumption will allow us to show that various
error products featuring a small factor $\varepsilon$
remain small for $t \leq 2 \mathring{A}_*^{-1}$, 
which is plenty of time for us to show that $\Ifact$ vanishes
and $\partial_t \Phi$ blows up.
\eqref{E:DATAEPSILONVSBOOTSTRAPEPSILON} is convenient
for closing our bootstrap argument.

\subsection{Existence of initial data verifying the smallness assumptions}
\label{SS:EXISTENCEOFDATA}
It is easy to construct initial data
such that the parameters
$\mathring{\upepsilon}$,
$\mathring{A}$,
and 
$\mathring{A}_*$
satisfy the size assumptions stated in Subsect.\ \ref{SS:SMALLNESSASSUMPTIONS}.
For example, we can start with \emph{any} smooth compactly supported data 
$(\mathring{\Psi}_0,\mathring{\Psi}_1,\mathring{\Psi}_2,\mathring{\Psi}_3)$
such that $\max_{\Sigma_0} \mathring{\Psi}_0 > 0$
and $-\frac{1}{4} \leq \min_{\Sigma_0} \mathring{\Psi}_0$.
We then consider the one-parameter family (for $i=1,2,3$)
\[
\left(
	\leftexp{(\uplambda)}{\mathring{\Psi}_0}(\underline{x}),
	\leftexp{(\uplambda)}{\mathring{\Psi}_i}(\underline{x})
\right)
:=
\left(
	\mathring{\Psi}_0(\uplambda^{-1} \underline{x}),
	\uplambda^{-1} \mathring{\Psi}_i(\underline{x})
\right).
\]
It is straightforward to check that for $\uplambda > 0$ sufficiently large,
all of the size assumptions of
Subsect.\ \ref{SS:SMALLNESSASSUMPTIONS} are satisfied by the rescaled data
(where, roughly speaking, the role of $\mathring{\upepsilon}$ is played by $\uplambda^{-1}$),
as is \eqref{E:PSI0NOTTOONEGATIVE}.
The proof relies on the simple scaling identities
$\nabla^k \leftexp{(\uplambda)}{\mathring{\Psi}_0}(\underline{x})
	= \uplambda^{-k} (\nabla^k \mathring{\Psi}_0)(\uplambda^{-1} \underline{x})
$,
$\nabla^k \leftexp{(\uplambda)}{\mathring{\Psi}_i}(\underline{x})
	= \uplambda^{-1} (\nabla^k \mathring{\Psi}_i)(\underline{x})
$,
$
\left\|
		\nabla^k \leftexp{(\uplambda)}{\mathring{\Psi}_0}
	\right\|_{L^2(\Sigma_0)}
= \uplambda^{3/2 - k}
			\| \mathring{\Psi}_0 \|_{L^2(\Sigma_0)}
$,
and
$
\left\|
		\nabla^k \leftexp{(\uplambda)}{\mathring{\Psi}_i}
	\right\|_{L^2(\Sigma_0)}
	= \uplambda^{-1} \| \mathring{\Psi}_i \|_{L^2(\Sigma_0)}
$.

\begin{remark}[\textbf{Blowup generically occurs for appropriately 
	rescaled non-trivial data}]
	\label{R:GENERICDEGENERACY}
	The discussion in Subsect.\ \ref{SS:EXISTENCEOFDATA}
	can easily be extended to show that
	if $\max_{\Sigma_0} \mathring{\Psi}_0 > 0$
	and
	$-\frac{1}{4} \leq \min_{\Sigma_0} \mathring{\Psi}_0$,
	then
	one \emph{always} generates data to which
	our results apply
	by considering the rescaled data
	$
	\left(
	\leftexp{(\uplambda)}{\mathring{\Psi}_0}(\underline{x}),
	\leftexp{(\uplambda)}{\mathring{\Psi}_i}(\underline{x})
	\right)
	$
	with $\uplambda$ sufficiently large.

\end{remark}

\section{Energy identities}
\label{S:ENERGY}
In this section, we define the energies that we use to control the solution in $L^2$
up to top order. We then derive energy identities.

\subsection{Definitions}
\label{SS:ENERGYDEFINITIONS}
The following energy functional serves as a building block for our energies. 

\begin{definition}[\textbf{Basic energy functional}]
	\label{D:BASICENERGY}
	To any array-valued function 
	$V = V(t,\underline{x}) := (V_0,V_1,V_2,V_3)$,
	we associated the following energy:
	\begin{align} \label{E:BASICENERGY}
	\mathbb{E}[V]
	& = \mathbb{E}[V](t)
		:= \int_{\Sigma_t}
					\left\lbrace
						V_0^2
						+
						\sum_{a=1}^3
						\weight(\Ifact^{-1} \Psi_0)
						V_a^2
					\right\rbrace
			 \, d \underline{x}.
	\end{align}
\end{definition}

We now define $\mathbb{Q}_{(\mathring{\upepsilon})}(t)$, which is the main $L^2$-type
quantity that we use to control the solution up to top order.

\begin{definition} [\textbf{The} $L^2$-\textbf{controlling quantity}]
	\label{D:L2CONTROLLINGQUANTITY}
	Let $\mathring{\upepsilon} > 0$ be the data-size parameter from Subsect.\ \ref{SS:DATAASSUMPTIONS}.
	We define the $L^2$-controlling quantity $\mathbb{Q}_{(\mathring{\upepsilon})}$ as follows:
	\begin{align} \label{E:ENERGYTOCONTROLSOLNS}
		\mathbb{Q}_{(\mathring{\upepsilon})}(t)
		& := 
		\sum_{k=2}^5
		\int_{\Sigma_t}
				\left\lbrace	
					|\nabla^{k} \Psi_0|^2
					+
					\sum_{a=1}^3
					\weight(\Ifact^{-1} \Psi_0)
					|\nabla^{k} \Psi_a|^2
				\right\rbrace
		\, d \underline{x}
			\\
	& \ \
		+
		\sum_{k=1}^4
		\int_{\Sigma_t}
			|\nabla^{k} \Psi_a|^2
		\, d \underline{x}
		+
		\mathring{\upepsilon}^3
		\int_{\Sigma_t}
			\left\lbrace
				|\nabla \Psi_0|^2
				+
				\sum_{a=1}^3
				|\Psi_a|^2
			\right\rbrace
		\, d \underline{x}.
		\notag
\end{align}
\end{definition}

\begin{remark}[\textbf{The} $\mathring{\upepsilon}$ \textbf{weight in the definition of} $\mathbb{Q}_{(\mathring{\upepsilon})}$]
	\label{R:EPSILONWEIGHTSINCONTROLLINGQUANTITY}
	Our main a priori energy estimate
	shows that $\mathbb{Q}_{(\mathring{\upepsilon})}(t) \lesssim \mathring{\upepsilon}^2$
	up to the singularity. 
	The small coefficient of $\mathring{\upepsilon}^3$
	in front of the last integral on RHS~\eqref{E:ENERGYTOCONTROLSOLNS}
	is needed to ensure the $\mathcal{O}(\mathring{\upepsilon}^2)$
	smallness of $\mathbb{Q}_{(\mathring{\upepsilon})}$.
  However, the small coefficient of $\mathring{\upepsilon}^3$
	implies that $\mathbb{Q}_{(\mathring{\upepsilon})}(t)$
	provides only weak $L^2$ control of
	$\nabla \Psi_0$
	and $\Psi_a$,
	i.e., their $L^2$ norms can be as large as
	$\mathcal{O}(\mathring{\upepsilon}^{-1/2})$.
	We clarify that the possible $\mathcal{O}(\mathring{\upepsilon}^{-1/2})$ size of $\nabla \Psi_0$
	is consistent with the construction of initial data described in
	Subsect.\ \ref{SS:EXISTENCEOFDATA}.
	Despite the possible $\mathcal{O}(\mathring{\upepsilon}^{-1/2})$ largeness,
	we will nonetheless be able to show, 
	through a separate argument, 
	the following crucial bounds:
	$\nabla \Psi_0$
	and $\Psi_a$
	are bounded in the norm $\| \cdot \|_{L^{\infty}(\Sigma_t)}$ by $\lesssim \mathring{\upepsilon}$,
	up to the singularity; see Prop.~\ref{P:APRIORIESTIMATES}.
\end{remark}

\subsection{Basic energy identity}
\label{SS:BASICENERGYIDENTITY}
We aim to derive an energy identity for
the controlling quantity $\mathbb{Q}_{(\mathring{\upepsilon})}$ 
defined in \eqref{E:ENERGYTOCONTROLSOLNS}.
As a preliminary step, in this subsection, 
we derive a standard energy identity for the building block energy from \eqref{E:BASICENERGY}.

\begin{lemma}[\textbf{Basic energy identity}]
	\label{L:BASICENERGYID}
	Let $\mathbb{E}[V](t)$ be the building block energy defined in \eqref{E:BASICENERGY}.
	Solutions $V := (V_0,V_1,V_2,V_3)$ to the inhomogeneous linear system
\begin{subequations}
	\begin{align} 
	\partial_t V_0
	& = 
		\sum_{a=1}^3
		\weight(\Ifact^{-1} \Psi_0) \partial_a V_a
		+ 
		F_0,
		 \label{E:LINEARPSI0EVOLUTION} \\
	\partial_t V_i
	& 
	= \partial_i V_0
	+ 
	F_i
	\label{E:LINEARPSIIEVOLUTION}
	\end{align}
	\end{subequations}
	verify the following energy identity:
	\begin{align} \label{E:BASICENERGYID}
		\frac{d}{dt}
		\mathbb{E}[V](t)
		& = 
		\sum_{a=1}^3
		\int_{\Sigma_t}
			(\Ifact^{-1} \Psi_0)^2 \weight'(\Ifact^{-1} \Psi_0)
			(V_a)^2
		\, d \underline{x}
				\\	
		& \ \
				+
				\sum_{a=1}^3
				\sum_{b=1}^3
				\int_{\Sigma_t}
					\Ifact^{-1} \weight'(\Ifact^{-1} \Psi_0)
					\weight(\Ifact^{-1} \Psi_0) 
					\partial_a \Psi_a
					(V_b)^2
				\, d \underline{x}
			\notag \\
	& \ \
		+ 	\sum_{a=1}^3
				\sum_{b=1}^3
				\int_{\Sigma_t}
				\Ifact^{-2}
				\weight'(\Ifact^{-1} \Psi_0)
				\weight(\Ifact^{-1} \Psi_0) 
				(\Psi_a)^2
				(V_b)^2
			\, d \underline{x}
		\notag \\
	& \ \
			-
				\sum_{a=1}^3
				\sum_{b=1}^3
				\int_{\Sigma_t}
				\Ifact^{-1}
				\weight'(\Ifact^{-1} \Psi_0)
				\weight(\Ifact^{-1} \Psi_0) 
				\mathring{\Psi}_a \Psi_a
				(V_b)^2
			\, d \underline{x}
		\notag \\
	& \ \
		-
		2
		\sum_{a=1}^3
		\int_{\Sigma_t}
			\Ifact^{-1} \weight'(\Ifact^{-1} \Psi_0)
			(\partial_a \Psi_0) V_a V_0
			 \, d \underline{x}
			\notag
			\\
	& \ \
		-
		2
		\sum_{a=1}^3
		\int_{\Sigma_t}
			\Ifact^{-2} \Psi_0
			\weight'(\Ifact^{-1} \Psi_0)
			\Psi_a V_a V_0
			 \, d \underline{x}
			\notag
			\\
	& \ \
		+
		2
		\sum_{a=1}^3
		\int_{\Sigma_t}
			\Ifact^{-1} \Psi_0
			\weight'(\Ifact^{-1} \Psi_0)
			\mathring{\Psi}_a V_a V_0
			 \, d \underline{x}
			\notag
			\\
	& \ \
			+
			2
			\int_{\Sigma_t}
					V_0 F_0
			\, d \underline{x}
			+
			2
			\sum_{a=1}^3
			\int_{\Sigma_s}
				\weight(\Ifact^{-1} \Psi_0) V_a F_a
			\, d \underline{x}.
			\notag
	\end{align}
	\end{lemma}
	
	\begin{proof}
		First, using equations \eqref{E:INTEGRATINFACTORODEANDIC} and \eqref{E:PARTALTPSI0EVOLUTION},
		we compute that
		\begin{align} \label{E:TIMEDERIVATIVEOFWEIGHT}
		\partial_t 
		\left\lbrace 
			\weight(\Ifact^{-1} \Psi_0)
		\right\rbrace
		& 
		=
		\Ifact^{-1} \weight'(\Ifact^{-1} \Psi_0) (\partial_t \Psi_0)
		+
		(\Ifact^{-1} \Psi_0)^2 \weight'(\Ifact^{-1} \Psi_0)
		\\
		& = 	
		\sum_{a=1}^3
		\Ifact^{-1} \weight'(\Ifact^{-1} \Psi_0)
		\weight(\Ifact^{-1} \Psi_0) 
		(\partial_a \Psi_a)
			\notag \\
		& \ \
		+
		\sum_{a=1}^3
		\Ifact^{-2}
		\weight'(\Ifact^{-1} \Psi_0)
		\weight(\Ifact^{-1} \Psi_0) 
		(\Psi_a)^2
			\notag \\
	& \ \
		-
		\sum_{a=1}^3
		\Ifact^{-1}
		\weight'(\Ifact^{-1} \Psi_0)
		\weight(\Ifact^{-1} \Psi_0)
		\mathring{\Psi}_a \Psi_a
		\notag \\
	& \ \
		+
		(\Ifact^{-1} \Psi_0)^2 
		\weight'(\Ifact^{-1} \Psi_0).
		\notag
		\end{align}
		Taking the time derivative of \eqref{E:BASICENERGY}, using \eqref{E:TIMEDERIVATIVEOFWEIGHT},  
		and using \eqref{E:LINEARPSI0EVOLUTION}-\eqref{E:LINEARPSIIEVOLUTION}
		for substitution, we obtain
		\begin{align} \label{E:TIMEDERIVATIVEOFBASICENERGY}
			\frac{d}{dt} \mathbb{E}[V](t)
			& = 
				2
				\sum_{a=1}^3
				\int_{\Sigma_t}
					\left\lbrace
						\weight(\Ifact^{-1} \Psi_0) 
						V_0 \partial_a V_a
						+ 
						\weight(\Ifact^{-1} \Psi_0)
						V_a \partial_a V_0
					\right\rbrace
				\, d \underline{x}
					\\
		& \ \
				+
				2
				\int_{\Sigma_t}
					\left\lbrace
						V_0 F_0
						+
						\sum_{a=1}^3
						\weight(\Ifact^{-1} \Psi_0)
						V_a F_a
					\right\rbrace
				 \, d \underline{x}
				\notag	\\
		& \ \
				+
			\sum_{a=1}^3 
			\sum_{b=1}^3
			\int_{\Sigma_t}
				\Ifact^{-1} \weight'(\Ifact^{-1} \Psi_0)
					\weight(\Ifact^{-1} \Psi_0) 
					(\partial_a \Psi_a)
					(V_b)^2
			 \, d \underline{x}
				\notag
					\\
		& \ \
				+
			\sum_{a=1}^3 
			\sum_{b=1}^3
			\int_{\Sigma_t}
					\Ifact^{-2}
					\weight'(\Ifact^{-1} \Psi_0)
					\weight(\Ifact^{-1} \Psi_0) 
					(\Psi_a)^2
				(V_b)^2
			 \, d \underline{x}
				\notag
					\\
		& \ \
			-
			\sum_{a=1}^3 
			\sum_{b=1}^3
			\int_{\Sigma_t}
					\Ifact^{-1}
					\weight'(\Ifact^{-1} \Psi_0)
					\weight(\Ifact^{-1} \Psi_0)
					\mathring{\Psi}_a \Psi_a
					(V_b)^2
			 \, d \underline{x}
				\notag
					\\
		& \ \
				+
			\sum_{a=1}^3 
			\int_{\Sigma_t}
				(\Ifact^{-1} \Psi_0)^2 \weight'(\Ifact^{-1} \Psi_0) (V_a)^2
			\, d \underline{x}.
			\notag
		\end{align}
		Integrating by parts in the first integral on RHS~\eqref{E:TIMEDERIVATIVEOFBASICENERGY}
		and using the identity \eqref{E:FORMULAFORDERIVATIVESOFINTEGRATINGFACTOR},
		we obtain
		\begin{align} \label{E:IBPFIRSTINTEGRALINTIMEDERIVATIVEOFBASICENERGY}
			&
			2
				\sum_{a=1}^3
				\int_{\Sigma_t}
					\left\lbrace
						\weight(\Ifact^{-1} \Psi_0) 
						V_0 \partial_a V_a
						+ 
						\weight(\Ifact^{-1} \Psi_0)
						V_a \partial_a V_0
					\right\rbrace
				\, d \underline{x}
					\\
			& = 
				- 2
				\sum_{a=1}^3
				\int_{\Sigma_t}
					\Ifact^{-1} \weight'(\Ifact^{-1} \Psi_0)
					(\partial_a \Psi_0) V_a V_0
				\, d \underline{x}
					\notag \\
			& \ \ 
				-
				2
				\sum_{a=1}^3
				\int_{\Sigma_t}
					\Ifact^{-2} \Psi_0
					\weight'(\Ifact^{-1} \Psi_0)
					\Psi_a V_a V_0
				\, d \underline{x}
				+
				2
				\sum_{a=1}^3
				\int_{\Sigma_t}
					\Ifact^{-1} \Psi_0
			\weight'(\Ifact^{-1} \Psi_0)
			\mathring{\Psi}_a V_a V_0
			 \, d \underline{x}.
			\notag
		\end{align}
		Using \eqref{E:IBPFIRSTINTEGRALINTIMEDERIVATIVEOFBASICENERGY} to substitute 
		for the first integral on RHS~\eqref{E:TIMEDERIVATIVEOFBASICENERGY},
		we arrive at \eqref{E:BASICENERGYID}.
	\end{proof}

\subsection{Integral identity for the fundamental $L^2$-controlling quantity}	
We now derive an energy identity for
the controlling quantity $\mathbb{Q}_{(\mathring{\upepsilon})}$.

\begin{lemma}[\textbf{Integral identity for the $L^2$-controlling quantity}]
	\label{L:INTEGRALIDENTITYFORENERGY}
	Consider the following inhomogeneous system,
	obtained by commuting \eqref{E:PARTALTPSI0EVOLUTION}-\eqref{E:PARTALTPSIIEVOLUTION} 
	with $\nabla^k$:
	\begin{subequations}
	\begin{align} 
	\partial_t \nabla^k \Psi_0
	& = 
		\weight(\Ifact^{-1} \Psi_0) \sum_{a=1}^3 \partial_a \nabla^k \Psi_a
		+
		F_0^{(k)},
		\label{E:ENERGYIDCOMMUTEDPARTALTPSI0EVOLUTION} \\
	\partial_t \nabla^k \Psi_i
	& 
	= \partial_i \nabla^k \Psi_0
		+ F_i^{(k)}.
	\label{E:ENERGYIDCOMMUTEDPARTALTPSIIEVOLUTION}
	\end{align}
	\end{subequations}
	For solutions,
	the $L^2$-controlling quantity $\mathbb{Q}_{(\mathring{\upepsilon})}$ of Def.~\ref{D:L2CONTROLLINGQUANTITY}
	satisfies the following integral identity:
	\begin{align} \label{E:INTEGRALIDENTITYFORENERGY}
	\mathbb{Q}_{(\mathring{\upepsilon})}(t)
	& =
		\mathbb{Q}_{(\mathring{\upepsilon})}(0)
		+
		\sum_{k=2}^5
		\sum_{a=1}^3
		\int_{s=0}^t
		\int_{\Sigma_s}
			(\Ifact^{-1} \Psi_0)^2 
			\weight'(\Ifact^{-1} \Psi_0)
				|\nabla^{k} \Psi_a|^2
		\, d \underline{x}	
		\, ds
		\\
	& \ \
		+
		\sum_{k=2}^5
		\sum_{a=1}^3
		\sum_{b=1}^3
		\int_{s=0}^t
		\int_{\Sigma_s}
					\Ifact^{-1} \weight'(\Ifact^{-1} \Psi_0)
					\weight(\Ifact^{-1} \Psi_0) 
					(\partial_a \Psi_a)
					|\nabla^{k} \Psi_b|^2
				\, d \underline{x}
		\, ds
			\notag \\
	& \ \
		+ \sum_{k=2}^5
			\sum_{a=1}^3
			\sum_{b=1}^3
			\int_{s=0}^t
			\int_{\Sigma_s}
				\Ifact^{-2} 
				\weight'(\Ifact^{-1} \Psi_0)
				\weight(\Ifact^{-1} \Psi_0) 
				 |\Psi_a|^2 
				|\nabla^{k} \Psi_b|^2
			\, d \underline{x}
			\, ds
		\notag \\
	& \ \
			-
			\sum_{k=2}^5
			\sum_{a=1}^3
			\sum_{b=1}^3
			\int_{s=0}^t
			\int_{\Sigma_s}
				\Ifact^{-1} 
				\weight'(\Ifact^{-1} \Psi_0)
				\weight(\Ifact^{-1} \Psi_0) 
				\mathring{\Psi}_a \Psi_a
				|\nabla^{k} \Psi_b|^2
			\, d \underline{x}
			\, ds
		\notag \\
	& \ \
		-
		2
		\sum_{k=2}^5
		\sum_{a=1}^3
		\int_{s=0}^t
		\int_{\Sigma_s}
			\Ifact^{-1} \weight'(\Ifact^{-1} \Psi_0)
			(\partial_a \Psi_0) \nabla^{k} \Psi_a \cdot \nabla^{k} \Psi_0
		\, d \underline{x}
		\, ds
			\notag
			\\
	& \ \
		-
		2
		\sum_{k=2}^5
		\sum_{a=1}^3
		\int_{s=0}^t
		\int_{\Sigma_s}
			\Ifact^{-2} \Psi_0
			\weight'(\Ifact^{-1} \Psi_0)
			\Psi_a \nabla^{k} \Psi_a \cdot \nabla^{k} \Psi_0
			\, d \underline{x}
		\, ds
			\notag \\
	& \ \
		+
		2
		\sum_{k=2}^5
		\sum_{a=1}^3
		\int_{s=0}^t
		\int_{\Sigma_s}
			\Ifact^{-1} \Psi_0
			\weight'(\Ifact^{-1} \Psi_0)
			\mathring{\Psi}_a \nabla^{k} \Psi_a \cdot \nabla^{k} \Psi_0
			\, d \underline{x}
		\, ds
			\notag
			\\
	& \ \
			+
		2
		\sum_{a=1}^3
		\sum_{k=1}^4
		\int_{s=0}^t
		\int_{\Sigma_s}
			\nabla^{k} \Psi_a \cdot \partial_a \nabla^{k} \Psi_0
		\, d \underline{x}
		\, ds
			\notag \\
	& \ \
			+
			2
			\sum_{k=2}^5
			\int_{s=0}^t
			\int_{\Sigma_s}
					\nabla^{k} \Psi_0 \cdot F_0^{(k)}
			\, d \underline{x}
			\, ds
				\notag \\
	& \ \
			+
			2
			\sum_{k=2}^5
			\sum_{a=1}^3
			\int_{s=0}^t
			\int_{\Sigma_s}
				\weight(\Ifact^{-1} \Psi_0) 
				\nabla^{k} \Psi_a \cdot F_a^{(k)}
			\, d \underline{x}
			\, ds
			\notag	\\
& \ \
		+
		2
		\sum_{k=1}^4
		\sum_{a=1}^3
		\int_{s=0}^t
		\int_{\Sigma_s}
			\nabla^{k} \Psi_a \cdot F_a^{(k)}
		\, d \underline{x}
		\, ds
			\notag \\
& \ \
		+
		2
		\mathring{\upepsilon}^3
		\sum_{a=1}^3
		\int_{s=0}^t
		\int_{\Sigma_s}
			\weight(\Ifact^{-1} \Psi_0) \nabla \Psi_0 \cdot \partial_a \nabla \Psi_a
		\, d \underline{x}
		\, ds
			\notag
			\\
& \ \
		+
		2
		\mathring{\upepsilon}^3
		\int_{s=0}^t
		\int_{\Sigma_s}
			\nabla \Psi_0 \cdot F_0^{(1)}
		\, d \underline{x}
		\, ds
		\notag \\
& \ \
		+
		2
		\mathring{\upepsilon}^3
		\sum_{a=1}^3
		\int_{s=0}^t
		\int_{\Sigma_s}
			\Psi_a \partial_a \Psi_0
		\, ds
		\notag \\
& \ \
		-
		2
		\mathring{\upepsilon}^3
		\sum_{a=1}^3
		\int_{s=0}^t
		\int_{\Sigma_s}
			\Psi_0 
			\Psi_a \mathring{\Psi}_a
			\, d \underline{x}
		\, ds.
		\notag
\end{align}

\end{lemma}

\begin{proof}
	We take the time derivative of both sides of \eqref{E:ENERGYTOCONTROLSOLNS}.
	The time derivative of the first line of RHS~\eqref{E:ENERGYTOCONTROLSOLNS}
	is given by \eqref{E:BASICENERGYID}, where the role of 
	$(V_0,V_1,V_2,V_3)$
	in \eqref{E:BASICENERGYID}
	is played by
	$(\nabla^{k} \Psi_0,\nabla^{k} \Psi_1,\nabla^{k} \Psi_2,\nabla^{k} \Psi_3)$
	and the role of the inhomogeneous terms
	$F_{\alpha}$ on RHS~\eqref{E:BASICENERGYID}
	is played by the terms $F_{\alpha}^{(k)}$
	from \eqref{E:ENERGYIDCOMMUTEDPARTALTPSI0EVOLUTION}-\eqref{E:ENERGYIDCOMMUTEDPARTALTPSIIEVOLUTION}.
	Moreover, with the help of 
	\eqref{E:PARTALTPSIIEVOLUTION}
	and
	\eqref{E:ENERGYIDCOMMUTEDPARTALTPSI0EVOLUTION}-\eqref{E:ENERGYIDCOMMUTEDPARTALTPSIIEVOLUTION},
	we compute that the time derivatives of the terms on the second line of RHS~\eqref{E:ENERGYTOCONTROLSOLNS}
	are equal to
	\begin{align} \label{E:ENERGYTIMEDERIVATIVESOFLOWERORDERTERMS}
		&
		2
		\sum_{k=1}^4
		\sum_{a=1}^3
		\int_{\Sigma_t}
			\nabla^{k} \Psi_a \cdot \partial_a \nabla^{k} \Psi_0
		\, d \underline{x}
		+
		2
		\sum_{k=1}^4
		\sum_{a=1}^3
		\int_{\Sigma_t}
			\nabla^{k} \Psi_a \cdot F_a^{(k)}
		\, d \underline{x}
			\\
& \ \
		+
		2 \mathring{\upepsilon}^3
		\sum_{a=1}^3
		\int_{\Sigma_t}
			\left\lbrace
			\weight(\Ifact^{-1} \Psi_0) \nabla \Psi_0 \cdot \partial_a \nabla \Psi_a
			+
			\nabla \Psi_0 \cdot F_0^{(1)}
			+
			\Psi_a \partial_a \Psi_0
			-
			\Psi_0 \Psi_a \mathring{\Psi}_a
			\right\rbrace
		\, d \underline{x}.
		\notag
	\end{align}
	Combining these calculations, we deduce
	\begin{align} \label{E:DIFFERENTIALFORMENERGYID}
		\frac{d}{dt}
		\mathbb{Q}_{(\mathring{\upepsilon})}(t)
		&
		=
		\sum_{k=2}^5
		\sum_{a=1}^3
		\int_{\Sigma_t}
			(\Ifact^{-1} \Psi_0)^2 \weight'(\Ifact^{-1} \Psi_0)
				|\nabla^{k} \Psi_a|^2
		\, d \underline{x}	
				\\
	& \ \ 
		+ 
		\sum_{k=2}^5
		\sum_{a=1}^3
			\sum_{b=1}^3
		\int_{\Sigma_t}
					\Ifact^{-1} \weight'(\Ifact^{-1} \Psi_0)
					\weight(\Ifact^{-1} \Psi_0) 
					 \partial_a \Psi_a
					|\nabla^{k} \Psi_b|^2
				\, d \underline{x}
			\notag \\
	& \ \
		+ \sum_{k=2}^5
			\sum_{a=1}^3
			\sum_{b=1}^3
			\int_{\Sigma_t}
				\Ifact^{-2} 
				\weight'(\Ifact^{-1} \Psi_0)
				\weight(\Ifact^{-1} \Psi_0) 
				|\Psi_a|^2 
				|\nabla^{k} \Psi_b|^2
			\, d \underline{x}
		\notag \\
	& \ \
			-
			\sum_{k=2}^5
			\sum_{a=1}^3
			\sum_{b=1}^3
			\int_{\Sigma_t}
				\Ifact^{-1}
				\weight'(\Ifact^{-1} \Psi_0)
				\weight(\Ifact^{-1} \Psi_0) 
				 \mathring{\Psi}_a \Psi_a
					|\nabla^{k} \Psi_b|^2
			\, d \underline{x}
		\notag \\
	& \ \
		-
		2
		\sum_{k=2}^5
		\sum_{a=1}^3
		\int_{\Sigma_t}
			\Ifact^{-1} \weight'(\Ifact^{-1} \Psi_0)
			(\partial_a \Psi_0) \nabla^{k} \Psi_a \cdot \nabla^{k} \Psi_0
			 \, d \underline{x}
			\notag
			\\
	& \ \
		-
		2
		\sum_{k=2}^5
		\sum_{a=1}^3
		\int_{\Sigma_t}
			\Ifact^{-2} \Psi_0
			\weight'(\Ifact^{-1} \Psi_0)
			\Psi_a \nabla^{k} \Psi_a \cdot \nabla^{k} \Psi_0
			 \, d \underline{x}
			\notag
			\\
	& \ \
		+
		2
		\sum_{k=2}^5
		\sum_{a=1}^3
		\int_{\Sigma_t}
			\Ifact^{-1} \Psi_0
			\weight'(\Ifact^{-1} \Psi_0)
			\mathring{\Psi}_a \nabla^{k} \Psi_a \cdot \nabla^{k} \Psi_0
			 \, d \underline{x}
			\notag
			\\
	& \ \
			+
			2
			\sum_{k=2}^5
			\int_{\Sigma_t}
					\nabla^{k} \Psi_0 \cdot F_0^{(k)}
			\, d \underline{x}
			+
			2
			\sum_{a=1}^3
			\int_{\Sigma_t}
				\weight(\Ifact^{-1} \Psi_0) \nabla^{k} \Psi_a \cdot F_a^{(k)}
			\, d \underline{x}
			+
			\mbox{\eqref{E:ENERGYTIMEDERIVATIVESOFLOWERORDERTERMS}}.
			\notag
	\end{align}
	Integrating \eqref{E:DIFFERENTIALFORMENERGYID}
	from time $0$ to time $t$,
	we arrive at the desired identity \eqref{E:INTEGRALIDENTITYFORENERGY}.
\end{proof}

\section{A priori estimates}
\label{S:ESTIMATES}
In this section, we use the data-size and bootstrap assumptions of Sect.\ \ref{S:DATAANDBOOTSTRAP}
and the energy identities of Sect.\ \ref{S:ENERGY} to derive a priori estimates.

\subsection{Conventions for constants}
	\label{SS:CONVENTIONSFORCONSTANTS}
	In our estimates, the explicit constants $C > 0$ and $c > 0$ are free to vary from line to line.
	\textbf{These explicit constants, and implicit ones as well, are allowed to depend on the data-size parameters
	$\mathring{A}$
	and 
	$\mathring{A}_*^{-1}$
	from Subsect.\ \ref{SS:DATAASSUMPTIONS}}.
	However, the constants can be chosen to be 
	independent of the parameters $\mathring{\upepsilon}$ 
	and $\varepsilon$ whenever $\mathring{\upepsilon}$ and $\varepsilon$
	are sufficiently small relative to 
	$\mathring{A}^{-1}$
	and $\mathring{A}_*$
	in the sense described in Subsect.\ \ref{SS:SMALLNESSASSUMPTIONS}.
	For example, under our conventions, we have that $\mathring{A}_*^{-2} \varepsilon = \mathcal{O}(\varepsilon)$.

\subsection{Pointwise estimates tied to the integrating factor}
\label{SS:POINTWISEESTIMATES}
In this subsection, we derive pointwise estimates that are important for
analyzing $\Ifact$.

We start by deriving sharp estimates for $\Psi_0$.
The proof is based on separately considering regions where
$\Ifact$ is small and $\Ifact$ is large.
In Lemma~\ref{L:INTEGRATINGFACTORCRUCIALESTIMATES}, 
we will use these estimates 
to derive further information about the behavior of $\Psi_0$ in regions where 
$\Ifact$ is small (i.e., near the singularity),
which is crucial for closing the energy estimates.

\begin{lemma}[\textbf{Pointwise estimates for} $\Psi_0$]
	\label{L:POINTWISEFORPSIANDPARTIALTPSI}
	Under the data-size assumptions of Subsect.\ \ref{SS:DATAASSUMPTIONS},
	the bootstrap assumptions of Subsect.\ \ref{SS:BOOTSTRAP},
	and the smallness assumptions of Subsect.\ \ref{SS:SMALLNESSASSUMPTIONS},
	the following pointwise estimates hold for
	$(t,\underline{x}) \in [0, T_{(Boot)}) \times \mathbb{R}^3$:
	\begin{align}
		\Psi_0(t,\underline{x})
		& = 
			\mathring{\Psi}_0(\underline{x})
			+
			\mathcal{O}(\varepsilon),
			\label{E:PSI0WELLAPPROXIMATED} 
	\end{align}
	where $\mathring{\Psi}_0(\underline{x}) = \Psi_0(0,\underline{x})$.
	
	In addition,
	\begin{align} \label{E:PSI0BIGGERTHANMINUSONEHALF}
		-5/16 & \leq \min_{\Sigma_t} \Psi_0.
	\end{align}
\end{lemma}

\begin{proof}
	We first prove \eqref{E:PSI0WELLAPPROXIMATED}.
	We will show that
	$\left| \partial_t \Psi_0(t,\underline{x}) \right| \lesssim \varepsilon$.
	Then from this estimate and the fundamental theorem of calculus,
	we obtain the desired bound \eqref{E:PSI0WELLAPPROXIMATED}.
	
	It remains for us to prove the bound
	$\left| \partial_t \Psi_0(t,\underline{x}) \right| \lesssim \varepsilon$.
	We first consider points $(t,\underline{x})$ such that $\Ifact(t,\underline{x}) > 1/8$.
	Then all factors of $\Ifact^{-1}$ in the evolution equation \eqref{E:PARTALTPSI0EVOLUTION} can be bounded
	by $\lesssim 1$. For this reason,
	the desired bound follows as a straightforward consequence of
	equation \eqref{E:PARTALTPSI0EVOLUTION},
	the bootstrap assumptions, 
	the data-size assumptions \eqref{E:DATASIZE},
	and the assumptions of Subsect.\ \ref{SS:WEIGHTASSUMPTIONS} on $\weight$.
	
	We now prove the desired bound at points $(t,\underline{x})$ such that
	$0 < \Ifact(t,\underline{x}) \leq 1/8$.
	From the bootstrap assumption \eqref{E:BOOTSTRAPSMALLINTFACTIMPMLIESPSI0ISLARGE},
	we deduce that $1 \lesssim \Psi_0(t,\underline{x})$ at such points.
	From this bound, the bootstrap assumptions,
	the data-size assumptions \eqref{E:DATASIZE},
	and the assumptions of Subsect.\ \ref{SS:WEIGHTASSUMPTIONS} on $\weight$,
	we deduce the following bound for some factors on RHS~\eqref{E:PARTALTPSI0EVOLUTION}
	at the spacetime points under consideration:
	\[
	\left|\Ifact^{-1} \weight(\Ifact^{-1} \Psi_0)\right| 
	=
	\Psi_0^{-1}
	\left|(\Ifact^{-1} \Psi_0) \weight(\Ifact^{-1} \Psi_0)\right|
	\lesssim 
	\left|(\Ifact^{-1} \Psi_0) \weight(\Ifact^{-1} \Psi_0)\right|
	\lesssim 1.
	\]
	With the help of this bound,
	the desired estimate
	$\left| \partial_t \Psi_0(t,\underline{x}) \right| \lesssim \varepsilon$
	follows as a straightforward consequence of
	equation \eqref{E:PARTALTPSI0EVOLUTION},
	the bootstrap assumptions, 
	the data-size assumptions \eqref{E:DATASIZE},
	and the assumptions of Subsect.\ \ref{SS:WEIGHTASSUMPTIONS} on $\weight$.
	We have therefore proved \eqref{E:PSI0WELLAPPROXIMATED}.
	
	The bound \eqref{E:PSI0BIGGERTHANMINUSONEHALF} then follows
	from \eqref{E:PSI0NOTTOONEGATIVE}
	and
	\eqref{E:PSI0WELLAPPROXIMATED}.
\end{proof}

In the next lemma, we derive sharp estimates for $\Ifact$.
The estimates are important for closing the energy estimates up to the singularity
and for precisely tying the vanishing of $\Ifact$ to the blowup of $\partial_t \Phi$.

\begin{lemma}[\textbf{Crucial estimates for the integrating factor}]
	\label{L:INTEGRATINGFACTORCRUCIALESTIMATES}
	Under the data-size assumptions of Subsect.\ \ref{SS:DATAASSUMPTIONS},
	the bootstrap assumptions of Subsect.\ \ref{SS:BOOTSTRAP},
	and the smallness assumptions of Subsect.\ \ref{SS:SMALLNESSASSUMPTIONS},
	the following estimates hold for $(t,\underline{x}) \in [0,\Tboot) \times \mathbb{R}^3$:
	\begin{subequations}
	\begin{align}  \label{E:IFACTCRUCIALPOINTWISE}
		\Ifact(t,\underline{x})
		& = 1 - t \mathring{\Psi}_0(\underline{x}) + \mathcal{O}(\varepsilon),
			\\
		\Ifact_{\star}(t)
		& = 1 - t \mathring{A}_* + \mathcal{O}(\varepsilon),
		\label{E:IFACTSTARCRUCIALPOINTWISE}
	\end{align}
	\end{subequations}
	where $\mathring{\Psi}_0(\underline{x}) = \Psi_0(0,\underline{x})$
	and $\mathring{A}_* > 0$ is the data-size parameter from Def.~\ref{D:CRUCIALDATASIZEPARAMETER}.
	
	Moreover, the following implications hold for $(t,\underline{x}) \in [0,\Tboot) \times \mathbb{R}^3$:
	\begin{subequations}
	\begin{align} \label{E:SMALLINTFACTIMPMLIESRATIOISLARGE}
		\Ifact(t,\underline{x}) 
		\leq \frac{1}{4} \min \lbrace 1, \mathring{A}_* \rbrace
		\implies
		\frac{\Psi_0(t,\underline{x})}{\Ifact(t,\underline{x})}
		\geq 1,
			\\
		\Ifact(t,\underline{x}) 
		\leq \frac{1}{4}
		\implies
		\Psi_0(t,\underline{x})
		\geq \frac{1}{4} \mathring{A}_*.
		\label{E:SMALLINTFACTIMPMLIESPSI0ISLARGE}
	\end{align}
	\end{subequations}
	
	Finally, the following implications hold for $(t,\underline{x}) \in [0,\Tboot) \times \mathbb{R}^3$:
	\begin{align} \label{E:PSI0NEGATIVEIMPMLIESINHYPERBOLICREGIME}
		\Psi(t,\underline{x}) 
		& \leq 0
		\implies
		\Ifact(t,\underline{x}) 
		\geq 1 - \mathcal{O}(\varepsilon)
		&&
		\mbox{and }
		\Psi(t,\underline{x}) 
		\leq 0
		\implies
		\frac{\Psi_0(t,\underline{x})}{\Ifact(t,\underline{x})}
		\geq - \frac{3}{8}.
	\end{align}
	
\end{lemma}

\begin{remark}[\textbf{Improvement of a bootstrap assumption}]
	\label{R:IMPROVEMENTOFANUNUSUALBOOTSTRAPASSUMPTION}
	Note in particular that 
	the estimate \eqref{E:SMALLINTFACTIMPMLIESPSI0ISLARGE}
	provides a strict improvement of the bootstrap assumption
	\eqref{E:BOOTSTRAPSMALLINTFACTIMPMLIESPSI0ISLARGE}.
\end{remark}

\begin{remark}[\textbf{The significance of} \eqref{E:PSI0NEGATIVEIMPMLIESINHYPERBOLICREGIME}]
	\label{R:SOLUTIONREMAINSINSIDEREGIMEOFHYPERBOLICITITY}
	Note that
	\eqref{E:PSI0NEGATIVEIMPMLIESINHYPERBOLICREGIME}
	is a strict improvement of the bootstrap assumption
	\eqref{E:BOOTSTRAPRATIO} 
	and implies that $-3/8 \leq \partial_t \Phi(t,\underline{x})$ 
	for
	$(t,\underline{x}) \in [0,\Tboot) \times \mathbb{R}^3$.
	In view of the assumption \eqref{E:WEIGHTISPOSITIVE} for $\weight$,
	we conclude that the solution never escapes the set of state-space
	values for which the wave equation \eqref{E:WAVE} is hyperbolic.
	In the rest of article, we often silently use this fact.
\end{remark}

\begin{proof}
	From equation \eqref{E:INTEGRATINFACTORODEANDIC} and the estimate \eqref{E:PSI0WELLAPPROXIMATED},
	we deduce
	$\partial_t \Ifact(t,\underline{x}) = - \mathring{\Psi}_0(\underline{x}) + \mathcal{O}(\varepsilon)$.
	Integrating in time and using the initial condition \eqref{E:INTEGRATINFACTORODEANDIC}, we
	find that $\Ifact(t,\underline{x}) = 1 - t \mathring{\Psi}_0(\underline{x}) + \mathcal{O}(\varepsilon)$,
	which is \eqref{E:IFACTCRUCIALPOINTWISE}.
	
	\eqref{E:IFACTSTARCRUCIALPOINTWISE} follows a simple consequence of \eqref{E:IFACTCRUCIALPOINTWISE}
	and Def.~\ref{D:CRUCIALDATASIZEPARAMETER}.
	
	To prove \eqref{E:SMALLINTFACTIMPMLIESRATIOISLARGE}, we
	first consider the case $\mathring{A}_* \geq 1$.
	From
	\eqref{E:IFACTCRUCIALPOINTWISE} and
	\eqref{E:PSI0WELLAPPROXIMATED},
	we deduce that
	$\Ifact(t,\underline{x}) = 1 - t \Psi_0(t,\underline{x}) + \mathcal{O}(\varepsilon)$.
	It follows that if $\Ifact(t,\underline{x}) \leq 1/4$,
	then $t \Psi_0(t,\underline{x}) \geq 1/2$.
	Since $0 \leq t \leq 2 \mathring{A}_*^{-1} \leq 2$,
	we deduce that
	$
	\frac{\Psi_0(t,\underline{x})}{\Ifact(t,\underline{x})}
	\geq 1  
	$,
	which is the desired conclusion.
	Next, we consider the case $\mathring{A}_* < 1$.
	Using \eqref{E:IFACTCRUCIALPOINTWISE} and
	\eqref{E:PSI0WELLAPPROXIMATED},
	we deduce that
	$\Ifact(t,\underline{x}) = 1 - t \Psi_0(t,\underline{x}) + \mathcal{O}(\varepsilon)$.
	It follows that if $\Ifact(t,\underline{x}) \leq (1/4) \mathring{A}_*$,
	then $t \Psi_0(t,\underline{x}) \geq 1 - (1/2) \mathring{A}_*$.
	Since $0 \leq t \leq 2 \mathring{A}_*^{-1}$,
	we deduce that
	$
	\frac{\Psi_0(t,\underline{x})}{\Ifact(t,\underline{x})}
	\geq 2 \left\lbrace 1 - (1/2) \mathring{A}_* \right\rbrace
	= 2 - \mathring{A}_*
	$, 
	which, in view of our assumption $\mathring{A}_* < 1$,
	is
	$
	> 1
	$.
	This completes our proof of \eqref{E:SMALLINTFACTIMPMLIESRATIOISLARGE}.
	
	The implication \eqref{E:SMALLINTFACTIMPMLIESPSI0ISLARGE}
	can be proved using arguments similar to the ones that we used to prove
	\eqref{E:SMALLINTFACTIMPMLIESRATIOISLARGE}, and we therefore omit the details.
	
	Next, we note that when $\Psi_0(t,\underline{x}) \leq 0$,
	the estimate
	$\Ifact(t,\underline{x}) = 1 - t \Psi_0(t,\underline{x}) + \mathcal{O}(\varepsilon)$
	proved above implies that $\Ifact(t,\underline{x}) \geq 1 - \mathcal{O}(\varepsilon)$,
	which yields the first implication stated in
	\eqref{E:PSI0NEGATIVEIMPMLIESINHYPERBOLICREGIME}.
	To obtain the second implication stated in
	\eqref{E:PSI0NEGATIVEIMPMLIESINHYPERBOLICREGIME},
	we use the first implication and the estimate
	\eqref{E:PSI0BIGGERTHANMINUSONEHALF}.
\end{proof}

In the next lemma, we derive some simple pointwise estimates showing
the spatial derivatives of $\Ifact$ up to top order can be controlled in terms
of the spatial derivatives of $\lbrace \Psi_a \rbrace_{a=1,2,3}$.

\begin{lemma}[\textbf{Estimates for the derivatives of the integrating factor}]
	\label{L:POINTWISEESTIMATEFORDERIVATIVESOFIFACT}
	Under the data-size assumptions of Subsect.\ \ref{SS:DATAASSUMPTIONS},
	the bootstrap assumptions of Subsect.\ \ref{SS:BOOTSTRAP},
	and the smallness assumptions of Subsect.\ \ref{SS:SMALLNESSASSUMPTIONS},
	the following pointwise estimates hold
	for $(t,\underline{x}) \in [0,T_{(Boot)}) \times \mathbb{R}^3$:
	\begin{subequations}
	\begin{align} 
		\left|
			\nabla \Ifact
		\right|
		& \lesssim
			\sum_{a=1}^3
			\left|
				\Psi_a
			\right|
			+ 
			\sum_{a=1}^3
			\left|
			 \mathring{\Psi}_a
			\right|.
				\label{E:LOWESTLEVELPOINTWISEESTIMATEFORDERIVATIVESOFIFACT} 
	\end{align}
	
	Moreover, for $2 \leq k \leq 6$, the following estimate holds:
	\begin{align}
		\left|
			\nabla^k \Ifact
		\right|
		& \lesssim 
			\sum_{a=1}^3
			\left|
				\nabla^{[1,k-1]} \Psi_a
			\right|
			+ 
			\sum_{a=1}^3
			\left|
				\nabla^{[1,k-1]} \mathring{\Psi}_a
			\right|
			+ 
			\varepsilon
			\sum_{a=1}^3
			\left|
			 \mathring{\Psi}_a
			\right|.
			\label{E:POINTWISEESTIMATEFORDERIVATIVESOFIFACT}
	\end{align}
	\end{subequations}
	
	Finally, the following estimate holds for $t \in [0,T_{(Boot)})$:
	\begin{align} \label{E:LINFINITYFORDERIVATIVESOFIFACT}
		\left\|
			\nabla^{[1,3]} \Ifact
		\right\|_{L^{\infty}(\Sigma_t)}
		& \lesssim 
			\varepsilon.
	\end{align}
\end{lemma}

\begin{proof}
	The estimate \eqref{E:LOWESTLEVELPOINTWISEESTIMATEFORDERIVATIVESOFIFACT}
	is straightforward consequence
	of equation \eqref{E:FORMULAFORDERIVATIVESOFINTEGRATINGFACTOR}
	and the bootstrap assumptions.
	Similarly, the estimate \eqref{E:POINTWISEESTIMATEFORDERIVATIVESOFIFACT}
	is straightforward to derive via induction in $k$ with the help
	of equation \eqref{E:FORMULAFORDERIVATIVESOFINTEGRATINGFACTOR},
	the bootstrap assumptions, the data-size assumptions \eqref{E:DATASIZE},
	and \eqref{E:DATAEPSILONVSBOOTSTRAPEPSILON}.
	\eqref{E:LINFINITYFORDERIVATIVESOFIFACT} then follows from
	\eqref{E:LOWESTLEVELPOINTWISEESTIMATEFORDERIVATIVESOFIFACT}-\eqref{E:POINTWISEESTIMATEFORDERIVATIVESOFIFACT},
	the bootstrap assumptions, the data-size assumptions \eqref{E:DATASIZE},
	and \eqref{E:DATAEPSILONVSBOOTSTRAPEPSILON}.
\end{proof}

\subsection{Pointwise estimates involving the weight}
\label{SS:POINTWISEWEIGHT}
In the next lemma, we derive precise pointwise estimates
for quantities that involve the weight function
$\weight$. The detailed information is important for
closing the energy estimates 
and for showing that the spatial derivatives of $\weight = \weight(\partial_t \Phi)$ 
are controllable. Some of the analysis is delicate in that
$\partial_t \Phi$ and its derivatives 
are allowed to be arbitrarily large
(i.e., the estimates hold uniformly, arbitrarily close to the expected singularity).

\begin{lemma}[\textbf{Pointwise estimates involving the weight} $\weight$]
\label{L:ESTIMATESINVOLVINGWEIGHT}
	Let $\mathbf{1}_{\left\lbrace 0 < \Ifact \leq (1/4) \min \lbrace 1, \mathring{A}_* \rbrace \right\rbrace}$
	be the characteristic function of the spacetime subset
	$
	\lbrace
		(t,\underline{x})
		\ | \
		0 < \Ifact (t,\underline{x}) 
		\leq (1/4) \min \lbrace 1, \mathring{A}_* \rbrace
	\rbrace
	$.
	Under the data-size assumptions of Subsect.\ \ref{SS:DATAASSUMPTIONS},
	the bootstrap assumptions of Subsect.\ \ref{SS:BOOTSTRAP},
	and the smallness assumptions of Subsect.\ \ref{SS:SMALLNESSASSUMPTIONS},
	the following pointwise estimates hold
	for $(t,\underline{x}) \in [0,T_{(Boot)}) \times \mathbb{R}^3$:
\begin{subequations}
\begin{align}
	\weight(\Ifact^{-1} \Psi_0)
	& \lesssim 1,
		\label{E:WEIGHTBOUNDEDBYONE}
			\\
	\left|
		\nabla
		\left\lbrace
			\weight(\Ifact^{-1} \Psi_0)
		\right\rbrace
	\right|
	& \lesssim 
		\varepsilon
		\mathbf{1}_{\left\lbrace 0 < \Ifact \leq (1/4) \min \lbrace 1, \mathring{A}_* \rbrace \right\rbrace}
		\left\lbrace
			\Ifact^{-2}  
			\left|
				\weight'(\Ifact^{-1} \Psi_0)
			\right|
		\right\rbrace^{1/2}
		+
		\varepsilon
		\left\lbrace
			\weight(\Ifact^{-1} \Psi_0)
		\right\rbrace^{1/2}
		\label{E:WEIGHTEXACTLYONEDERIVATIVEPOINTWISE}	\\
	& \lesssim 
	\varepsilon.
		\label{E:WEIGHTEXACTLYONEDERIVATIVEPOINTWISESMALL} 
\end{align}
\end{subequations}

In addition, for $2 \leq k \leq 5$, the following estimates hold:
\begin{align}
	\left|
		\nabla^k
		\left\lbrace
			\weight(\Ifact^{-1} \Psi_0)
		\right\rbrace
	\right|
	& \lesssim 
		\left|
			\nabla^{[1,k]} \Psi_0
		\right|
		+
		\sum_{a=1}^3
		\left|
			\nabla^{\leq k-1} \Psi_a
		\right|
		+
		\sum_{a=1}^3
		\left|
			\nabla^{\leq k-1} \mathring{\Psi}_a
		\right|.
		\label{E:WEIGHTATLEASTONEDERIVATIVEPOINTWISE}	
	\end{align}

	Furthermore, the following estimates hold:
	\begin{subequations}
	\begin{align}
		\Ifact^{-1}
		\weight(\Ifact^{-1} \Psi_0)
	& \lesssim 
	\mathbf{1}_{\left\lbrace 0 < \Ifact \leq (1/4) \min \lbrace 1, \mathring{A}_* \rbrace \right\rbrace}
	\left\lbrace
	\Ifact^{-2} 
	\left|
		\weight'(\Ifact^{-1} \Psi_0)
	\right|
	\right\rbrace^{1/2}
	+
	\left\lbrace
		\weight(\Ifact^{-1} \Psi_0)
	\right\rbrace^{1/2}
	\label{E:SINGULARFACTORTIMESWEIGHTPOINTWISE}	
			\\
	& \lesssim 1.
		\label{E:SINGULARFACTORTIMESWEIGHTBOUNDEDBYONE}  
	\end{align}
	\end{subequations}
	
	Moreover, 
	for $1 \leq k \leq 5$,
	the following estimates hold:
	\begin{align}
	\left|
		\nabla^k
		\left\lbrace
			\Ifact^{-1}
			\weight(\Ifact^{-1} \Psi_0)
		\right\rbrace
	\right|
	& \lesssim 
		\left|
			\nabla^{[1,k]} \Psi_0
		\right|
		+
		\sum_{a=1}^3
		\left|
			\nabla^{\leq k-1} \Psi_a
		\right|
		+
		\sum_{a=1}^3
		\left|
			\nabla^{\leq k-1} \mathring{\Psi}_a
		\right|.
		\label{E:ATLEASTONEDERIVATIVESINGULARFACTORTIMESWEIGHTPOINTWISE}
\end{align}

Finally, for $P \in [0,2]$, the following estimates hold:
\begin{subequations}
\begin{align} 
		\left|
		\Ifact^{-2}
		\weight'(\Ifact^{-1} \Psi_0)
		+
		\mathbf{1}_{\left\lbrace 0 < \Ifact \leq (1/4) \min \lbrace 1, \mathring{A}_* \rbrace \right\rbrace}
		\Ifact^{-2}
			\left|
				\weight'(\Ifact^{-1} \Psi_0)
			\right|
		\right|
		&
		\lesssim
		\weight(\Ifact^{-1} \Psi_0),
		\label{E:POINTWISEESTIMATESIFACTMINUSTWOTIMESWEIGHTDERIVATIVE}	\\
	\left|
		\Ifact^{-P}
		\weight'(\Ifact^{-1} \Psi_0)
	\right|
	& \lesssim 1.
		\label{E:BOUNDEDBYONEPOINTWISEESTIMATESIFACTMINUSTWOTIMESWEIGHTDERIVATIVE}
\end{align}
\end{subequations}

\end{lemma}

\begin{proof}
	Throughout this proof, we denote
	$
	\displaystyle
	y = y(t,\underline{x}):= \frac{\Psi_0(t,\underline{x})}{\Ifact(t,\underline{x})}
	$.
	Also, we silently use the observations of Remark~\ref{R:SOLUTIONREMAINSINSIDEREGIMEOFHYPERBOLICITITY}.
	
	\medskip
	
	\noindent \textbf{Proof of \eqref{E:WEIGHTBOUNDEDBYONE}}:
	This bound is a trivial consequence of our assumption \eqref{E:ESTIMATEFORDERIVATIVESOFWEIGHT} on $\weight$.
	
	\medskip
	
	\noindent \textbf{Proof of \eqref{E:WEIGHTEXACTLYONEDERIVATIVEPOINTWISE} and \eqref{E:WEIGHTEXACTLYONEDERIVATIVEPOINTWISESMALL}}:
	We first prove \eqref{E:WEIGHTEXACTLYONEDERIVATIVEPOINTWISE} 
	at spacetime points $(t,\underline{x})$ such that 
	$\Ifact(t,\underline{x}) > (1/4) \min \lbrace 1, \mathring{A}_* \rbrace$.
	This is the easy case because $\Ifact^{-1} < 4 \max \lbrace 1, \mathring{A}^{-1} \rbrace \leq C$, 
	and we therefore do not have to concern ourselves
	with the possibility of small denominators. Specifically,
	using the identity \eqref{E:FORMULAFORDERIVATIVESOFINTEGRATINGFACTOR},
	the bootstrap assumptions,
	the data-size assumptions \eqref{E:DATASIZE},
	and the assumptions of Subsect.\ \ref{SS:WEIGHTASSUMPTIONS},
	we deduce that when $\Ifact(t,\underline{x}) > (1/4) \min \lbrace 1, \mathring{A}_* \rbrace$, 
	we have
	\begin{align} \label{E:EASYCASEWEIGHTEXACTLYONEDERIVATIVEPOINTWISE}	
	\left|
		\nabla
		\left\lbrace
			\weight(\Ifact^{-1} \Psi_0)
		\right\rbrace
	\right|
	& \lesssim 
		\left|
			\nabla \Psi_0
		\right|
		+
		\sum_{a=1}^3
		\left|
			\Psi_a
		\right|
		+
		\sum_{a=1}^3
		\left|
			\mathring{\Psi}_a
		\right|
	\lesssim \varepsilon.
\end{align}
Next, we use the bootstrap assumptions
and the assumptions of Subsect.\ \ref{SS:WEIGHTASSUMPTIONS} on $\weight$
(specifically the uniform positivity of $\weight(y)$ for $y \in [-3/8,C]$)
to obtain
\[
\mathbf{1}_{\left\lbrace \Ifact > (1/4) \min \lbrace 1, \mathring{A}_* \rbrace \right\rbrace} 
\lesssim 
\weight(\Ifact^{-1} \Psi_0)
\lesssim 
\left\lbrace 
	\weight(\Ifact^{-1} \Psi_0)
\right\rbrace^{1/2}.
\]
It follows that 
$\mbox{RHS~\eqref{E:EASYCASEWEIGHTEXACTLYONEDERIVATIVEPOINTWISE}}$ 
is
$\lesssim$ the second term on RHS~\eqref{E:WEIGHTEXACTLYONEDERIVATIVEPOINTWISE}
as desired.

	We now prove \eqref{E:WEIGHTEXACTLYONEDERIVATIVEPOINTWISE} 
	at points $(t,\underline{x})$ such that $0 < \Ifact(t,\underline{x}) \leq (1/4) \min \lbrace 1, \mathring{A}_* \rbrace$.
	Along the way, we will prove some additional estimates that we will use later on.
	We start by defining the following weighted differential operator, which acts on
	functions $f = f(y)$:
	$
	D_Y f := y^2 \frac{d}{dy} f
	$.	
	Note that the chain rule implies that
	\begin{align}
			\nabla \weight(y)
			= - D_Y \weight(y) \nabla (y^{-1}).
	\end{align}
	We therefore inductively deduce that for $1 \leq k \leq 5$, we have
	\begin{align} \label{E:FIRSTPOINTWISEESTIMATECHAINRULEWITHWEIGHT}
		\left|
			\nabla^k \weight(y)
		\right|
		& \lesssim 
			\sum_{n=1}^k
			\left|
				D_Y^n \weight(y)
			\right|
			\left\lbrace
			\mathop{\mathop{\sum}_{\sum_{i=1}^n k_i = k}}_{k_i \geq 1}
			\prod_{i=1}^n
			\left|
				\nabla^{k_i} (y^{-1})
			\right|
			\right\rbrace.
	\end{align}
	The case $k=1$ in \eqref{E:FIRSTPOINTWISEESTIMATECHAINRULEWITHWEIGHT} yields
	$
	\left|
		\nabla \weight(\Ifact^{-1} \Psi_0)
	\right|
	\lesssim 
	(\Ifact^{-1} \Psi_0)^2
	\left|
		\weight'(\Ifact^{-1} \Psi_0)
	\right|
	\left|
		\nabla
		\left\lbrace
			\Ifact \Psi_0^{-1}
		\right\rbrace
	\right|
	$.
	Also using the identity \eqref{E:FORMULAFORDERIVATIVESOFINTEGRATINGFACTOR},
	the bootstrap assumptions,
	the data-size assumptions \eqref{E:DATASIZE},
	\eqref{E:DATAEPSILONVSBOOTSTRAPEPSILON},
	the assumptions of Subsect.\ \ref{SS:WEIGHTASSUMPTIONS},
	and the crucially important estimate \eqref{E:SMALLINTFACTIMPMLIESPSI0ISLARGE}
	(which implies that $\Psi_0^{-1} \lesssim 1$),
	we deduce that when $\Ifact(t,\underline{x}) \leq (1/4) \min \lbrace 1, \mathring{A}_* \rbrace$, 
	we have 
	$
	\left|
		\nabla
		\left\lbrace
			\Ifact \Psi_0^{-1}
		\right\rbrace
	\right|
	\lesssim \varepsilon
	$
	and thus
	\begin{align} \label{E:INEQUALITYCHAIN}
	\left|
		\nabla \weight(\Ifact^{-1} \Psi_0)
	\right|
	& \lesssim 
	\varepsilon
	(\Ifact^{-1} \Psi_0)^2
	\left|
		\weight'(\Ifact^{-1} \Psi_0)
	\right|
	\lesssim
	\varepsilon
	\left\lbrace
	(\Ifact^{-1} \Psi_0)^2
	\left|
		\weight'(\Ifact^{-1} \Psi_0)
	\right|
	\right\rbrace^{1/2}
		\\
	& \lesssim
	\varepsilon
	\left\lbrace
	\Ifact^{-2} 
	\left|
		\weight'(\Ifact^{-1} \Psi_0)
	\right|
	\right\rbrace^{1/2},
	\notag
	\end{align}
	which is $\lesssim$ the first term on RHS~\eqref{E:WEIGHTEXACTLYONEDERIVATIVEPOINTWISE} as desired.
	This finishes the proof of \eqref{E:WEIGHTEXACTLYONEDERIVATIVEPOINTWISE}.
	We clarify that to derive the next-to-last inequality in \eqref{E:INEQUALITYCHAIN}, 
	in which we bounded
	$
	(\Ifact^{-1} \Psi_0)^2
	\left|
		\weight'(\Ifact^{-1} \Psi_0)
	\right|
	$
	by its square root,
	we used \eqref{E:ESTIMATEFORDERIVATIVESOFWEIGHT} 
	to deduce
	$y^2 \left| \weight'(y) \right| \lesssim 1$.
	
	We now prove \eqref{E:WEIGHTEXACTLYONEDERIVATIVEPOINTWISESMALL}.
	From Remark~\ref{R:SOLUTIONREMAINSINSIDEREGIMEOFHYPERBOLICITITY},
	the assumptions of Subsect.\ \ref{SS:WEIGHTASSUMPTIONS} on $\weight$,
	and \eqref{E:SMALLINTFACTIMPMLIESPSI0ISLARGE},
	we deduce that 
	\begin{align} \label{E:MAINTERMWEIGHTEXACTLYONEDERIVATIVEPOINTWISESMALL}
	\mathbf{1}_{\left\lbrace 0 < \Ifact \leq (1/4) \min \lbrace 1, \mathring{A}_* \rbrace \right\rbrace}
		\left\lbrace
			\Ifact^{-2}  
			\left|
				\weight'(\Ifact^{-1} \Psi_0)
			\right|
		\right\rbrace^{1/2}
	& \lesssim
	\mathbf{1}_{\left\lbrace 0 < \Ifact \leq (1/4) \min \lbrace 1, \mathring{A}_* \rbrace \right\rbrace}
		\left\lbrace
			(\Ifact^{-2} \Psi_0^2)  
			\left|
				\weight'(\Ifact^{-1} \Psi_0)
			\right|
		\right\rbrace^{1/2}
			\\
	& \lesssim 1
	\notag
	\end{align}
	and that
	$
	\left\lbrace
			\weight(\Ifact^{-1} \Psi_0)
		\right\rbrace^{1/2}
	\lesssim 1
	$.
	That is, the non-$\varepsilon$ factors on RHS~\eqref{E:WEIGHTEXACTLYONEDERIVATIVEPOINTWISE}
	are $\lesssim 1$. This yields \eqref{E:WEIGHTEXACTLYONEDERIVATIVEPOINTWISESMALL}.
	
	\medskip
	
	\noindent \textbf{Proof of \eqref{E:WEIGHTATLEASTONEDERIVATIVEPOINTWISE}}:
	The proof is similar to that of \eqref{E:WEIGHTEXACTLYONEDERIVATIVEPOINTWISE} but slightly simpler.
	Note that $k \in [2,5]$ by assumption in this estimate.
	We first prove the estimate at points $(t,\underline{x})$ such that 
	$\Ifact(t,\underline{x}) > (1/4) \min \lbrace 1, \mathring{A}_* \rbrace$.
	This is the easy case because $\Ifact^{-1} < 4 \max \lbrace 1, \mathring{A}^{-1} \rbrace \leq C$, 
	and we therefore do not have to concern ourselves
	with the possibility of small denominators. Specifically,
	using the identity \eqref{E:FORMULAFORDERIVATIVESOFINTEGRATINGFACTOR},
	the bootstrap assumptions,
	the data-size assumptions \eqref{E:DATASIZE},
	\eqref{E:DATAEPSILONVSBOOTSTRAPEPSILON},
	and the assumptions of Subsect.\ \ref{SS:WEIGHTASSUMPTIONS},
	we deduce that when $\Ifact(t,\underline{x}) > (1/4) \min \lbrace 1, \mathring{A}_* \rbrace$, 
	we have
	\begin{align} \label{E:EASYCASEWEIGHTATLEASTONEDERIVATIVEPOINTWISE}	
	\left|
		\nabla^k
		\left\lbrace
			\weight(\Ifact^{-1} \Psi_0)
		\right\rbrace
	\right|
	& \lesssim 
		\left|
			\nabla^{[1,k]} \Psi_0
		\right|
		+
		\sum_{a=1}^3
		\left|
			\nabla^{\leq k-1} \Psi_a
		\right|
		+
		\sum_{a=1}^3
		\left|
			\nabla^{\leq k-1} \mathring{\Psi}_a
		\right|,
\end{align}
which is $\lesssim \mbox{RHS~\eqref{E:WEIGHTATLEASTONEDERIVATIVEPOINTWISE}}$ as desired.

It remains for us to prove \eqref{E:WEIGHTATLEASTONEDERIVATIVEPOINTWISE} at points $(t,\underline{x})$ 
such that $0 < \Ifact(t,\underline{x}) \leq (1/4) \min \lbrace 1, \mathring{A}_* \rbrace$.
Note that the estimate \eqref{E:FIRSTPOINTWISEESTIMATECHAINRULEWITHWEIGHT} holds
and that by \eqref{E:ESTIMATEFORDERIVATIVESOFWEIGHT} and Remark~\ref{R:SOLUTIONREMAINSINSIDEREGIMEOFHYPERBOLICITITY},
we have the following bound\footnote{In obtaining this bound, it is helpful to note that
$D_Y f = - \frac{d}{dz} f$, where $z := 1/y$.
\label{FN:DBIGYVSDDZ}} 
for the factors of $D_Y^n \weight(y)$ on RHS~\eqref{E:FIRSTPOINTWISEESTIMATECHAINRULEWITHWEIGHT}:
$
\left|	
	D_Y^n \weight(y)
\right|
\lesssim 1
$.
From this bound,
\eqref{E:FIRSTPOINTWISEESTIMATECHAINRULEWITHWEIGHT}, 
the bootstrap assumptions,
and the data-size assumptions \eqref{E:DATASIZE},
we see that the desired bound
\eqref{E:WEIGHTATLEASTONEDERIVATIVEPOINTWISE} will follow once we show that the following bound holds
when $2 \leq k \leq 5$ and $\Ifact(t,\underline{x}) \leq (1/4) \min \lbrace 1, \mathring{A}_* \rbrace$:
\begin{align} \label{E:POINTWISEBOUNDFORDERIVATIVESOFIFACTOVERPSI0}
		\left|
			\nabla^k (y^{-1})
		\right|
		& \lesssim 
			\left|
				\nabla^{[1,k]} \Psi_0
			\right|
			+
			\sum_{a=1}^3
			\left|
				\nabla^{\leq k-1} \Psi_a
			\right|
			+
			\sum_{a=1}^3
			\left|
				\nabla^{\leq k-1} \mathring{\Psi}_a
			\right|.
	\end{align}
	To prove \eqref{E:POINTWISEBOUNDFORDERIVATIVESOFIFACTOVERPSI0},
	we first note that \eqref{E:SMALLINTFACTIMPMLIESPSI0ISLARGE} implies that
	$1 \lesssim \Psi_0(t,\underline{x})$ in the present context. 
	Thus, 
	$	
		\mbox{LHS~\eqref{E:EASYCASEWEIGHTATLEASTONEDERIVATIVEPOINTWISE}}
		=
		\left|
			\nabla^k 
			\left(
				\frac{\Ifact}{\Psi_0}
			\right)
		\right|
	$
	is the $k^{th}$ derivative of a ratio with a denominator uniformly bounded from below away from $0$, 
	and the desired estimate \eqref{E:WEIGHTATLEASTONEDERIVATIVEPOINTWISE}	
	follows as a straightforward consequence of
	the identity \eqref{E:FORMULAFORDERIVATIVESOFINTEGRATINGFACTOR},
	the data-size assumptions \eqref{E:DATASIZE},
	and the bootstrap assumptions.
	
	\medskip
	
	\noindent \textbf{Proof of \eqref{E:SINGULARFACTORTIMESWEIGHTPOINTWISE},
	\eqref{E:SINGULARFACTORTIMESWEIGHTBOUNDEDBYONE},
	and \eqref{E:ATLEASTONEDERIVATIVESINGULARFACTORTIMESWEIGHTPOINTWISE}}:
	These estimates can be proved using arguments similar to the ones
	we used to prove
	\eqref{E:WEIGHTEXACTLYONEDERIVATIVEPOINTWISE}
	and
	\eqref{E:WEIGHTATLEASTONEDERIVATIVEPOINTWISE},
	based on separately considering the cases
	$\Ifact(t,\underline{x}) > (1/4) \min \lbrace 1, \mathring{A}_* \rbrace$
	and
	$0 < \Ifact(t,\underline{x}) \leq (1/4) \min \lbrace 1, \mathring{A}_* \rbrace$
	and using the assumptions of Subsect.\ \ref{SS:WEIGHTASSUMPTIONS}.
	We omit the details, noting only that we can write
	$
	\Ifact^{-1} \weight(\Ifact^{-1} \Psi_0)
	= \Psi_0^{-1} y \weight(y)
	$
	and that
	the assumptions of Subsect.\ \ref{SS:WEIGHTASSUMPTIONS}
	(especially \eqref{E:WEIGHTVSWEIGHTDERIVATIVECOMPARISON}),
	\eqref{E:SMALLINTFACTIMPMLIESRATIOISLARGE},
	\eqref{E:SMALLINTFACTIMPMLIESPSI0ISLARGE},
	and Remark~\ref{R:SOLUTIONREMAINSINSIDEREGIMEOFHYPERBOLICITITY}
	imply that we have the following key estimates,
	relevant for the more difficult case $0 < \Ifact(t,\underline{x}) \leq (1/4) \min \lbrace 1, \mathring{A}_* \rbrace$:
	\begin{align} \label{E:USEOFWEIGHTVSWEIGHTPRIMEASSUMPTION}
	\mathbf{1}_{\left\lbrace 0 < \Ifact \leq (1/4) \min \lbrace 1, \mathring{A}_* \rbrace \right\rbrace}
	\left\lbrace
		\Psi_0^{-1} y \weight(y)
	\right\rbrace
	& \lesssim
	\mathbf{1}_{\left\lbrace 0 < \Ifact \leq (1/4) \min \lbrace 1, \mathring{A}_* \rbrace \right\rbrace}
	\left\lbrace
	y^2 
	\left|
		\weight'(y)
	\right|
	\right\rbrace^{1/2}
		\\
	& \lesssim
	\mathbf{1}_{\left\lbrace 0 < \Ifact \leq (1/4) \min \lbrace 1, \mathring{A}_* \rbrace \right\rbrace}
	\left\lbrace
	\Ifact^{-2} 
	\left|
		\weight'(y)
	\right|
	\right\rbrace^{1/2}
	\notag
	\end{align}
	and,
	for $n \leq 5$:
	$\left|
			D_Y^n (y \weight(y))
	\right|
	\lesssim 1
	$
	(Footnote~\ref{FN:DBIGYVSDDZ} is also relevant for obtaining this latter bound).
	
	\medskip
	
	\noindent \textbf{Proof of \eqref{E:POINTWISEESTIMATESIFACTMINUSTWOTIMESWEIGHTDERIVATIVE}}:
	We first note that by \eqref{E:WEIGHTPRIMEISNEGATIVE}
	and
	\eqref{E:SMALLINTFACTIMPMLIESRATIOISLARGE},
	we have $\weight'(\Ifact^{-1} \Psi_0) < 0$
	at points $(t,\underline{x})$ such that
	$\Ifact(t,\underline{x}) \leq (1/4) \min \lbrace 1, \mathring{A}_* \rbrace$.
	From this fact and the identity
	$
	1
	=
	\mathbf{1}_{\left\lbrace \Ifact > (1/4) \min \lbrace 1, \mathring{A}_* \rbrace \right\rbrace}
	+
	\mathbf{1}_{\left\lbrace 0 < \Ifact \leq (1/4) \min \lbrace 1, \mathring{A}_* \rbrace \right\rbrace}
	$,
	it follows that
	\[
	\mbox{LHS~\eqref{E:POINTWISEESTIMATESIFACTMINUSTWOTIMESWEIGHTDERIVATIVE}}
	=
	\left|
	\mathbf{1}_{\left\lbrace \Ifact > (1/4) \min \lbrace 1, \mathring{A}_* \rbrace \right\rbrace}
	\Ifact^{-2}
	\weight'(\Ifact^{-1} \Psi_0)
	\right|
	\lesssim
	\mathbf{1}_{\left\lbrace \Ifact > (1/4) \min \lbrace 1, \mathring{A}_* \rbrace \right\rbrace}
	\left|
		\weight'(\Ifact^{-1} \Psi_0)
	\right|.
	\]
	Also using the bound $\left|\weight'(y) \right| \lesssim 1$,
	which is a simple consequence of \eqref{E:ESTIMATEFORDERIVATIVESOFWEIGHT},
	we find that
	$
	\mbox{LHS~\eqref{E:POINTWISEESTIMATESIFACTMINUSTWOTIMESWEIGHTDERIVATIVE}}
	\lesssim 
	\mathbf{1}_{\left\lbrace \Ifact > (1/4) \min \lbrace 1, \mathring{A}_* \rbrace \right\rbrace}
	$.
	Next, we recall the estimate
	$\mathbf{1}_{\left\lbrace \Ifact > (1/4) \min \lbrace 1, \mathring{A}_* \rbrace \right\rbrace} 
	\lesssim \weight(\Ifact^{-1} \Psi_0)$
	that we derived in our proof of
	\eqref{E:WEIGHTEXACTLYONEDERIVATIVEPOINTWISE}.
	Combining the above estimates, we conclude the desired bound \eqref{E:POINTWISEESTIMATESIFACTMINUSTWOTIMESWEIGHTDERIVATIVE}.
	
	\medskip
	
	\noindent \textbf{Proof of \eqref{E:BOUNDEDBYONEPOINTWISEESTIMATESIFACTMINUSTWOTIMESWEIGHTDERIVATIVE}}:
	We first prove \eqref{E:BOUNDEDBYONEPOINTWISEESTIMATESIFACTMINUSTWOTIMESWEIGHTDERIVATIVE} at points $(t,\underline{x})$ 
	such that $\Ifact(t,\underline{x}) > (1/4) \min \lbrace 1, \mathring{A}_* \rbrace$.
	Using the bootstrap assumptions
	and the assumptions of Subsect.\ \ref{SS:WEIGHTASSUMPTIONS} on $\weight$, we deduce,
	in view of Remark~\ref{R:SOLUTIONREMAINSINSIDEREGIMEOFHYPERBOLICITITY},
	that
	$
	\left|
		\Ifact^{-P}
		\weight'(\Ifact^{-1} \Psi_0)
	\right|
	\lesssim 
	\left|
		\weight'(\Ifact^{-1} \Psi_0)
	\right|
	\lesssim 1
	$
	as desired.
	
	It remains for us to prove \eqref{E:BOUNDEDBYONEPOINTWISEESTIMATESIFACTMINUSTWOTIMESWEIGHTDERIVATIVE} at points $(t,\underline{x})$ 
	such that $0 < \Ifact(t,\underline{x}) \leq (1/4) \min \lbrace 1, \mathring{A}_* \rbrace$.
	Using \eqref{E:SMALLINTFACTIMPMLIESPSI0ISLARGE}, we see
	that $1 \lesssim \Psi_0(t,\underline{x})$ at such points,
	and it follows that
	$
	\left|
		\Ifact^{-P}
		\weight'(\Ifact^{-1} \Psi_0)
	\right|
	\lesssim 
	\left|
		\left\lbrace
			\Ifact^{-1} \Psi_0
		\right\rbrace^P
		\weight'(\Ifact^{-1} \Psi_0)
	\right|
	$.
	Using the assumptions of Subsect.\ \ref{SS:WEIGHTASSUMPTIONS} on $\weight$ and the assumption $P \in [0,2]$, 
	we deduce,
	in view of Remark~\ref{R:SOLUTIONREMAINSINSIDEREGIMEOFHYPERBOLICITITY},
	that the RHS of the previous expression is $\lesssim 1$ as desired.
	This finishes the proof of \eqref{E:BOUNDEDBYONEPOINTWISEESTIMATESIFACTMINUSTWOTIMESWEIGHTDERIVATIVE}
	and completes the proof of the lemma.
	
\end{proof}

\subsection{Pointwise estimates for the inhomogeneous terms in the commuted evolution equations}
\label{SS:POINTWISEINHOMOGENEOUSEVOLUTIONEQUATIONS}
With the estimates of Lemma~\ref{L:ESTIMATESINVOLVINGWEIGHT} in hand, we are now ready to derive
pointwise estimates for the inhomogeneous terms in the $\nabla^k$-commuted evolution equations.

\begin{lemma}[\textbf{Pointwise estimates for the inhomogeneous terms}]
	\label{L:POINTWISEESTIMATES}
	Let $\Ifact$ be a solution to \eqref{E:INTEGRATINFACTORODEANDIC} and let
	$\lbrace \Psi_{\alpha} \rbrace_{\alpha =0,1,2,3}$ be a solution to the system \eqref{E:PARTALTPSI0EVOLUTION}-\eqref{E:PARTALTPSIIEVOLUTION}.
	Consider the following system,\footnote{We do not bother to state the precise form of $F^{(k)}$ here.}
	obtained by commuting \eqref{E:PARTALTPSI0EVOLUTION}-\eqref{E:PARTALTPSIIEVOLUTION} 
	with $\nabla^k$:
	\begin{subequations}
	\begin{align} 
	\partial_t \nabla^k \Psi_0
	& = 
		\weight(\Ifact^{-1} \Psi_0) \sum_{a=1}^3 \partial_a \nabla^k \Psi_a
		+
		F_0^{(k)},
		\label{E:COMMUTEDPARTALTPSI0EVOLUTION} \\
	\partial_t \nabla^k \Psi_i
	& 
	= \partial_i \nabla^k \Psi_0
		+ F_i^{(k)}.
	\label{E:COMMUTEDPARTALTPSIIEVOLUTION}
	\end{align}
	\end{subequations}
	Under the data-size assumptions of Subsect.\ \ref{SS:DATAASSUMPTIONS},
	the bootstrap assumptions of Subsect.\ \ref{SS:BOOTSTRAP},
	and the smallness assumptions of Subsect.\ \ref{SS:SMALLNESSASSUMPTIONS},
	for 
	$k=2,3,4,5$
	and
	$(t,\underline{x}) \in [0,T_{(Boot)}) \times \mathbb{R}^3$, 
	the following estimate holds:
	\begin{align}
			\left|
				F_0^{(k)}
			\right|
			& \lesssim
				\varepsilon
				\left|
					\nabla^{[2,k]} \Psi_0
				\right|
				+
				\varepsilon
				\mathbf{1}_{\left\lbrace 0 < \Ifact \leq (1/4) \min \lbrace 1, \mathring{A}_* \rbrace \right\rbrace}
				\sum_{a=1}^3
				\left\lbrace
				\Ifact^{-2} 
				\left|
					\weight'(\Ifact^{-1} \Psi_0)
				\right|
				\right\rbrace^{1/2}
				\left|
					\nabla^k \Psi_a
				\right|
					\label{E:PSI0INHOMOGENEOUSTERMPOINTWISEBOUND} \\
		& \ \
				+
				\varepsilon
				\sum_{a=1}^3
				\left\lbrace
					\weight(\Ifact^{-1} \Psi_0)
				\right\rbrace^{1/2}
				\left|
					\nabla^k \Psi_a
				\right|
				\notag	\\
		& \ \
				+ 
				\sum_{a=1}^3
				\left|
					\nabla^{[1,k-1]} \Psi_a
				\right|
				+ 
				\varepsilon^2
				\sum_{a=1}^3
				\left|
					\Psi_a
				\right|
				+ 
				\sum_{a=1}^3
				\left|
					\nabla^{\leq k} \mathring{\Psi}_a
				\right|.
				\notag	
	\end{align}
	
	Moreover, for $k=0,1,2,3,4$, the following estimate holds:
	\begin{align} \label{E:ANNOYINGLOWORDERPSI0INHOMOGENEOUSTERMPOINTWISEBOUND}
			\left|
				F_0^{(k)}
			\right|
			& \lesssim
				\underbrace{
				\varepsilon
				\left|
					\nabla^{[1,k]} \Psi_0
				\right|
				}_{\mbox{\upshape Absent if $k=0$}}
				+
				\underbrace{
				\sum_{a=1}^3
				\left|
					\nabla^{[1,k]} \Psi_a
				\right|
				}_{\mbox{\upshape Absent if $k=0$}}
				+
				\varepsilon
				\sum_{a=1}^3
				\left|
					\Psi_a
				\right|
				+
				\sum_{a=1}^3
				\left|
					\nabla^{\leq k} \mathring{\Psi}_a
				\right|.
	\end{align}
	
	Finally, for $k=0,1,2,3,4,5$, the following estimate holds:
	\begin{align}
			\sum_{a=1}^3
			\left|
				F_a^{(k)}
			\right|
			& \lesssim
				\underbrace{
				\varepsilon
				\left|
					\nabla^{[2,k]} \Psi_0
				\right|
				}_{\mbox{\upshape Absent if $k=0,1$}}
				+
				\sum_{a=1}^3
				\left|
					\nabla^{\leq k} \mathring{\Psi}_a
				\right|.
				\label{E:PSIIINHOMOGENEOUSTERMPOINTWISEBOUND}
		\end{align}

\end{lemma}

\begin{proof}
	The estimate \eqref{E:PSIIINHOMOGENEOUSTERMPOINTWISEBOUND}
	follows in a straightforward fashion from commuting
	equation \eqref{E:PARTALTPSIIEVOLUTION} 
	with $\nabla^k$ and using the bootstrap assumptions, 
	the data-size assumptions \eqref{E:DATASIZE},
	and \eqref{E:DATAEPSILONVSBOOTSTRAPEPSILON}.
	
	To prove \eqref{E:PSI0INHOMOGENEOUSTERMPOINTWISEBOUND},
	we first commute equation \eqref{E:PARTALTPSI0EVOLUTION} 
	with $\nabla^k$ to obtain equation \eqref{E:COMMUTEDPARTALTPSI0EVOLUTION}.
	The only products in $F_0^{(k)}$ that are difficult to bound are those
	that feature a factor in which $k$ derivatives fall on $\Psi_a$,
	specifically the products
	$
	\sum_{a=1}^3
	\left\lbrace
	\nabla
	\left[
		\weight(\Ifact^{-1} \Psi_0)
	\right]
	\right\rbrace
	\partial_a \nabla^{k-1} \Psi_a
	$,
	$
	\sum_{a=1}^3
	\Ifact^{-1} \weight(\Ifact^{-1} \Psi_0) \Psi_a \nabla^k \Psi_a
	$,
	and
	$
	\sum_{a=1}^3
	\weight(\Ifact^{-1} \Psi_0) \mathring{\Psi}_a \nabla^k \Psi_a
	$.
	To bound the first of these, we use
	the estimate \eqref{E:WEIGHTEXACTLYONEDERIVATIVEPOINTWISE},
	which implies that the product is bounded by the second and third terms on RHS~\eqref{E:PSI0INHOMOGENEOUSTERMPOINTWISEBOUND}
	as desired.
	To handle the second and third products,
	we use 
	\eqref{E:WEIGHTBOUNDEDBYONE},
	\eqref{E:SINGULARFACTORTIMESWEIGHTPOINTWISE},
	the bootstrap assumptions, 
	the data-size assumptions \eqref{E:DATASIZE},
	and \eqref{E:DATAEPSILONVSBOOTSTRAPEPSILON}
	to bound them in magnitude by
	\begin{align*}
	&\lesssim 
	\varepsilon
	\mathbf{1}_{\left\lbrace 0 < \Ifact \leq (1/4) \min \lbrace 1, \mathring{A}_* \rbrace \right\rbrace}
	\sum_{a=1}^3
	\left\lbrace
	\Ifact^{-2} 
	\left|
		\weight'(\Ifact^{-1} \Psi_0)
	\right|
	\right\rbrace^{1/2}
	\left|
		\nabla^k \Psi_a
	\right|
		\\
	& \ \ 
	+
	\varepsilon
	\sum_{a=1}^3
	\left\lbrace
		\weight(\Ifact^{-1} \Psi_0)
	\right\rbrace^{1/2}
	\left|
		\nabla^k \Psi_a
	\right|,
	\end{align*}
	which is in turn bounded by the second and third terms on RHS~\eqref{E:PSI0INHOMOGENEOUSTERMPOINTWISEBOUND}
	as desired.
	The remaining terms in $F_0^{(k)}$ feature $\leq k-1$ derivatives of $\Psi_a$.
	These terms are easily seen to be $\lesssim \mbox{RHS}~\eqref{E:PSI0INHOMOGENEOUSTERMPOINTWISEBOUND}$
	with the help of the estimates
	\eqref{E:WEIGHTATLEASTONEDERIVATIVEPOINTWISE},
	\eqref{E:SINGULARFACTORTIMESWEIGHTBOUNDEDBYONE},
	and
	\eqref{E:ATLEASTONEDERIVATIVESINGULARFACTORTIMESWEIGHTPOINTWISE},
	the bootstrap assumptions, 
	the data-size assumptions \eqref{E:DATASIZE},
	and \eqref{E:DATAEPSILONVSBOOTSTRAPEPSILON}.
	
	The estimate \eqref{E:ANNOYINGLOWORDERPSI0INHOMOGENEOUSTERMPOINTWISEBOUND}
	is easier to prove and can be obtained in a similar fashion
	with the help of the estimates 
	\eqref{E:WEIGHTBOUNDEDBYONE},
	\eqref{E:WEIGHTEXACTLYONEDERIVATIVEPOINTWISESMALL},
	\eqref{E:WEIGHTATLEASTONEDERIVATIVEPOINTWISE},
	\eqref{E:SINGULARFACTORTIMESWEIGHTBOUNDEDBYONE},
	\eqref{E:ATLEASTONEDERIVATIVESINGULARFACTORTIMESWEIGHTPOINTWISE},
	the bootstrap assumptions, 
	the data-size assumptions \eqref{E:DATASIZE},
	and \eqref{E:DATAEPSILONVSBOOTSTRAPEPSILON}.

\end{proof}

\subsection{The main a priori estimates}
\label{SS:MAINAPRIORI}
We now derive the main result of this section: a priori estimates
that hold up to top order and that in particular yield a strict
improvement of the bootstrap assumptions. These are the main
ingredients in the proof of our main theorem.

\begin{proposition}[\textbf{The main a priori estimates}]
	\label{P:APRIORIESTIMATES}
	Let $\mathbf{1}_{\left\lbrace 0 < \Ifact \leq (1/4) \min \lbrace 1, \mathring{A}_* \rbrace \right\rbrace}$
	be the characteristic function of the spacetime subset
	$
	\lbrace
		(t,\underline{x})
		\ | \
		0 < \Ifact (t,\underline{x}) 
		\leq (1/4) \min \lbrace 1, \mathring{A}_* \rbrace
	\rbrace
	$.
	There exists a constant $C > 0$ such that 
	under the data-size assumptions of Subsect.\ \ref{SS:DATAASSUMPTIONS},
	the bootstrap assumptions of Subsect.\ \ref{SS:BOOTSTRAP},
	and the smallness assumptions of Subsect.\ \ref{SS:SMALLNESSASSUMPTIONS},
	for solutions to the system 
	\eqref{E:INTEGRATINFACTORODEANDIC}
	+
	\eqref{E:PARTALTPSI0EVOLUTION}-\eqref{E:PARTALTPSIIEVOLUTION},
	the $L^2$-controlling quantity $\mathbb{Q}_{(\mathring{\upepsilon})}$ of Def.~\ref{D:L2CONTROLLINGQUANTITY}
	verifies the following estimate for $t \in [0,\Tboot)$:
	\begin{align}  \label{E:MAINAPRIORIENERGYESTIMATES}
		&
		\mathbb{Q}_{(\mathring{\upepsilon})}(t)
		+
		\frac{1}{20} \mathring{A}_*^2
		\sum_{k=2}^5
		\sum_{a=1}^3
		\int_{s=0}^t
		\int_{\Sigma_s}
				\mathbf{1}_{\left\lbrace 0 < \Ifact \leq (1/4) \min \lbrace 1, \mathring{A}_* \rbrace \right\rbrace}
				\Ifact^{-2} 
				\left|
					\weight'(\Ifact^{-1} \Psi_0)
				\right|
				|\nabla^{k} \Psi_a|^2
		\, d \underline{x}	
		\, ds
			\\
		&
		\leq C \mathring{\upepsilon}^2.
		\notag
\end{align}

In addition the following estimates hold for $t \in [0,\Tboot)$ and $i=1,2,3$:
\begin{subequations}
\begin{align} \label{E:PARTIALTPSI0H4ESTIMATE}
	\mathring{\upepsilon}
	\| \partial_t \Psi_0 \|_{L^2(\Sigma_t)}^2
	+
	\| \nabla \partial_t \Psi_0 \|_{H^3(\Sigma_t)}^2
	& \leq C \mathring{\upepsilon}^2,
		\\
	\mathring{\upepsilon}^3
	\| \partial_t \Psi_i \|_{L^2(\Sigma_t)}^2
	+
	\| \nabla \partial_t \Psi_i \|_{H^3(\Sigma_t)}^2
	& \leq C \mathring{\upepsilon}^2.
	\label{E:PARTIALTPSIIH4ESTIMATE}
\end{align}
\end{subequations}

Moreover, the integrating factor $\Ifact$ from Def.~\ref{D:INTEGRATINGFACTOR}
verifies the following estimate for $t \in [0,\Tboot)$:
\begin{align} \label{E:IFACTAPRIORIENERGYESTIMATE}
	&
	\mathring{\upepsilon}^3
	\| \nabla \Ifact \|_{L^2(\Sigma_t)}^2
	+
	\| \nabla^{[2,5]} \Ifact \|_{L^2(\Sigma_t)}^2
		\\
	& 
	+
	\int_{s=0}^t
		\int_{\Sigma_s}
		\mathbf{1}_{\left\lbrace 0 < \Ifact \leq (1/4) \min \lbrace 1, \mathring{A}_* \rbrace \right\rbrace}
				\Ifact^{-2}
				\left|
					\weight'(\Ifact^{-1} \Psi_0)
				\right|
			\left|
				\nabla^6 \Ifact
			\right|^2
	\, d \underline{x}
	\, ds
	\leq C \mathring{\upepsilon}^2.
	\notag
\end{align}

Finally, we have the following estimates for $t \in [0,\Tboot)$,
which in particular yield strict improvements of the bootstrap assumptions 
\eqref{E:PSI0ITSELFBOOTSTRAP}-\eqref{E:IFACTITSELFBOOTSTRAP}
whenever $C \mathring{\upepsilon} < \varepsilon$:
\begin{subequations}
\begin{align} \label{E:PSI0ITSELFIMPROVED}
	\| \Psi_0 \|_{L^{\infty}(\Sigma_t)}
	& \leq 
	\mathring{A} 
	+ 
	C \mathring{\upepsilon},
			\\
	\| \nabla^{[1,3]} \Psi_0 \|_{L^{\infty}(\Sigma_t)}
	& \leq C \mathring{\upepsilon},
			 \label{E:PSI0DERIVATVESIMPROVED} \\
	\| \nabla^{\leq 2} \Psi_i \|_{L^{\infty}(\Sigma_t)}
	& \leq C \mathring{\upepsilon},
			\label{E:PSIIANDDERIVATIVESIMPROVED} 
			\\
	\| \Ifact \|_{L^{\infty}(\Sigma_t)}
	& \leq 
		1
		+
		2 \mathring{A}_*^{-1} \mathring{A}	
		+ 
		C \mathring{\upepsilon},
			\label{E:IFACTITSELFIMPROVED} 
				\\
\| \nabla^{[1,3]} \Ifact \|_{L^{\infty}(\Sigma_t)}
	& \leq C \mathring{\upepsilon}.
		\label{E:IFACTDERIVATIVESLIFTYIMPROVED}
\end{align}
\end{subequations}

\end{proposition}

\begin{proof}
\noindent \textbf{Proof of \eqref{E:MAINAPRIORIENERGYESTIMATES}}:
The main step is to derive the following estimate:
\begin{align}  \label{E:CONTROLLINGQUANTITYGRONWALLREADY}
		&
		\mathbb{Q}_{(\mathring{\upepsilon})}(t)
		+
		\frac{1}{16} \mathring{A}_*^2
		\sum_{k=2}^5
			\sum_{a=1}^3
		\int_{s=0}^t
		\int_{\Sigma_s}
				\mathbf{1}_{\left\lbrace 0 < \Ifact \leq (1/4) \min \lbrace 1, \mathring{A}_* \rbrace \right\rbrace}
				\Ifact^{-2} 
				\left|
					\weight'(\Ifact^{-1} \Psi_0)
				\right|
				|\nabla^{k} \Psi_a|^2
		\, d \underline{x}	
		\, ds
			\\
		&
		\leq C \mathring{\upepsilon}^2
		+
		C \varepsilon
		\sum_{k=2}^5
		\sum_{a=1}^3
		\int_{s=0}^t
		\int_{\Sigma_s}
				\mathbf{1}_{\left\lbrace 0 < \Ifact \leq (1/4) \min \lbrace 1, \mathring{A}_* \rbrace \right\rbrace}
				\Ifact^{-2} 
				\left|
					\weight'(\Ifact^{-1} \Psi_0)
				\right|
				|\nabla^{k} \Psi_a|^2
		\, d \underline{x}		
			\notag \\
	& \ \
		+
		C
		\int_{s=0}^t
			\mathbb{Q}_{(\mathring{\upepsilon})}(s)
		\, ds.
			\notag
\end{align}
Once we have shown
\eqref{E:CONTROLLINGQUANTITYGRONWALLREADY}, 
we can absorb the second term on RHS~\eqref{E:CONTROLLINGQUANTITYGRONWALLREADY}
into the second term on LHS~\eqref{E:CONTROLLINGQUANTITYGRONWALLREADY}, which, 
for $\varepsilon$ sufficiently small,
at most reduces the coefficient of $\frac{1}{16} \mathring{A}_*^2$ in front of the second term on the left to the value of 
$\frac{1}{20} \mathring{A}_*^2$,
as is stated on LHS~\eqref{E:MAINAPRIORIENERGYESTIMATES}.
We then use
Gronwall's inequality and the assumption $0 < t < \Tboot \leq 2 \mathring{A}_*^{-1}$
to conclude that 
$\mbox{LHS}~\eqref{E:MAINAPRIORIENERGYESTIMATES} 
\leq C \exp(C t) \mathring{\upepsilon}^2 
\leq C \exp(C \mathring{A}_*^{-1}) \mathring{\upepsilon}^2 
\leq C \mathring{\upepsilon}^2$
as desired.

To prove \eqref{E:CONTROLLINGQUANTITYGRONWALLREADY},
we must bound the terms on 
RHS~\eqref{E:INTEGRALIDENTITYFORENERGY}.
First, we note the following bound for the first term on the RHS:
$\mathbb{Q}_{(\mathring{\upepsilon})}(0) \leq C \mathring{\upepsilon}^2$,
an estimate that follows as a straightforward consequence of 
definition \eqref{E:ENERGYTOCONTROLSOLNS},
the data-size assumptions \eqref{E:DATASIZE}-\eqref{E:PSI0NOTTOONEGATIVE},
the initial condition $\Ifact|_{\Sigma_0} = 1$ stated in \eqref{E:INTEGRATINFACTORODEANDIC},
and the assumptions of Subsect.\ \ref{SS:WEIGHTASSUMPTIONS} on $\weight$.

Next, we treat the spacetime integral
on the first line of RHS~\eqref{E:INTEGRALIDENTITYFORENERGY}.
Using 
\eqref{E:SMALLINTFACTIMPMLIESPSI0ISLARGE},
\eqref{E:POINTWISEESTIMATESIFACTMINUSTWOTIMESWEIGHTDERIVATIVE},
and the bootstrap assumption \eqref{E:PSI0DERIVATVESBOOTSTRAP} for
$\| \Psi_0 \|_{L^{\infty}(\Sigma_t)}$,
we can express the integral as the negative integral
\[
-   \sum_{k=2}^5
		\sum_{a=1}^3
		\int_{s=0}^t
		\int_{\Sigma_s}
			\mathbf{1}_{\left\lbrace 0 < \Ifact \leq (1/4) \min \lbrace 1, \mathring{A}_* \rbrace \right\rbrace}
				(\Ifact^{-1} \Psi_0)^2 
				\left|
					\weight'(\Ifact^{-1} \Psi_0)
				\right|
				|\nabla^{k} \Psi_a|^2
		\, d \underline{x}	
	\, ds,
\]
which is bounded from above by the negative ``favorable integral''
\[
		-   
		\frac{1}{16} \mathring{A}_*^2
		\sum_{k=2}^5
		\sum_{a=1}^3
		\int_{s=0}^t
		\int_{\Sigma_s}
			\mathbf{1}_{\left\lbrace 0 < \Ifact \leq (1/4) \min \lbrace 1, \mathring{A}_* \rbrace \right\rbrace}
				\Ifact^{-2} 
				\left|
					\weight'(\Ifact^{-1} \Psi_0)
				\right|
				|\nabla^{k} \Psi_a|^2
		\, d \underline{x}	
	\, ds,
\]
plus an error integral that is bounded in magnitude by
\[
\lesssim
\sum_{k=2}^5
\sum_{a=1}^3
		\int_{s=0}^t
		\int_{\Sigma_s}
	\weight(\Ifact^{-1} \Psi_0)
	|\nabla^{k} \Psi_a|^2
		\, d \underline{x}	
	\, ds.
\]
We can therefore bring the favorable integral over to LHS~\eqref{E:CONTROLLINGQUANTITYGRONWALLREADY},
where it appears with a ``$+$'' sign.
Moreover, from Def.~\ref{D:L2CONTROLLINGQUANTITY}, we deduce that the error integral
is bounded by the last term on RHS~\eqref{E:CONTROLLINGQUANTITYGRONWALLREADY} as desired.

We now bound the spacetime integrals on lines two to four of
RHS~\eqref{E:INTEGRALIDENTITYFORENERGY}.
Using 
the estimates 
\eqref{E:WEIGHTBOUNDEDBYONE}
and
\eqref{E:BOUNDEDBYONEPOINTWISEESTIMATESIFACTMINUSTWOTIMESWEIGHTDERIVATIVE},
the bootstrap assumptions,
the data-size assumptions \eqref{E:DATASIZE},
and \eqref{E:DATAEPSILONVSBOOTSTRAPEPSILON},
we deduce that all three integrands are bounded in magnitude
by 
$
\lesssim
\sum_{k=2}^5
\sum_{a=1}^3
	\weight(\Ifact^{-1} \Psi_0)
	|\nabla^{k} \Psi_a|^2.
$
From Def.~\ref{D:L2CONTROLLINGQUANTITY},
we conclude that the corresponding error integrals
are bounded by the last term on RHS~\eqref{E:CONTROLLINGQUANTITYGRONWALLREADY} as desired.
Using similar reasoning, we bound the last two spacetime integrals on RHS~\eqref{E:INTEGRALIDENTITYFORENERGY}
by $\leq \mbox{RHS~\eqref{E:CONTROLLINGQUANTITYGRONWALLREADY}}$.

We now bound the spacetime integrals on lines five to seven of RHS~\eqref{E:INTEGRALIDENTITYFORENERGY}.
Using 
the estimates
\eqref{E:POINTWISEESTIMATESIFACTMINUSTWOTIMESWEIGHTDERIVATIVE}
and
\eqref{E:BOUNDEDBYONEPOINTWISEESTIMATESIFACTMINUSTWOTIMESWEIGHTDERIVATIVE},
the bootstrap assumptions,
the data-size assumptions \eqref{E:DATASIZE},
\eqref{E:DATAEPSILONVSBOOTSTRAPEPSILON},
and Young's inequality,
we deduce that all three integrands are bounded in magnitude
by 
\begin{align*}
& \lesssim
		\varepsilon
		\mathbf{1}_{\left\lbrace 0 < \Ifact \leq (1/4) \min \lbrace 1, \mathring{A}_* \rbrace \right\rbrace}
		\sum_{k=2}^5
		\sum_{a=1}^3
				\Ifact^{-2}  
				\left|
					\weight'(\Ifact^{-1} \Psi_0)
				\right|
				|\nabla^{k} \Psi_a|^2
		\\
& 
+
\varepsilon
		\sum_{k=2}^5
		\sum_{a=1}^3
		\weight(\Ifact^{-1} \Psi_0)
		|\nabla^{k} \Psi_a|^2
		+
		\sum_{k=2}^5
		|\nabla^{k} \Psi_0|^2.
\end{align*}
Appealing to Def.~\ref{D:L2CONTROLLINGQUANTITY},
we conclude that the corresponding error integrals
are bounded in magnitude by 
$\lesssim \mbox{RHS~\eqref{E:CONTROLLINGQUANTITYGRONWALLREADY}}$ as desired.

We now bound the spacetime integral on line eight of
RHS~\eqref{E:INTEGRALIDENTITYFORENERGY}, 
in which the integrand is
$
2
\sum_{k=1}^4
\sum_{a=1}^3
\nabla^{k} \Psi_a \cdot \partial_a \nabla^{k} \Psi_0
$.
Using Young's inequality, we bound this integrand by
$
\lesssim
\left|
	\nabla^{[2,5]} \Psi_0
\right|^2
+
\sum_{a=1}^3
\left|
	\nabla^{[1,4]} \Psi_a
\right|^2
$.
From Def.~\ref{D:L2CONTROLLINGQUANTITY},
we conclude that the integral of the RHS of this expression
over the spacetime domain $(s,\underline{x}) \in [0,t] \times \mathbb{R}^3$
is bounded by the last term on RHS~\eqref{E:CONTROLLINGQUANTITYGRONWALLREADY} as desired.

We now bound the spacetime integral on line nine of
RHS~\eqref{E:INTEGRALIDENTITYFORENERGY}, 
in which the integrand is
$
2 \sum_{k=2}^5
\nabla^{k} \Psi_0 \cdot F_0^{(k)}
$.
Using Young's inequality,
\eqref{E:PSI0INHOMOGENEOUSTERMPOINTWISEBOUND},
and \eqref{E:DATAEPSILONVSBOOTSTRAPEPSILON},
we pointwise bound this integrand in magnitude by
\begin{align} \label{E:MOSTDIFFICULTERRORINTEGRANDPOINTWISEBOUND}
		& \lesssim
				\left|
					\nabla^{[2,5]} \Psi_0
				\right|^2
				+
				\varepsilon
				\mathbf{1}_{\left\lbrace 0 < \Ifact \leq (1/4) \min \lbrace 1, \mathring{A}_* \rbrace \right\rbrace}
				\sum_{a=1}^3
				\Ifact^{-2} 
				\left|
					\weight'(\Ifact^{-1} \Psi_0)
				\right|
				\left|
					\nabla^{[2,5]} \Psi_a
				\right|^2
					\\
		& \ \
				+
				\sum_{a=1}^3
				\weight(\Ifact^{-1} \Psi_0)
				\left|
					\nabla^{[2,5]} \Psi_a
				\right|^2
				\notag	\\
		& \ \
				+ 
				\sum_{a=1}^3
				\left|
					\nabla^{[1,4]} \Psi_a
				\right|^2
				+ 
				\mathring{\upepsilon}^3
				\sum_{a=1}^3
				\left|
					\Psi_a
				\right|^2
				+ 
				\sum_{a=1}^3
				\left|
					\nabla^{\leq 5} \mathring{\Psi}_a
				\right|^2.
				\notag
\end{align}
From Def.~\ref{D:L2CONTROLLINGQUANTITY}
and the data-size assumptions \eqref{E:DATASIZE},
we conclude that the integral of RHS~\eqref{E:MOSTDIFFICULTERRORINTEGRANDPOINTWISEBOUND}
over the spacetime domain $(s,\underline{x}) \in [0,t] \times \mathbb{R}^3$
is $\lesssim \mbox{RHS~\eqref{E:CONTROLLINGQUANTITYGRONWALLREADY}}$ as desired.

We now bound the spacetime integral on line ten of RHS~\eqref{E:INTEGRALIDENTITYFORENERGY},
in which the integrand is
$2 \sum_{k=2}^5
\sum_{a=1}^3 
\weight(\Ifact^{-1} \Psi_0) \nabla^{k} \Psi_a \cdot F_a^{(k)}$.
Using Young's inequality,
\eqref{E:WEIGHTBOUNDEDBYONE},
and
\eqref{E:PSIIINHOMOGENEOUSTERMPOINTWISEBOUND},
we pointwise bound this integrand in magnitude by
$\lesssim
\left|
	\nabla^{[2,5]} \Psi_0
\right|^2
+
\sum_{a=1}^3
\weight(\Ifact^{-1} \Psi_0)
\left|
		\nabla^{[2,5]} \Psi_a
\right|^2
+
\sum_{a=1}^3
\left|
	\nabla^{\leq 5} \mathring{\Psi}_a
\right|^2
$.
From Def.~\ref{D:L2CONTROLLINGQUANTITY}
and the data-size assumptions \eqref{E:DATASIZE},
we conclude that the integral of the RHS of this expression
over the spacetime domain $(s,\underline{x}) \in [0,t] \times \mathbb{R}^3$
is $\lesssim \mbox{RHS~\eqref{E:CONTROLLINGQUANTITYGRONWALLREADY}}$ as desired.

We now bound the spacetime integral on line eleven of RHS~\eqref{E:INTEGRALIDENTITYFORENERGY},
in which the integrand is
$
2
\sum_{k=1}^4
\sum_{a=1}^3 
\nabla^{k} \Psi_a \cdot F_a^{(k)}
$.
Using Young's inequality
and
\eqref{E:PSIIINHOMOGENEOUSTERMPOINTWISEBOUND},
we pointwise bound this integrand in magnitude by
$\lesssim
\left|
	\nabla^{[2,4]} \Psi_0
\right|^2
+
\sum_{a=1}^3
\left|
		\nabla^{[1,4]} \Psi_a
\right|^2
+
\sum_{a=1}^3
\left|
	\nabla^{\leq 4} \mathring{\Psi}_a
\right|^2
$.
From Def.~\ref{D:L2CONTROLLINGQUANTITY}
and the data-size assumptions \eqref{E:DATASIZE},
we conclude that the integral of the RHS of this expression
over the spacetime domain $(s,\underline{x}) \in [0,t] \times \mathbb{R}^3$
is $\lesssim \mbox{RHS~\eqref{E:CONTROLLINGQUANTITYGRONWALLREADY}}$ as desired.

We now bound the spacetime integral on line twelve of RHS~\eqref{E:INTEGRALIDENTITYFORENERGY},
in which the integrand is
$
2
\mathring{\upepsilon}^3
\weight(\Ifact^{-1} \Psi_0) \nabla \Psi_0 \cdot \sum_{a=1}^3 \partial_a \nabla \Psi_a
$.
Using the estimate \eqref{E:WEIGHTBOUNDEDBYONE} and Young's inequality,
we bound this integrand by
$
\lesssim
\mathring{\upepsilon}^3
\left|
	\nabla \Psi_0
\right|^2
+
\sum_{a=1}^3 
\left|
	\nabla^2 \Psi_a
\right|^2
$.
From Def.~\ref{D:L2CONTROLLINGQUANTITY},
we conclude that the integral of the RHS of this expression
over the spacetime domain $(s,\underline{x}) \in [0,t] \times \mathbb{R}^3$
is bounded by the last term on RHS~\eqref{E:CONTROLLINGQUANTITYGRONWALLREADY} as desired.

Finally, we bound the spacetime integral on line thirteen of RHS~\eqref{E:INTEGRALIDENTITYFORENERGY},
in which the integrand is
$
2
\mathring{\upepsilon}^3
\nabla \Psi_0 \cdot F_0^{(1)}
$.
Using Young's inequality
and \eqref{E:ANNOYINGLOWORDERPSI0INHOMOGENEOUSTERMPOINTWISEBOUND},
we pointwise bound this integrand in magnitude by
\begin{align} \label{E:LOWORDERANNOYINGERRORINTEGRANDPOINTWISEBOUND}
		& \lesssim
				\mathring{\upepsilon}^3
				\left|
					\nabla \Psi_0
				\right|^2
				+
				\sum_{a=1}^3
				\left|
					\nabla \Psi_a
				\right|^2
				+
				\mathring{\upepsilon}^3
				\sum_{a=1}^3
				\left|
					\Psi_a
				\right|^2
				+
			\sum_{a=1}^3
			\left|
				\nabla^{\leq 1} \mathring{\Psi}_a
			\right|^2.
\end{align}
From Def.~\ref{D:L2CONTROLLINGQUANTITY}
and the data-size assumptions \eqref{E:DATASIZE},
we conclude that the integral of RHS~\eqref{E:LOWORDERANNOYINGERRORINTEGRANDPOINTWISEBOUND}
over the spacetime domain $(s,\underline{x}) \in [0,t] \times \mathbb{R}^3$
is $\lesssim \mbox{RHS~\eqref{E:CONTROLLINGQUANTITYGRONWALLREADY}}$ as desired.
This completes our proof of \eqref{E:CONTROLLINGQUANTITYGRONWALLREADY} 
and therefore finishes the proof of \eqref{E:MAINAPRIORIENERGYESTIMATES}.

\medskip

\noindent \textbf{Proof of \eqref{E:PSI0DERIVATVESIMPROVED} and \eqref{E:PSIIANDDERIVATIVESIMPROVED}}:
In view of Def.~\ref{D:L2CONTROLLINGQUANTITY},
we see that the estimates
$
\| \nabla^{[2,3]} \Psi_0 \|_{L^{\infty}(\Sigma_t)} \lesssim \mathring{\upepsilon}
$
and
$\| \nabla^{[1,2]} \Psi_i \|_{L^{\infty}(\Sigma_t)}  \lesssim \mathring{\upepsilon}$
follow from \eqref{E:MAINAPRIORIENERGYESTIMATES} and Sobolev embedding $H^2(\mathbb{R}^3) \hookrightarrow L^{\infty}(\mathbb{R}^3)$.
To bound $\| \nabla \Psi_0 \|_{L^{\infty}(\Sigma_t)}$, 
we first use equation \eqref{E:COMMUTEDPARTALTPSI0EVOLUTION},
the bootstrap assumptions,
the data-size assumptions \eqref{E:DATASIZE},
the estimates \eqref{E:WEIGHTBOUNDEDBYONE} and \eqref{E:ANNOYINGLOWORDERPSI0INHOMOGENEOUSTERMPOINTWISEBOUND},
inequality \eqref{E:DATAEPSILONVSBOOTSTRAPEPSILON},
and the already proven bound 
$\| \nabla^{[1,2]} \Psi_i \|_{L^{\infty}(\Sigma_t)}  \lesssim \mathring{\upepsilon}$
to obtain
$\left|
	\partial_t \nabla \Psi_0
\right|
\lesssim 
\varepsilon^2
+
\mathring{\upepsilon}
+
\sum_{a=1}^3
\left|
	\nabla^{[1,2]} \Psi_a
\right|
\lesssim
\mathring{\upepsilon}
$.
From this bound, the fundamental theorem of calculus, and the data-size assumptions \eqref{E:DATASIZE},
we find that
$\left|
	\nabla \Psi_0
\right|
\lesssim
\mathring{\upepsilon}
+
\int_{s=0}^t
	\mathring{\upepsilon}
\, ds
\lesssim \mathring{\upepsilon}
$.
This implies that 
$
\| \nabla \Psi_0 \|_{L^{\infty}(\Sigma_t)} \lesssim \mathring{\upepsilon}
$, which completes the proof of \eqref{E:PSI0DERIVATVESIMPROVED}.
Similarly, from equation \eqref{E:PARTALTPSIIEVOLUTION},
the bootstrap assumptions, 
the data-size assumptions \eqref{E:DATASIZE},
and the already proven bound 
$\| \nabla \Psi_0 \|_{L^{\infty}(\Sigma_t)} \lesssim \mathring{\upepsilon}$,
we deduce
$
\sum_{a=1}^3
\left|
	\partial_t \Psi_a
\right|
\lesssim 
\mathring{\upepsilon}
$.
From this bound, the fundamental theorem of calculus, and the data-size assumption \eqref{E:DATASIZE},
we find that
$\sum_{a=1}^3
\left|
	\Psi_a
\right|
\lesssim
\mathring{\upepsilon}
$,
which implies that 
$
\sum_{a=1}^3 \| \Psi_a \|_{L^{\infty}(\Sigma_t)} \lesssim \mathring{\upepsilon}
$, 
thereby completing the proof of \eqref{E:PSIIANDDERIVATIVESIMPROVED}.

\medskip

\noindent \textbf{Proof of \eqref{E:PSI0ITSELFIMPROVED}}:
We first use equation \eqref{E:PARTALTPSI0EVOLUTION},
the estimates \eqref{E:WEIGHTBOUNDEDBYONE} and \eqref{E:SINGULARFACTORTIMESWEIGHTBOUNDEDBYONE},
the bootstrap assumptions, 
the data-size assumptions \eqref{E:DATASIZE}, 
and the already proven bound 
$\| \nabla^{\leq 1} \Psi_i \|_{L^{\infty}(\Sigma_t)} \lesssim \mathring{\upepsilon}$
to obtain
$\left|
	\partial_t \Psi_0
\right|
\lesssim 
\mathring{\upepsilon}
$.
From this bound, the fundamental theorem of calculus, the data-size assumption \eqref{E:LARGEDATASIZE},
and the fact that $0 < t \leq 2 \mathring{A}_*^{-1}$,
we find that
$\|
	\Psi_0
\|_{L^{\infty}(\Sigma_t)}
\leq
\|
	\mathring{\Psi}_0
\|_{L^{\infty}(\Sigma_0)}
+ 
C \mathring{\upepsilon}
\leq
\mathring{A} 
+ 
C \mathring{\upepsilon}
$,
which is the desired bound \eqref{E:PSI0ITSELFIMPROVED}.

\medskip

\noindent \textbf{Proof of \eqref{E:IFACTITSELFIMPROVED} and \eqref{E:IFACTDERIVATIVESLIFTYIMPROVED}}:
We repeat the proof of
\eqref{E:IFACTCRUCIALPOINTWISE}, 
but using the bootstrap assumption
\eqref{E:IFACTITSELFBOOTSTRAP} and the
estimates \eqref{E:PSI0ITSELFIMPROVED}-\eqref{E:PSIIANDDERIVATIVESIMPROVED} 
instead of using the full set of bootstrap assumptions.
We find that
$\Ifact(t,\underline{x}) = 1 - t \mathring{\Psi}_0(\underline{x}) + \mathcal{O}(\mathring{\upepsilon})$.
From this estimate, the fact that $0 < t < 2 \mathring{A}_*^{-1}$,
and the data-size assumption \eqref{E:LARGEDATASIZE},
we conclude the desired bound \eqref{E:IFACTITSELFIMPROVED}.
Similarly, to prove \eqref{E:IFACTDERIVATIVESLIFTYIMPROVED}, 
we repeat the proof of \eqref{E:LINFINITYFORDERIVATIVESOFIFACT},
but using 
the estimates \eqref{E:PSI0ITSELFIMPROVED}-\eqref{E:IFACTITSELFIMPROVED}
instead of the bootstrap assumptions.

\medskip

\noindent \textbf{Proof of \eqref{E:IFACTAPRIORIENERGYESTIMATE}}:
The estimate \eqref{E:IFACTAPRIORIENERGYESTIMATE}	
follows as a straightforward consequence of
the pointwise estimates
\eqref{E:LOWESTLEVELPOINTWISEESTIMATEFORDERIVATIVESOFIFACT}-\eqref{E:POINTWISEESTIMATEFORDERIVATIVESOFIFACT},
the weight estimate \eqref{E:BOUNDEDBYONEPOINTWISEESTIMATESIFACTMINUSTWOTIMESWEIGHTDERIVATIVE},
the energy estimate \eqref{E:MAINAPRIORIENERGYESTIMATES},
and the data-size assumptions \eqref{E:DATASIZE}.

\medskip

\noindent \textbf{Proof of \eqref{E:PARTIALTPSI0H4ESTIMATE} and \eqref{E:PARTIALTPSIIH4ESTIMATE}}:
To prove \eqref{E:PARTIALTPSI0H4ESTIMATE}, 
we first use equation \eqref{E:COMMUTEDPARTALTPSI0EVOLUTION} 
and the estimate \eqref{E:WEIGHTBOUNDEDBYONE}
to deduce that
\begin{align} \label{E:PARTIATPSI0FIRSTH4BOUND}
	\mathring{\upepsilon}
	\| \partial_t \Psi_0 \|_{L^2(\Sigma_t)}^2
	+
	\| \nabla \partial_t \Psi_0 \|_{H^3(\Sigma_t)}^2
	& \lesssim
	\sum_{k=2}^5
	\sum_{a=1}^3
	\left\|
		\left\lbrace
			\weight(\Ifact^{-1} \Psi_0) 
		\right\rbrace^{1/2}
		\nabla^{k} \Psi_a
	\right\|_{L^2(\Sigma_t)}^2
	+
	\sum_{a=1}^3
	\left\|
		\nabla \Psi_a
	\right\|_{L^2(\Sigma_t)}^2
		\\
& \ \
	+
	\mathring{\upepsilon} \left\| F_0^{(0)} \right\|_{L^2(\Sigma_t)}^2
	+
	\sum_{k=1}^4
	\left\| F_0^{(k)} \right\|_{L^2(\Sigma_t)}^2.
	\notag
\end{align}
Next, we recall that
the already proven estimates
\eqref{E:PSI0ITSELFIMPROVED}-\eqref{E:IFACTITSELFIMPROVED} 
imply the bootstrap assumptions
\eqref{E:PSI0ITSELFBOOTSTRAP}-\eqref{E:IFACTITSELFBOOTSTRAP}
hold with $C \mathring{\upepsilon}$ in place of $\varepsilon$.
It follows that the pointwise estimate \eqref{E:ANNOYINGLOWORDERPSI0INHOMOGENEOUSTERMPOINTWISEBOUND}
holds with $C \mathring{\upepsilon}$ in place of $\varepsilon$.
From this fact,
Def.~\ref{D:L2CONTROLLINGQUANTITY},
the energy estimate \eqref{E:MAINAPRIORIENERGYESTIMATES},
and the data-size assumptions \eqref{E:DATASIZE},
we deduce that $\mbox{RHS~\eqref{E:PARTIATPSI0FIRSTH4BOUND}} \lesssim \mathring{\upepsilon}^2$,
which is the desired bound \eqref{E:PARTIALTPSI0H4ESTIMATE}.

The estimate \eqref{E:PARTIALTPSIIH4ESTIMATE} can be proved using similar arguments
based on the evolution equation \eqref{E:COMMUTEDPARTALTPSIIEVOLUTION} 
and the pointwise estimate \eqref{E:PSIIINHOMOGENEOUSTERMPOINTWISEBOUND}, 
and we omit the details.

\end{proof}

\section{Local well-posedness and continuation criteria}
\label{S:WELLPOSEDNESS}
In this section, we provide a proposition that yields
standard well-posedness results and continuation criteria
pertaining to the quantities
$\lbrace \partial_{\alpha} \Phi \rbrace_{\alpha = 0,1,2,3}$, $\Ifact$, and $\lbrace \Psi_{\alpha} \rbrace_{\alpha = 0,1,2,3}$.

\begin{proposition}
	\label{P:LOCALWELLPOSEDNESSANDCONTINUATIONCRITERIA}
	Let $N \geq 3$ be an integer and let
	$(\partial_t \Phi|_{\Sigma_0},\partial_1 \Phi|_{\Sigma_0}, \partial_2 \Phi|_{\Sigma_0}, \partial_3 \Phi|_{\Sigma_0}) 
	= (\mathring{\Psi}_0,\mathring{\Psi}_1,\mathring{\Psi}_2,\mathring{\Psi}_3)$ 
	be initial data (see Remark~\ref{R:NOPHIINWAVEEQUATION})
	for the equation \eqref{E:WAVE} verifying
	$\mathring{\Psi}_{\alpha} \in H^N(\mathbb{R}^3)$, $(\alpha = 0,1,2,3)$.
	Let $\mathcal{H} := (-1/2,\infty)$,
	and note that the following holds: 
	equation \eqref{E:WAVE} is a non-degenerate\footnote{By non-degenerate, 
	we mean that relative to the Cartesian coordinates, the 
	$4 \times 4$ matrix of components $g_{\alpha \beta}$ has signature $(-,+,+,+)$,
	where $g := - dt^2 + \frac{1}{\weight(\partial_t \Phi)}\sum_{a=1}^3 (dx^a)^2$
	is the metric corresponding to equation \eqref{E:WAVE}.} 
	wave equation
	at points $(t,\underline{x})$ such that
	$\partial_t \Phi(t,\underline{x}) \in \mathcal{H}$
	(see \eqref{E:WEIGHTISPOSITIVE} for justification of this assertion).
	Assume that $\mathring{\Psi}_0(\mathbb{R}^3)$ is contained in a compact subset $\mathfrak{K}$ of $\mathcal{H}$.
	Let $\Ifact$, $\Ifact_{\star}$, and $\lbrace \Psi_{\alpha} \rbrace_{\alpha=0,1,2,3}$ be the quantities defined in
	Defs.\ \ref{D:INTEGRATINGFACTOR} and \ref{D:RENORMALIZEDSOLUTION}.
	Then there exist a compact set $\mathfrak{K}'$ of $\mathcal{H}$ containing $\mathfrak{K}$ in its interior
	and a time $T > 0$ depending on $\mathfrak{K}$
	and
	$\sum_{\alpha = 0}^3 \| \mathring{\Psi}_{\alpha} \|_{H^N}$,
	such that a unique classical solution to equation \eqref{E:WAVE} exists on
	$[0,T) \times \mathbb{R}^3$, such that
	$\partial_t \Phi([0,T) \times \mathbb{R}^3) \subset \mathfrak{K}'$,
	and such that the following regularity properties hold for $\alpha = 0,1,2,3$:
	\begin{align}	 \label{E:LWPREGULARITY}
		\partial_{\alpha} \Phi \in C\left([0,T),H^N \right).
	\end{align}
	In addition, the solution depends continuously on the data.
	
	Let $T_{(Lifespan)}$ be the supremum of all times $T > 0$ such that the 
	classical solution to \eqref{E:WAVE} exists on $[0,T) \times \mathbb{R}^3$ and verifies the above properties. 
	Then either $T_{(Lifespan)} =  \infty$, or $T_{(Lifespan)} < \infty$ and one of the following two breakdown scenarios must occur:
	\begin{enumerate}
		\item There exists a sequence of points 
		$\lbrace (t_n,\underline{x}_n) \rbrace_{n=1}^{\infty} \subset [0,T_{(Lifespan)}) \times \mathbb{R}^3$
		such that $\partial_t \Phi(t_n,\underline{x}_n)$ escapes every compact subset of $\mathcal{H}$ as $n \to \infty$.
		\item 
			$
			\displaystyle
			\lim_{t \uparrow T_{(Lifespan)}} 
			\sup_{s \in [0,t)}
			\left\|
				\nabla^{\leq 1} \partial_t \Phi
			\right\|_{L^{\infty}(\Sigma_s)}
			= 
			\infty
			$.
	\end{enumerate}	
	
	Moreover, 
	on $[0,T_{(Lifespan)}) \times \mathbb{R}^3$,
	$\Ifact$ and $\lbrace \Psi_{\alpha} \rbrace_{\alpha=0,1,2,3}$
	are classical solutions to
	equations \eqref{E:INTEGRATINFACTORODEANDIC} and
	\eqref{E:PARTALTPSI0EVOLUTION}-\eqref{E:PARTALTPSIIEVOLUTION}
	such that
	\begin{align}	 \label{E:PSILWPREGULARITY}
		& \Ifact - 1 \in C\left([0,T_{(Lifespan)}),H^{N+1}(\mathbb{R}^3) \right),
		&& \Psi_{\alpha} \in C\left([0,T_{(Lifespan)}),H^N(\mathbb{R}^3) \right).
	\end{align}
	
	Finally, $\Ifact_{\star}$ satisfies the following estimates:
	\begin{align} \label{E:IFACTSTRICTLYPOSITIVEDURINGCLASSICALLIFESPAN}
		& 0 < \Ifact_{\star}(t) < \infty,
		&& \mbox{for } t \in [0,T_{(Lifespan)}).
	\end{align}

\end{proposition}

\begin{proof}
	The statements concerning $\Phi$ are standard
	and can be proved using the ideas found, for example, in \cite{jS2008c}.
	
	Next, we note that the evolution equation + initial condition for $\Ifact$
	stated in \eqref{E:INTEGRATINFACTORODEANDIC},
	the fact that $\Ifact(t,\cdot) - 1$ is compactly supported in space (see Remark~\ref{R:BOUNDEDWAVESPEED}),
	and the fact that
	$\partial_t \Phi \in C\left([0,T_{(Lifespan)}),H^N(\mathbb{R}^3) \right) \subset C\left([0,T_{(Lifespan)}),C^1(\mathbb{R}^3) \right)$
	(i.e., \eqref{E:LWPREGULARITY})
	can be used to deduce \eqref{E:IFACTSTRICTLYPOSITIVEDURINGCLASSICALLIFESPAN}.
	Similarly, from 
	\eqref{E:INTEGRATINFACTORODEANDIC},
	the identity \eqref{E:FORMULAFORDERIVATIVESOFINTEGRATINGFACTOR},
	the definition $\Psi_{\alpha} := \Ifact \partial_{\alpha} \Phi$ (see Def.\ \ref{D:RENORMALIZEDSOLUTION}),
	\eqref{E:LWPREGULARITY},
	and the standard Sobolev--Moser calculus,
	it is straightforward to deduce \eqref{E:PSILWPREGULARITY}.
	
\end{proof}

\section{The main theorem}
\label{S:MAINTHM}
In this section, we state and prove our main stable blowup result.

\begin{theorem}[\textbf{Stable ODE-type blowup}]
	\label{T:STABILITYOFODEBLOWUP}
	Assume that the weight function $\weight$ verifies the assumptions
	stated in Subsect.\ \ref{SS:WEIGHTASSUMPTIONS}.
	Consider compactly supported initial data 
	$(\partial_t \Phi|_{\Sigma_0},\partial_1 \Phi|_{\Sigma_0}, \partial_2 \Phi|_{\Sigma_0}, \partial_3 \Phi|_{\Sigma_0}) 
	= (\mathring{\Psi}_0,\mathring{\Psi}_1,\mathring{\Psi}_2,\mathring{\Psi}_3)$
	for the wave equation \eqref{E:WAVE}
	(see Remark~\ref{R:NOPHIINWAVEEQUATION} concerning the data)
	that verify the data-size assumptions \eqref{E:DATASIZE}-\eqref{E:PSI0NOTTOONEGATIVE}
	involving the parameters $\mathring{\upepsilon}$ and $\mathring{A}$,
	and let $\mathring{A}_*$ be the data-size parameter defined in \eqref{E:CRUCIALDATASIZEPARAMETER}.
	Let $\Ifact$, $\Ifact_{\star}$, and $\lbrace \Psi_{\alpha} \rbrace_{\alpha=0,1,2,3}$ be the quantities defined in
	Defs.\ \ref{D:INTEGRATINGFACTOR} and \ref{D:RENORMALIZEDSOLUTION}.
	We define
	\begin{align} \label{E:TLIFESPAN}
		T_{(Lifespan)}
		& := \sup 
					\left\lbrace
						t > 0 \ | \ \lbrace \partial_{\alpha} \Phi \rbrace_{\alpha=0,1,2,3} \mbox{ exist classically on } [0,t) \times \mathbb{R}^3
					\right\rbrace.
	\end{align}
	If
	$\mathring{\upepsilon} > 0$,
	$\mathring{A} > 0$,
	and
	$\mathring{A}_* > 0$,
	and if $\mathring{\upepsilon}$ is small relative to 
	$\mathring{A}^{-1}$
	and
	$\mathring{A}_*$ 
	in the sense explained in Subsect.\ \ref{SS:SMALLNESSASSUMPTIONS}, then
	the following conclusions hold.
	
	\medskip
	\noindent \underline{\textbf{Characterization of the solution's classical lifespan}}:
	The solution's classical lifespan is characterized by $\Ifact_{\star}$ as follows:
	\begin{align} \label{E:LIFESPANCHARACTERIZATION}
		T_{(Lifespan)}
		= \sup\left\lbrace t > 0 \ | \  \inf_{s \in [0,t)} \Ifact_{\star}(s) > 0 \right\rbrace.
	\end{align}
	Moreover,
	\begin{subequations}
	\begin{align}
		& \Ifact(t,\underline{x}) > 0 \mbox{ for } (t,\underline{x}) \in [0,T_{(Lifespan)}) \times \mathbb{R}^3,
			 \label{E:IFACTPOSITIVEDURINGCLASSICALLIFESPAN} \\
		& \lim_{t \uparrow T_{(Lifespan)}} \Ifact_{\star}(t) = 0.	
			\label{E:IFACTMINVANISHESATLIFESPAN}
	\end{align}
	\end{subequations}
	
	In addition, the following estimate holds:
	\begin{align} \label{E:MAINTHMLIFESPANESTIMATE}
		T_{(Lifespan)} 
		& = \mathring{A}_*^{-1} 
				\left\lbrace
					1
					+ 
					\mathcal{O}(\mathring{\upepsilon})
				\right\rbrace.
	\end{align}
	
	\medskip
	\noindent \underline{\textbf{Regularity properties of} 
	$\Psi_{\alpha}$ \textbf{and } $\Ifact$ \textbf{on } $[0,T_{(Lifespan)}) \times \mathbb{R}^3$}:
	On the slab
	$[0,T_{(Lifespan)}) \times \mathbb{R}^3$,
	the solution verifies the energy bounds
	\eqref{E:MAINAPRIORIENERGYESTIMATES}-\eqref{E:IFACTAPRIORIENERGYESTIMATE},
	the $L^{\infty}$ estimates \eqref{E:PSI0ITSELFIMPROVED}-\eqref{E:IFACTDERIVATIVESLIFTYIMPROVED},
	\eqref{E:PSI0WELLAPPROXIMATED}-\eqref{E:PSI0BIGGERTHANMINUSONEHALF},
	and
	\eqref{E:IFACTCRUCIALPOINTWISE}-\eqref{E:SMALLINTFACTIMPMLIESRATIOISLARGE}
	(with $C \mathring{\upepsilon}$ on the RHS in place of $\varepsilon$ in these equations).
	Moreover, $\lbrace \Psi_{\alpha} \rbrace_{\alpha=0,1,2,3}$ and $\Ifact$
	enjoy the following regularity:
	\begin{subequations}
	\begin{align}
		\label{E:MAINTHEOREMPSI0REGULARITY}
		\Psi_0  
		& \in 
			C\left([0,T_{(Lifespan)}), H^5(\mathbb{R}^3) \right)
			\cap
			L^{\infty}\left([0,T_{(Lifespan)}), H^5(\mathbb{R}^3) \right),
			\\
		\Psi_i  
		& \in C\left([0,T_{(Lifespan)}), H^5(\mathbb{R}^3) \right)
			\cap
			L^{\infty}\left([0,T_{(Lifespan)}), H^4(\mathbb{R}^3) \right),
			\label{E:MAINTHEOREMPSIIREGULARITY}
				\\
		\Ifact - 1
		& \in 
			C\left([0,T_{(Lifespan)}), H^6(\mathbb{R}^3) \right)
			\cap
			L^{\infty}\left([0,T_{(Lifespan)}), H^5(\mathbb{R}^3) \right).
		\label{E:MAINTHEOREMIFACTREGULARITY}
	\end{align}
	\end{subequations}

	\medskip
	\noindent \underline{\textbf{Regularity properties of} 
	$\Psi_{\alpha}$ \textbf{and } $\Ifact$ \textbf{on } $[0,T_{(Lifespan)}] \times \mathbb{R}^3$}:
	$\Psi_{\alpha}$ and $\Ifact$
	\textbf{do not blow up} at time $T_{(Lifespan)}$, but rather
	continuously extend to $[0,T_{(Lifespan)}] \times \mathbb{R}^3$
	as functions that enjoy the following regularity for any $N < 5$:
	\begin{subequations}
	\begin{align}
		\label{E:EXTENDEDMAINTHEOREMPSI0REGULARITY}
		\Psi_0  
		& \in 
			L^{\infty}\left([0,T_{(Lifespan)}], H^5(\mathbb{R}^3) \right)
			\cap
			C\left([0,T_{(Lifespan)}], H^N(\mathbb{R}^3) \right),
			\\
		\Psi_i  
		& \in C\left([0,T_{(Lifespan)}], H^4(\mathbb{R}^3) \right),
			\label{E:EXTENDEDMAINTHEOREMPSIIREGULARITY}
				\\
		\Ifact - 1
		& \in C\left([0,T_{(Lifespan)}], H^5(\mathbb{R}^3) \right).
		\label{E:EXTENDEDMAINTHEOREMIFACTREGULARITY}
	\end{align}
	\end{subequations}
	
	\medskip
	\noindent \underline{\textbf{Description of the vanishing of } $\Ifact$ \textbf{and the blowup of} $\partial_t \Phi$}:
	For $(t,\underline{x}) \in [0,T_{(Lifespan)}) \times \mathbb{R}^3$,
	we have
	\begin{align} \label{E:MAINTHMSMALLINTFACTIMPMLIESPSI0ISLARGE}
		\Ifact(t,\underline{x}) 
		\leq \frac{1}{2}
		\implies
		\partial_t \Phi(t,\underline{x})
		\geq \frac{1}{4 \Ifact(t,\underline{x})} \mathring{A}_*.
	\end{align}
	
	Let 
	\begin{align} \label{E:BLOWUPSET}
		\Sigma_{T_{(Lifespan)}}^{Blowup}
		& := \lbrace (T_{(Lifespan)}, \underline{x}) \ | \ \Ifact(T_{(Lifespan)},\underline{x}) = 0 \rbrace.
	\end{align}
	
	Then if $(T_{(Lifespan)}, \underline{x}) \in \Sigma_{T_{(Lifespan)}}^{Blowup}$,
	we have\footnote{See also Remark~\ref{R:BLOWUPOFPHI} concerning the blowup of $\Phi$ itself, if initial data for
	$\Phi$ itself are prescribed.}
	\begin{align} \label{E:BLOWUPBEHAVIOR}
		\lim_{t \uparrow T_{(Lifespan)}} \partial_t \Phi(t,\underline{x}) = \infty.
	\end{align}
	
	Finally, if $(T_{(Lifespan)}, \underline{x}) \notin \Sigma_{T_{(Lifespan)}}^{Blowup}$,
	then there exists an open ball $B_{\underline{x}} \subset \mathbb{R}^3$ centered at $\underline{x}$
	such that for $\alpha = 0,1,2,3$, we have
	$\partial_{\alpha} \Phi \in C\left([0,T_{(Lifespan)}], H^5(B_{\underline{x}}) \right)$.
	
\end{theorem}

\begin{proof}
	Let $C_* > 1$ be a constant; we will enlarge $C_*$ as needed throughout the proof.
	Let $T_{(Max)}$ be the supremum of times $0 \leq T \leq 2 \mathring{A}_*^{-1}$ such that 
	the following properties hold:
	\begin{itemize}
		\item $\lbrace \partial_{\alpha} \Phi \rbrace_{\alpha=0,1,2,3}$ is a classical solution to \eqref{E:WAVE} on $[0,T) \times \mathbb{R}^3$
			(see Remark~\ref{R:NOPHIINWAVEEQUATION})
			verifying the properties stated in Prop.~\ref{P:LOCALWELLPOSEDNESSANDCONTINUATIONCRITERIA} 
			(with $N=5$ in the proposition).
		\item $\Ifact$ is a classical solution to \eqref{E:INTEGRATINFACTORODEANDIC} on $[0,T) \times \mathbb{R}^3$
			verifying the properties stated in Prop.~\ref{P:LOCALWELLPOSEDNESSANDCONTINUATIONCRITERIA}.
		\item $\lbrace \Psi_{\alpha} \rbrace_{\alpha = 0,1,2,3}$ are classical solutions to 
			\eqref{E:PARTALTPSI0EVOLUTION}-\eqref{E:PARTALTPSIIEVOLUTION}
			for $(t,\underline{x}) \in [0,T) \times \mathbb{R}^3$
			such that the properties stated in Prop.~\ref{P:LOCALWELLPOSEDNESSANDCONTINUATIONCRITERIA}
			hold.
		\item The bootstrap assumptions \eqref{E:BOOTSTRAPRATIO} and \eqref{E:BOOTSTRAPSMALLINTFACTIMPMLIESPSI0ISLARGE} 
			hold for $(t,\underline{x}) \in [0,T) \times \mathbb{R}^3$.
		\item The $L^{\infty}$ bootstrap assumptions \eqref{E:PSI0ITSELFBOOTSTRAP}-\eqref{E:IFACTITSELFBOOTSTRAP} 
			hold for $t \in [0,T)$ with $\varepsilon := C_* \mathring{\upepsilon}$.
		\item $\inf \left\lbrace \Ifact_{\star}(t) \ | \ t \in [0,T) \right\rbrace > 0$,
			where $\Ifact_{\star}$ is defined in \eqref{E:IFACTMIN}.
			Note that this implies that the bootstrap assumption \eqref{E:HYPERBOLICBOOTSTRAP} holds on $[0,T)$.
	\end{itemize}
	Throughout the rest of the proof, we will assume that
	$\mathring{\upepsilon}$ is sufficiently small and that $C_*$ is sufficiently large
	without explicitly mentioning it every time.
	Next, we note that the hypotheses of
	Prop.~\ref{P:LOCALWELLPOSEDNESSANDCONTINUATIONCRITERIA} hold with $N=5$.
	Hence, by Prop.~\ref{P:LOCALWELLPOSEDNESSANDCONTINUATIONCRITERIA} and Sobolev embedding,
	we have $T_{(Max)} > 0$.
	
	We will now show that
	$T_{(Max)} = T_{(Lifespan)}$.
	Clearly $T_{(Max)} \leq T_{(Lifespan)}$ and thus it suffices to show that $T_{(Lifespan)} \leq T_{(Max)}$.
	To proceed, we assume for the sake of deriving a contradiction that
	\[ 
		\inf_{s \in [0,T_{(Max)})} \Ifact_{\star}(s) > 0.
	\]
	Then, in view of Defs.~\ref{D:INTEGRATINGFACTOR} and ~\ref{D:RENORMALIZEDSOLUTION} and the bootstrap assumptions,
	we see that this assumption implies that
	\[
		\lim_{t \uparrow T_{(Max)}} 
			\sup_{s \in [0,t)}
			\left\lbrace
			\left\|
				\partial_t \Phi
			\right\|_{L^{\infty}(\Sigma_s)}
			+
			\sum_{\alpha = 0}^3
			\left\|
				\partial_{\alpha}
				\partial_t \Phi
			\right\|_{L^{\infty}(\Sigma_s)}
			\right\rbrace
			<
			\infty.
	\]
	It follows that neither of the two breakdown scenarios of 
	Prop.~\ref{P:LOCALWELLPOSEDNESSANDCONTINUATIONCRITERIA} occur on $[0,T_{(Max)}) \times \mathbb{R}^3$.
	Moreover, by Prop.~\ref{P:APRIORIESTIMATES}, if $C_*$ is large enough, 
	then the bootstrap assumption inequalities
	\eqref{E:PSI0ITSELFBOOTSTRAP}-\eqref{E:IFACTITSELFBOOTSTRAP} 
	hold in a strict sense (that is, with ``$\leq$'' replaced by ``$<$'') on $[0,T_{(Max)}) \times \mathbb{R}^3$.
	Moreover, all estimates proved prior to Prop.~\ref{P:APRIORIESTIMATES} hold with 
	$\varepsilon$ replaced by $C \mathring{\upepsilon}$, and we will use this fact in the rest of the proof
	without mentioning it again.
	Furthermore,
	\eqref{E:PSI0NEGATIVEIMPMLIESINHYPERBOLICREGIME}
	and
	\eqref{E:SMALLINTFACTIMPMLIESPSI0ISLARGE}
	respectively yield strict improvements of the bootstrap assumptions
	\eqref{E:BOOTSTRAPRATIO} and \eqref{E:BOOTSTRAPSMALLINTFACTIMPMLIESPSI0ISLARGE}	
	for $(t,\underline{x}) \in [0,T_{(Max)}) \times \mathbb{R}^3$.
	Next, we note that the estimate \eqref{E:IFACTSTARCRUCIALPOINTWISE}
	implies that $\Ifact_{\star}(t)$ cannot remain positive for $t$ larger than
	$
	\mathring{A}_*^{-1} 
				\left\lbrace
					1
					+ 
					\mathcal{O}(\mathring{\upepsilon})
		\right\rbrace
	$.
	From this fact, it follows that $T_{(Max)} < 2 \mathring{A}_*^{-1}$.
	Combining these facts and appealing to Prop.~\ref{P:LOCALWELLPOSEDNESSANDCONTINUATIONCRITERIA},
	we deduce that
	$\lbrace \partial_{\alpha} \Phi \rbrace_{\alpha=0,1,2,3}$, 
	$\lbrace \Psi_{\alpha} \rbrace_{\alpha=0,1,2,3}$, 
	and $\Ifact$ extend as classical solutions to 
	a region of the form $[0,T_{(Max)} + \Delta) \times \mathbb{R}^3$ for some $\Delta > 0$ with $T_{(Max)} + \Delta < 2 \mathring{A}_*^{-1}$
	on which the solution has the same Sobolev regularity as the data,
	such that  
	$\inf_{s \in [0,T_{(Max)} + \Delta]} \Ifact_{\star}(s) > 0$,
	and such that the bootstrap assumptions
	\eqref{E:BOOTSTRAPRATIO}-\eqref{E:IFACTITSELFBOOTSTRAP}
	hold for $(t,\underline{x}) \in [0,T_{(Max)} + \Delta] \times \mathbb{R}^3$.
	In total, this contradicts the definition of $T_{(Max)}$.
	We therefore conclude that
	\begin{align} \label{E:IFACTVANISHESONTMAX}
		T_{(Max)} 
		= \sup
			\left\lbrace 
				t > 0 \ | \ \inf_{s \in [0,t)} \Ifact_{\star}(s) > 0 
			\right\rbrace
	\end{align}
	and that the estimates \eqref{E:MAINAPRIORIENERGYESTIMATES}-\eqref{E:IFACTAPRIORIENERGYESTIMATE}
	and
	\eqref{E:PSI0ITSELFIMPROVED}-\eqref{E:IFACTDERIVATIVESLIFTYIMPROVED} 
	hold for $t \in [0,T_{(Max)})$.
	
	Next, we note that the estimate
	\eqref{E:MAINTHMSMALLINTFACTIMPMLIESPSI0ISLARGE} follows from \eqref{E:SMALLINTFACTIMPMLIESPSI0ISLARGE}.
	In particular, it follows from
	\eqref{E:MAINTHMSMALLINTFACTIMPMLIESPSI0ISLARGE} and \eqref{E:IFACTVANISHESONTMAX} and that
	$
	\displaystyle
			\lim_{t \uparrow T_{(Max)}} \sup_{s \in [0,t)}
			\left\|
				\partial_t \Phi
			\right\|_{L^{\infty}(\Sigma_s)}
			=
			\infty
	$,
	that is, that $\partial_t \Phi$ blows up at time $T_{(Max)}$.
	We have therefore shown that 
	$T_{(Max)} = T_{(Lifespan)}$ 
	and that $T_{(Lifespan)}$ is characterized by \eqref{E:LIFESPANCHARACTERIZATION}.
	Moreover, the arguments given in the previous paragraph imply that
	$\Ifact_{\star}$ vanishes for the first time at
	$	
	T_{(Lifespan)}
	=
	\mathring{A}_*^{-1} 
				\left\lbrace
					1
					+ 
					\mathcal{O}(\mathring{\upepsilon})
		\right\rbrace$,
	which in total yields 
	\eqref{E:IFACTPOSITIVEDURINGCLASSICALLIFESPAN}
	and
	\eqref{E:MAINTHMLIFESPANESTIMATE}.
	
	In the rest of this proof,
	we sometimes silently use that
	$
	\Psi_0 \in L^{\infty}\left([0,T_{(Lifespan)}),L^2(\mathbb{R}^3)\right)
	$
	and
	$
	\Ifact - 1 \in L^{\infty}\left([0,T_{(Lifespan)}),L^2(\mathbb{R}^3)\right)
	$.
	These facts do not follow from 
	the energy estimates \eqref{E:MAINAPRIORIENERGYESTIMATES} and \eqref{E:IFACTAPRIORIENERGYESTIMATE},
	but instead follow from
	\eqref{E:PSI0ITSELFIMPROVED}, \eqref{E:IFACTITSELFIMPROVED}, 
	and the compactly supported (in space) nature of $\Psi_0$ and $\Ifact - 1$.
	Next, we easily conclude from the definition \eqref{E:ENERGYTOCONTROLSOLNS}
	of $\mathbb{Q}_{(\mathring{\upepsilon})}(t)$
	and the fact that the estimate \eqref{E:MAINAPRIORIENERGYESTIMATES}
	holds on $[0,T_{(Lifespan)})$ 
	that
	$
		\Psi_0 \in L^{\infty}\left([0,T_{(Lifespan)}), H^5(\mathbb{R}^3) \right)
	$
	(as is stated in \eqref{E:MAINTHEOREMPSI0REGULARITY})
	and that
	$
	\Psi_i  
	\in 
	L^{\infty} \left([0,T_{(Lifespan)}), H^4(\mathbb{R}^3) \right)
	$
	(as is stated in \eqref{E:MAINTHEOREMPSIIREGULARITY}).
	The same reasoning yields that
	$
		\Psi_0 \in L^{\infty}\left([0,T_{(Lifespan)}], H^5(\mathbb{R}^3) \right)
	$
	(as is stated in \eqref{E:EXTENDEDMAINTHEOREMPSI0REGULARITY}),
	where the open time interval is replaced with $[0,T_{(Lifespan)}]$.
	The fact that
	$
		\Psi_{\alpha} \in C\left([0,T_{(Lifespan)}), H^5(\mathbb{R}^3) \right)
	$
	(as is stated in \eqref{E:MAINTHEOREMPSI0REGULARITY})
	is a standard result that can be proved using 
	energy estimate arguments (similar to the ones we used to prove \eqref{E:MAINAPRIORIENERGYESTIMATES}) 
	and standard facts from functional analysis.
	We omit the details and instead refer the reader to 
	\cite{jS2008c}*{Section 2.7.5}.
	We clarify that in proving this ``soft result,''
	it is important that for fixed $t \in [0,T_{(Lifespan)})$,
	we have
	$\min_{[0,t] \times \mathbb{R}^3} \Ifact > 0$,
	which implies in particular that the weight $\weight(\Ifact^{-1} \Psi_0)$
	on the right-hand side of \eqref{E:ENERGYTOCONTROLSOLNS}
	is bounded from above and from below away from $0$ on $[0,t] \times \mathbb{R}^3$
	(and thus the energy estimates are non-degenerate away from $\Sigma_{T_{(Lifespan)}}$).
	Through similar reasoning based on 
	equation \eqref{E:INTEGRATINFACTORODEANDIC} (which states that $\partial_t \Ifact = - \Psi_0$),
	the identity \eqref{E:FORMULAFORDERIVATIVESOFINTEGRATINGFACTOR},
	and the estimate \eqref{E:IFACTAPRIORIENERGYESTIMATE},
	we deduce that
	$
	\mathcal{I} - 1  
	\in 
	L^{\infty} \left([0,T_{(Lifespan)}), H^5(\mathbb{R}^3) \right)
	\cap
	C\left([0,T_{(Lifespan)}), H^6(\mathbb{R}^3) \right)
	$.
	We have therefore proved \eqref{E:MAINTHEOREMPSI0REGULARITY}-\eqref{E:MAINTHEOREMIFACTREGULARITY}.
	
	We will now prove \eqref{E:EXTENDEDMAINTHEOREMPSI0REGULARITY}-\eqref{E:EXTENDEDMAINTHEOREMIFACTREGULARITY}.
	We first note that the estimates 
	\eqref{E:MAINAPRIORIENERGYESTIMATES} and \eqref{E:PARTIALTPSI0H4ESTIMATE}-\eqref{E:PARTIALTPSIIH4ESTIMATE}
	and equation \eqref{E:INTEGRATINFACTORODEANDIC}
	imply that
	$
	\partial_t
		\Psi_{\alpha}
		\in 
		L^{\infty}\left([0,T_{(Lifespan)}), H^4(\mathbb{R}^3) \right)
	$
	and 
	$
	\partial_t
		\Ifact
		\in 
	  L^{\infty}\left([0,T_{(Lifespan)}), H^5(\mathbb{R}^3) \right)
	$.
	Hence, from the fundamental theorem of calculus, 
	the initial conditions \eqref{E:INTEGRATINFACTORODEANDIC} and \eqref{E:DATASIZE}-\eqref{E:LARGEDATASIZE},
	and the completeness
	of the Sobolev spaces $H^M(\mathbb{R}^3)$, we obtain
	$
		\Psi_{\alpha}
		\in 
			C\left([0,T_{(Lifespan)}],H^4(\mathbb{R}^3) \right)
	$
	and
	$
		\Ifact - 1
		\in 
			C\left([0,T_{(Lifespan)}],H^5(\mathbb{R}^3) \right)
	$.
	In particular, we have shown
	\eqref{E:EXTENDEDMAINTHEOREMPSIIREGULARITY}-\eqref{E:EXTENDEDMAINTHEOREMIFACTREGULARITY}.
	Moreover, \eqref{E:EXTENDEDMAINTHEOREMIFACTREGULARITY} and Sobolev embedding together yield that
	$\Ifact \in C\left([0,T_{(Lifespan)}],C(\mathbb{R}^3) \right)$
	and thus $\Ifact_{\star} \in C\left([0,T_{(Lifespan)}],C(\mathbb{R}^3) \right)$.
	Since we have already shown that $T_{(Lifespan)} = \mbox{RHS~\eqref{E:IFACTVANISHESONTMAX}}$,
	it follows that $\Ifact_{\star}(T_{(Lifespan)}) = 0$ and that
	$\lim_{t \uparrow T_{(Lifespan)}} \Ifact_{\star}(t) = 0$,
	that is, that \eqref{E:IFACTMINVANISHESATLIFESPAN} holds.
	To obtain that for $N < 5$, we have
	$\Psi_0
		\in 
			C\left([0,T_{(Lifespan)}],H^N(\mathbb{R}^3) \right)
	$
	(as is stated in \eqref{E:EXTENDEDMAINTHEOREMPSI0REGULARITY}),
	we interpolate between\footnote{Here, we mean the following standard inequality: 
	if $f \in H^5(\Sigma_t)$ and $0 \leq N \leq 5$, 
	then there exists a constant $C_N > 0$ such that
	$\| f \|_{H^N(\Sigma_t)} \leq C_N \| f \|_{L^2(\Sigma_t)}^{1-N/5} \| f \|_{H^5(\Sigma_t)}^{N/5}$. \label{FN:INTERPOLATION}}
	$L^2$ and $H^5$
	and
	use the already shown facts
	$
		\Psi_0 \in L^{\infty}\left([0,T_{(Lifespan)}], H^5(\mathbb{R}^3) \right)
		\cap
		C\left([0,T_{(Lifespan)}],H^4(\mathbb{R}^3) \right)
	$.
	We have therefore proved \eqref{E:EXTENDEDMAINTHEOREMPSI0REGULARITY}.
	
	The desired localized blowup result \eqref{E:BLOWUPBEHAVIOR}
	for points in $\Sigma_{T_{(Lifespan)}}^{Blowup}$
	(where $\Sigma_{T_{(Lifespan)}}^{Blowup}$ is defined in \eqref{E:BLOWUPSET})
	now follows from
	\eqref{E:MAINTHMSMALLINTFACTIMPMLIESPSI0ISLARGE}
	and the continuous extension property $\Ifact \in C\left([0,T_{(Lifespan)}],C(\mathbb{R}^3) \right)$
	mentioned in the previous paragraph.
	
	 Finally, we will show that 
	if $(T_{(Lifespan)}, \underline{x}) \notin \Sigma_{T_{(Lifespan)}}^{Blowup}$,
	then there exists an open ball $B_{\underline{x}} \subset \mathbb{R}^3$ centered at $\underline{x}$
	such that for $\alpha = 0,1,2,3$, we have
	$\partial_{\alpha} \Phi \in C\left([0,T_{(Lifespan)}], H^5(B_{\underline{x}}) \right)$.
	To proceed, we first note that if $(T_{(Lifespan)}, \underline{x}) \notin \Sigma_{T_{(Lifespan)}}^{Blowup}$, 
	then the results proved above imply that there
	exist a $\delta > 0$ and a radius $r_{\underline{x}} > 0$
	such that,
	with $B_{\underline{x};r_{\underline{x}}} \subset \mathbb{R}^3$ denoting 
	the open ball of radius $r_{\underline{x}}$ centered at $\underline{x}$,
	we have $\Ifact(t,\underline{y}) > 0$ for 
	$(t,\underline{y}) \in [T_{(Lifespan)} - \delta, T_{(Lifespan)}] \times \bar{B}_{\underline{x};r_{\underline{x}}}$
	(where $\bar{B}_{\underline{x};r_{\underline{x}}}$ denotes the closure of $B_{\underline{x};r_{\underline{x}}}$)
	and such that for $t \in [T_{(Lifespan)} - \delta, T_{(Lifespan)})$,
	we have
	$\| \Psi_{\alpha} \|_{H^5(\lbrace t \rbrace \times B_{\underline{x};r_{\underline{x}}})} < \infty$
	and
	$\| \Ifact - 1 \|_{H^6(\lbrace t \rbrace \times B_{\underline{x};r_{\underline{x}}})} < \infty$.
	Hence, since the wave speed of the system is uniformly bounded on
	$[T_{(Lifespan)} - \delta, T_{(Lifespan)}] \times \bar{B}_{\underline{x};r_{\underline{x}}}$
	(see Remark~\ref{R:BOUNDEDWAVESPEED}),
	since 
	$\Ifact$ is uniformly bounded from above and from below \emph{strictly away from $0$}
	on $[T_{(Lifespan)} - \delta, T_{(Lifespan)}] \times \bar{B}_{\underline{x};r_{\underline{x}}}$,
	and since the estimates 
	\eqref{E:MAINTHEOREMPSI0REGULARITY}-\eqref{E:MAINTHEOREMIFACTREGULARITY}
	\eqref{E:EXTENDEDMAINTHEOREMPSI0REGULARITY}-\eqref{E:EXTENDEDMAINTHEOREMIFACTREGULARITY} hold,
	we can derive Sobolev estimates 
	(based on energy arguments)
	similar to the ones that we derived three paragraphs above,
	but localized in space,\footnote{For example, for $\upsigma > 0$ chosen sufficiently large,
	for $t$ near $T_{(Lifespan)}$,
	and for $s \in [t,T_{(Lifespan)}]$,
	one can consider the state of the solution on
	$\lbrace t \rbrace \times B_{\underline{x};r_{\underline{x}}}$
	as an initial condition and use energy identities to obtain 
	Sobolev estimates on
	$\lbrace s \rbrace \times B_{\underline{x};r_{\underline{x}} - \upsigma s} \subset \Sigma_s$.}
	for equations 
	\eqref{E:INTEGRATINFACTORODEANDIC}
	and
	\eqref{E:PARTALTPSI0EVOLUTION}-\eqref{E:PARTALTPSIIEVOLUTION},
	starting from initial conditions on
	$\lbrace t \rbrace \times B_{\underline{x};r_{\underline{x}}}$
	for some $t$ sufficiently close 
	(in a manner that depends on $\underline{x}$)
	to $T_{(Lifespan)}$.
	This yields the existence of an open ball 
	$B_{\underline{x}} \subset B_{\underline{x};r_{\underline{x}}}$
	centered at $\underline{x}$
	such that the following regularity properties hold:
	$\Psi_{\alpha} \in C\left([0,T_{(Lifespan)}], H^5(B_{\underline{x}}) \right)$
	and
	$\Ifact - 1 \in C\left([0,T_{(Lifespan)}], H^6(B_{\underline{x}}) \right)$.
	We clarify that to derive the localized energy estimates on the closed time interval $[0,T_{(Lifespan)}]$,
	it is crucially important that the bounds noted above imply
	that the spatial derivative weight $\weight(\Ifact^{-1} \Psi_0)$
	(which appears, for example, on the right-hand side of \eqref{E:ENERGYTOCONTROLSOLNS})
	is strictly positive on the domain
	$[T_{(Lifespan)} - \delta, T_{(Lifespan)}] \times \bar{B}_{\underline{x}}$.
	From the above regularity properties of $\lbrace \Psi_{\alpha} \rbrace_{\alpha = 0,1,2,3}$ and $\Ifact$,
	the positivity of $\Ifact$ on $[T_{(Lifespan)} - \delta, T_{(Lifespan)}] \times \bar{B}_{\underline{x}}$,
	and the standard Sobolev--Moser calculus,
	we conclude,
	in view of Def.~\ref{D:RENORMALIZEDSOLUTION},
	the desired result
	$\partial_{\alpha} \Phi \in C\left([0,T_{(Lifespan)}], H^5(B_{\underline{x}}) \right)$.
	We have therefore proved the theorem.
\end{proof}

\bibliographystyle{amsalpha}
\bibliography{JBib}

\end{document}